\documentclass{amsart}

\usepackage{graphicx}
\usepackage{amsmath}
\usepackage{amsfonts}
\usepackage{amssymb}
\usepackage{bbm}
\usepackage{epic}
\usepackage{amscd}

\usepackage{amsthm}

\usepackage[all]{xy}

\usepackage[margin=0.75in]{geometry}

\usepackage{appendix}

\usepackage{multirow}
\usepackage{tikz-cd}
\usepackage{url}

\newtheorem{theorem}{Theorem}[section]
\newtheorem{corollary}[theorem]{Corollary}
\newtheorem{lemma}[theorem]{Lemma}
\newtheorem{conjecture}[theorem]{Conjecture}

\theoremstyle{definition}
\newtheorem{definition}[theorem]{Definition}
\newtheorem{remark}[theorem]{Remark}
\newtheorem{example}[theorem]{Example}

\DeclareMathOperator{\im}{im}
\DeclareMathOperator{\coker}{coker}

\DeclareMathOperator{\Div}{div}

\DeclareMathOperator{\rank}{rank}

\DeclareMathOperator{\diag}{diag}

\DeclareMathOperator{\pdeg}{pdeg}
\DeclareMathOperator{\slope}{slope}
\DeclareMathOperator{\pslope}{pslope}

\DeclareMathOperator{\Bun}{Bun}

\newcommand{\Rats}{\mathbbm{Q}}
\newcommand{\Reals}{\mathbbm{R}}

\newcommand{\CP}{\mathbbm{CP}}
\newcommand{\RP}{\mathbbm{RP}}
\newcommand{\Complex}{\mathbbm{C}}

\newcommand{\calO}{{\mathcal O}}
\newcommand{\calH}{{\mathcal H}}
\newcommand{\calP}{{\mathcal P}}

\newcommand{\calE}{{\mathcal E}}
\newcommand{\calF}{{\mathcal F}}
\newcommand{\calY}{{\mathcal Y}}

\DeclareMathOperator{\Aut}{Aut}
\DeclareMathOperator{\Hom}{Hom}
\DeclareMathOperator{\Ext}{Ext}

\DeclareMathOperator{\Conf}{Conf}

\DeclareMathOperator{\Jac}{Jac}

\DeclareMathOperator{\gr}{gr}

\DeclareMathOperator{\MCG}{MCG}

\newcommand{\hmod}[2]{
  \xleftarrow
      [{\raisebox{.4ex}[0pt][0pt]{\ensuremath{\scriptstyle {#2}}}}]
      {\ensuremath{\scriptstyle \,\,{#1}\,\,}}}

\begin{document}

\title{Moduli spaces of Hecke modifications for rational and elliptic
  curves}

\author{David Boozer} 

\date{\today}

\begin{abstract}
  We propose definitions of complex manifolds $\calP_M(X,m,n)$ that
  could potentially be used to construct the symplectic Khovanov
  homology of $n$-stranded links in lens spaces.
  The manifolds $\calP_M(X,m,n)$ are defined as moduli spaces 
  of Hecke modifications of rank 2 parabolic bundles over an elliptic
  curve $X$.
  To characterize these spaces, we describe all possible Hecke
  modifications of all possible rank 2 vector bundles over $X$, and
  we use these results to define a canonical open embedding of
  $\calP_M(X,m,n)$ into $M^s(X,m+n)$, the moduli space of
  stable rank 2 parabolic bundles over $X$ with trivial determinant
  bundle and $m+n$ marked points.
  We explicitly compute $\calP_M(X,1,n)$ for $n=0,1,2$.
  For comparison, we present analogous results for the case of
  rational curves, for which a corresponding complex manifold
  $\calP_M(\CP^1,3,n)$ is isomorphic for $n$ even to a space
  $\calY(S^2,n)$ defined by Seidel and Smith that can be used to
  compute the symplectic Khovanov homology of $n$-stranded links in
  $S^3$.
\end{abstract}

\maketitle

\setcounter{tocdepth}{2} 
\tableofcontents

\section{Introduction}

Khovanov homology is a powerful invariant for distinguishing links in
$S^3$ \cite{Khovanov}.
Khovanov homology can be viewed as a categorification of the Jones
polynomial \cite{Jones}: one can recover the Jones polynomial of a
link from its Khovanov homology, but the Khovanov homology generally
contains more information.
For example, Khovanov homology can sometimes distinguish distinct
links that have the same Jones polynomial, and Khovanov homology
detects the unknot \cite{Kronheimer-2}, but it is not known
whether the Jones polynomial has this property.
The Khovanov homology of a link can be obtained in a purely
algebraic fashion by computing the homology of a chain complex
constructed from a generic planar projection of the link.
The Khovanov homology can also be obtained in a
geometric fashion by Heegaard-splitting $S^3$ into two 3-balls in
such a way that the intersection of the link with each 3-ball consists
of $r$ unknotted arcs.
Each 3-ball determines a Lagrangian in a symplectic manifold
$\calY(S^2,2r)$ known as the Seidel--Smith space, and the Lagrangian
Floer homology of the pair of Lagrangians yields the Khovanov homology
of the link (modulo grading) \cite{Abouzaid-Smith,Seidel-Smith}.
Recently Witten has outlined gauge theory interpretations of Khovanov
homology and the Jones polynomial in which the Seidel--Smith space is
viewed as a moduli space of solutions to the Bogomolny equations
\cite{Witten,Witten-2,Witten-3}.

Little is known about how Khovanov homology could be generalized to
describe links in 3-manifolds other than $S^3$, but such results would
be of great interest.
As a first step towards this goal, one might consider the problem of
generalizing Khovanov homology to links in 3-manifolds with Heegaard
genus 1; that is, lens spaces.
In analogy with the Seidel--Smith approach to Khovanov homology, one
could Heegaard-split a lens space into two solid tori, each
containing $r$ unknotted arcs, and compute the Lagrangian Floer
homology of a pair of Lagrangians intersecting in a symplectic
manifold $\calY(T^2,2r)$ that generalizes the Seidel--Smith space
$\calY(S^2,2r)$.

In this paper we propose candidates for the manifold $\calY(T^2,2r)$.
In outline, our approach is as follows.
First, using a result due to Kamnitzer \cite{Kamnitzer}, we
reinterpret the Seidel--Smith space $\calY(S^2,2r)$ as a moduli
space $\calH(\CP^1,2r)$ of equivalence classes of sequences of Hecke
modifications of rank 2 holomorphic vector bundles over a rational
curve.
Roughly speaking, a \emph{Hecke modification} is a way of locally
modifying a holomorphic vector bundle near a point to obtain a new
vector bundle.
We show that there is a close relationship between Hecke modifications
and parabolic bundles, and we use this relationship to reinterpret
the Kamnitzer space $\calH(\CP^1,2r)$ as a moduli space
$\calP_M(\CP^1,2r)$ of isomorphism classes of parabolic bundles with
marking data.
The space $\calP_M(\CP^1,2r)$ has two natural generalizations
$\calP_M(X,1,2r)$ and $\calP_M(X,3,2r)$ to the case of an elliptic
curve $X$, and we propose these spaces as candidates for
$\calY(T^2,2r)$.

To explain our approach in detail, we first introduce some additional
spaces.
Given a rank 2 holomorphic vector bundle $E$ over a curve $C$, we
define a set $\calH^{tot}(C,E,n)$ of equivalence classes of sequences
of $n$ Hecke modifications of $E$.
As is well-known, the set $\calH^{tot}(C,E,n)$ has the structure of
a complex manifold that is (noncanonically) isomorphic to
$(\CP^1)^{n}$, where each factor of $\CP^1$ corresponds to a single
Hecke modification.
The Kamnitzer space $\calH(\CP^1,2r)$ is then defined to be the open
submanifold of $\calH^{tot}(\CP^1,\calO \oplus \calO,2r)$ consisting
of equivalence classes of sequences of Hecke modifications for which
the terminal vector is semistable.

\begin{example}
For $r=1$, we have that
\begin{align}
  \nonumber
  \calH^{tot}(\CP^1, \calO \oplus \calO, 2) &= (\CP^1)^2, &
  \calH(\CP^1,2) &= (\CP^1)^2 - \{(a,a) \mid a \in \CP^1\}.
\end{align}
\end{example}

To generalize the Kamnitzer space to curves of higher genus,
we want to define moduli spaces of sequences of Hecke modifications in
which the initial vector bundle in the sequence is allowed to vary.
We define such spaces via the use of parabolic bundles.
For any curve $C$ and any natural numbers $m$ and $n$, we define a
moduli space $\calP_M^{tot}(C,m,n)$ of marked parabolic bundles.
We prove:

\begin{theorem}
\label{theorem:intro-PMtot}
The moduli space $\calP_M^{tot}(C,m,n)$ naturally has the structure of
a complex manifold isomorphic to a $(\CP^1)^n$-bundle over $M^s(C,m)$.
\end{theorem}

Here the complex manifold $M^s(C,m)$ is the moduli space of stable
rank 2 parabolic bundles over a curve $C$ with trivial determinant
bundle and $m$ marked points.
Roughly speaking, the space $\calP_M^{tot}(C,m,n)$ describes
isomorphism classes of sequences of $n$ Hecke modifications in which
the initial vector bundle in the sequence is allowed to range over
isomorphism classes that are parameterized by $M^s(C,m)$.
By imposing a condition analogous to the semistability condition
used to define the Kamnitzer space $\calH(\CP^1,2r)$,
we identify an open submanifold $\calP_M(C,m,n)$ of
$\calP_M^{tot}(C,m,n)$.
We prove that
$\calP_M^{tot}(\CP^1,n) := \calP_M^{tot}(\CP^1,3,n)$ is
isomorphic to $\calH^{tot}(\CP^1,\calO \oplus \calO,n)$ and
$\calP_M(\CP^1,2r) := \calP_M(\CP^1,3,2r)$ is isomorphic to to
$\calH(\CP^1,2r)$.
Thus the Seidel--Smith space $\calY(S^2,2r)$, the Kamnitzer space
$\calH(\CP^1,2r)$, and the space of marked parabolic bundles
$\calP_M(\CP^1,2r)$ are all isomorphic.
However, unlike $\calY(S^2,2r)$ or $\calH(\CP^1,2r)$, the space
$\calP_M(\CP^1,2r)$ naturally generalizes to the case of elliptic curves.

Although our primary motivation for introducing the spaces
$\calP_M^{tot}(\CP^1,n)$ and $\calP_M(\CP^1,n)$ is to facilitate
generalization, they are also useful for proving canonical versions of
certain results for $\CP^1$.
For example, we prove a canonical version of the noncanonical
isomorphism
$\calH^{tot}(\CP^1,\calO \oplus \calO,n) \rightarrow (\CP^1)^n$ for
$\calP_M^{tot}(\CP^1,n)$:

\begin{theorem}
\label{theorem:rational-isom-h-intro}
There is a canonical isomorphism
$\calP_M^{tot}(\CP^1,n) \rightarrow (M^{ss}(\CP^1,4))^n$.
\end{theorem}

Here the complex manifold $M^{ss}(\CP^1,4) \cong \CP^1$ is the moduli
space of semistable rank 2 parabolic bundles over $\CP^1$ with trivial
determinant bundle and 4 marked points.
We also prove:

\begin{theorem}
\label{theorem:cp1-hecke-embedding-intro}
There is a canonical open embedding
$\calP_M(\CP^1,m,n) \rightarrow M^s(\CP^1,m+n)$.
\end{theorem}

\begin{corollary}
There is a canonical open embedding
$\calP_M(\CP^1,2r) \rightarrow M^s(\CP^1,2r+3)$.
\end{corollary}

Here the complex manifold $M^s(\CP^1,m+n)$ is the moduli space of
stable rank 2 parabolic bundles over $\CP^1$ with trivial determinant
bundle and $m+n$ marked points.
For $r=1,2$ we have verified that the embedding of
$\calP_M(\CP^1,2r)$ into $M^s(\CP^1,2r+3)$ agrees with a
(noncanonical) embedding due to Woodward of the Seidel--Smith space
$\calY(S^2,2r)$ into $M^s(\CP^1,2r+3)$, and we conjecture that the
agreement holds for all $r$.

We next proceed to the case of elliptic curves.
We show that our reinterpretation $\calP_M(\CP^1,2r)$ of the
Seidel--Smith space $\calY(S^2,2r)$ has two natural generalizations to
the case of an elliptic curve $X$, namely $\calP_M(X,1,2r)$ and
$\calP_M(X,3,2r)$.
We prove an elliptic-curve analog to Theorem
\ref{theorem:rational-isom-h-intro}:

\begin{theorem}
\label{theorem:elliptic-isom-h-intro}
There is a canonical isomorphism
$\calP_M^{tot}(X,1,n) \rightarrow (M^{ss}(X))^{n+1}$.
\end{theorem}

Here the complex manifold $M^{ss}(X) \cong \CP^1$ is the moduli space
of semistable rank 2 vector bundles over an elliptic curve $X$ with
trivial determinant bundle.
Comparing Theorems \ref{theorem:rational-isom-h-intro} and
\ref{theorem:elliptic-isom-h-intro}, we see that
$\calP_M^{tot}(\CP^1,n)$ is (noncanonically) isomorphic to
$(\CP^1)^n$, whereas $\calP_M^{tot}(X,1,n)$ is (noncanonically)
isomorphic to $(\CP^1)^{n+1}$.
The extra factor of $\CP^1$ for $\calP_M^{tot}(X,1,n)$ can be
understood from Theorem \ref{theorem:intro-PMtot}, which states
that $\calP_M^{tot}(\CP^1,n) = \calP_M^{tot}(\CP^1,3,n)$ is a
$(\CP^1)^n$-bundle over $M^s(\CP^1,3)$ and
$\calP_M^{tot}(X,1,n)$ is a
$(\CP^1)^n$-bundle over $M^s(X,1)$.
But $M^s(\CP^1,3)$ is a single point, whereas $M^s(X,1)$ is isomorphic
to $\CP^1$.
We use Theorem \ref{theorem:elliptic-isom-h-intro} to explicitly
compute $\calP_M(X,1,n)$ for $n=0,1,2$:

\begin{theorem}
\label{theorem:elliptic-examples-intro}
The space $\calP_M(X,1,n)$ for $n=0,1,2$ is given by
\begin{align}
  \nonumber
  \calP_M(X,1,0) &= \CP^1, &
  \calP_M(X,1,1) &= (\CP^1)^2 - g(X), &
  \calP_M(X,1,2) &= (\CP^1)^3 - f(X),
\end{align}
where $g: X \rightarrow (\CP^1)^2$ and
$f:X \rightarrow (\CP^1)^3$ are
holomorphic embeddings defined in Sections
\ref{sec:pm-x-1-1} and \ref{sec:pm-x-1-2}.
\end{theorem}

We also generalize the embedding result of Theorem
\ref{theorem:cp1-hecke-embedding-intro} to the case of elliptic
curves:

\begin{theorem}
\label{theorem:elliptic-hecke-embedding-intro}
There is a canonical open embedding
$\calP_M(X,m,n) \rightarrow M^s(X,m+n)$.
\end{theorem}

Here the complex manifold $M^s(X,m+n)$ is the moduli space of stable
rank 2 parabolic bundles on $X$ with trivial determinant bundle and
$m+n$ marked points.
In order to prove Theorems
\ref{theorem:elliptic-isom-h-intro},
\ref{theorem:elliptic-examples-intro}, and
\ref{theorem:elliptic-hecke-embedding-intro},
we construct a list of all possible Hecke modifications of all
possible rank 2 vector bundles on an elliptic curve $X$.

We conclude by discussing possible applications of
our results to the problem of generalizing symplectic Khovanov homology
to lens spaces.
We observe that the embedding results of Theorems
\ref{theorem:cp1-hecke-embedding-intro} and
\ref{theorem:elliptic-hecke-embedding-intro}
could be related to a conjectural
spectral sequence from symplectic Khovanov homology to symplectic
instanton homology, which would generalize a spectral sequence due to
Kronheimer and Mrowka from Khovanov homology to singular instanton
homology \cite{Kronheimer-3}.
Based on such considerations, we make the following conjectures:

\begin{conjecture}
The space $\calP_M(C,3,2r)$ is the correct generalization of the
Seidel--Smith space $\calY(S^2,2r)$ to a curve $C$ of arbitrary genus.
\end{conjecture}

\begin{conjecture}
Given a curve $C$ of arbitrary genus, there is a canonical open
embedding $\calP_M(C,m,n) \rightarrow M^s(C,m+n)$.
\end{conjecture}

The paper is organized as follows.
In Section
\ref{sec:terminology}
we introduce some terminology specific to this paper.
In Section
\ref{sec:hecke-modifications}
we review the relevant background material on Hecke
modifications, discuss the relationship between Hecke modifications
and parabolic bundles, and define moduli spaces of marked parabolic
bundles $\calP_M^{tot}(C,m,n)$ and $\calP_M(C,m,n)$.
In Section
\ref{sec:rational-curves}
we describe Hecke modifications of rank 2 vector bundles on rational
curves and show that the Seidel--Smith space $\calY(S^2,2r)$ can be
reinterpreted as the space of marked parabolic bundles
$\calP_M(\CP^1,2r) = \calP_M(\CP^1,3,2r)$.
In Section
\ref{sec:elliptic-curves}
we describe Hecke modifications of rank 2 vector bundles on elliptic
curves and discuss possible elliptic-curve generalizations of the
Seidel--Smith space.
In Section
\ref{sec:applications}
we consider possible applications of our results to topology.
In Appendix
\ref{sec:vector-bundles} and
\ref{sec:parabolic-bundles}
we review the material we will need regarding vector bundles and
parabolic bundles on curves.

\section{Terminology}
\label{sec:terminology}

Here we introduce some terminology specific to this paper.

\begin{definition}
\label{def:ideg}
Given a rank 2 holomorphic vector bundle $E$ over a curve $C$, we
define the \emph{instability degree} of $E$ to be
$\deg L - \deg E/L$, where $L \subset E$ is a line subbundle of
maximal degree.
\end{definition}

The degree of the proper subbundles of a vector bundle $E$ on a curve
$C$ is bounded above (see for example \cite[Lemma 3.21]{Schaffhauser}),
so the notion of instability degree is well-defined.
The instability degree is positive for unstable bundles, 0 for
strictly semistable bundles, and negative for stable bundles.

\begin{definition}
\label{def:good-line}
Given a rank 2 holomorphic vector bundle $E$ over a curve $C$ and a
point $p \in C$, we say that a line
$\ell_{p} \in \mathbbm{P}(E_p)$ in the fiber $E_p$ over $p$
is {\em bad} if there is a line subbundle $L \subset E$ of maximal
degree such that $L_p = \ell_p$, and {\em good} otherwise.
\end{definition}

\begin{example}
For the trivial bundle $\calO \oplus \calO$ over $\CP^1$, all lines
are bad.
\end{example}

\begin{definition}
Given a rank 2 holomorphic vector bundle $E$ over a curve $C$ and
points $p_1, \cdots, p_n \in C$, we say that lines
$\ell_{p_i} \in \mathbbm{P}(E_{p_i})$ for $i=1, \cdots, n$ are
{\em bad in the same direction} if there is a line subbundle
$L \subset E$ of maximal degree such that
$L_{p_i} = \ell_{p_i}$ for $i=1,\cdots,n$.
\end{definition}

\section{Hecke modifications and parabolic bundles}
\label{sec:hecke-modifications}

\subsection{Hecke modifications at a single point}
\label{sec:cp1-single-hecke}

A fundamental concept for us is the notion of a Hecke modification of
a rank 2 holomorphic vector bundle.
This notion is described in \cite{Kamnitzer,Kapustin-Witten}.
Here we consider the case of a single Hecke modification.

\begin{definition}
\label{def:single-hecke}

Let $\pi_E:E \rightarrow C$ be a rank 2 holomorphic vector bundle over
a curve $C$.
A {\em Hecke modification} $E \hmod{\alpha}{p} F$ of $E$ at a point
$p \in C$ is a rank 2 holomorphic vector bundle $\pi_F:F\rightarrow C$
and a bundle map $\alpha:F \rightarrow E$ that satisfies the
following two conditions:
\begin{enumerate}
\item
  The induced maps on fibers $\alpha_q:F_q \rightarrow E_q$ are
  isomorphisms for all points $q \in C$ such that $q \neq p$.

\item
  We also impose a condition on the behavior of $\alpha$ near $p$.
  We require that there is an open neighborhood $U \subset C$ of $p$,
  local coordinates $\xi:U \rightarrow V$ for $V \subset \Complex$
  such that $\xi(p) = 0$,
  and local trivializations
  $\psi_E:\pi_E^{-1}(U) \rightarrow U \times \Complex^2$ and
  $\psi_F:\pi_F^{-1}(U) \rightarrow U \times \Complex^2$ of
  $E$ and $F$ over $U$ such that the following diagram commutes:
  \begin{eqnarray}
    \nonumber
    \begin{tikzcd}
      \pi_E^{-1}(U) \arrow{d}{\psi_E}[swap]{\cong} &
      \pi_F^{-1}(U) \arrow{d}{\psi_F}[swap]{\cong}
      \arrow{l}[swap]{\alpha} \\
      U \times \Complex^2 &
      U \times \Complex^2, \arrow{l}
    \end{tikzcd}
  \end{eqnarray}
  where the bottom horizontal arrow is
  \begin{align}
    \nonumber
    (\psi_E \circ \alpha \circ \psi_F^{-1})(q, v) =
    (q,\bar{\alpha}(\xi(q))v)
  \end{align}
  and $\bar{\alpha}:V \rightarrow M(2,\Complex)$ has the form
  \begin{align}
    \nonumber
    \bar{\alpha}(z) = \left(\begin{array}{cc}
      1 & 0 \\
      0 & z
    \end{array}
    \right).
  \end{align}
\end{enumerate}
\end{definition}

It follows directly from Definition \ref{def:single-hecke} that
$\det F = (\det E) \otimes \calO(-p)$ and
$\deg F = \deg E - 1$.
There is a natural notion of equivalence of Hecke modifications:

\begin{definition}
\label{def:hecke-equiv}
We say that two Hecke modifications
$E \hmod{\alpha}{p} F$ and
$E \hmod{\alpha'}{p} F'$
of $E$ at a point $p \in C$ are {\em equivalent} if there is an
isomorphism $\phi:F \rightarrow F'$ such that
$\alpha = \alpha' \circ \phi$.
\end{definition}

\begin{definition}
\label{def:Htot}
We define the \emph{total space of Hecke modifications}
$\calH^{tot}(C,E;p)$ to be the set of equivalence classes of Hecke
modifications of a rank 2 vector bundle $\pi_E:E\rightarrow C$ at a
point $p \in C$.
\end{definition}

As is well-known, the set $\calH^{tot}(C,E;p)$ naturally has the
structure of a complex manifold that is (noncanonically) isomorphic to
$\CP^1$.
A canonical version of this statement is:

\begin{theorem}
\label{theorem:single-isomorphism}
There is a canonical isomorphism
$\calH^{tot}(C,E;p) \rightarrow \mathbbm{P}(E_p)$,
$[E \hmod{\alpha}{p} F] \mapsto \im(\alpha_p:F_p \rightarrow E_p)$.
\end{theorem}

It is also useful to think about Hecke modifications in terms of
sheaves of sections.
Consider a rank 2 vector bundle $E$ and a line
$\ell_p \in \mathbbm{P}(E_p)$.
Let $\calE$ be the sheaf of sections of $E$, and define a subsheaf
$\calF$ of $\calE$ whose set of sections over an open set
$U \subset C$ is given by
\begin{align}
  \nonumber
  \calF(U) = \{s \in \calE(U) \mid
  p \in U \implies s(p) \in \ell_p\}.
\end{align}
Define $F$ to be the vector bundle whose sheaf of sections is $\calF$,
and define $\alpha:F \rightarrow E$ to be the bundle map corresponding
to the inclusion of sheaves $\calF \rightarrow \calE$.
Then $[E \hmod{\alpha}{p} F] \in \calH^{tot}(C,E;p)$ corresponds to
$\ell_p \in \mathbbm{P}(E_p)$ under the isomorphism described in
Theorem \ref{theorem:single-isomorphism}.
We have an exact sequence of sheaves
\begin{eqnarray}
  \nonumber
  \begin{tikzcd}
    0 \arrow{r} &
    {\mathcal F} \arrow{r} &
    {\mathcal E} \arrow{r} &
    \Complex_p \arrow{r} &
    0,
  \end{tikzcd}
\end{eqnarray}
where $\Complex_p$ is a skyscraper sheaf supported at the point $p$.
It is important to note, however, that the usual notion of equivalence
of extensions differs from the notion of equivalence of Hecke
modifications given in Definition \ref{def:hecke-equiv}.

\subsection{Sequences of Hecke modifications}
\label{sec:sequences-hecke}

We would now like to generalize the notion of a Hecke modification
of a vector bundle at a single point $p \in C$ to the notion of a
sequence of Hecke modifications at distinct points
$(p_1, \cdots, p_n) \in C^n$.

\begin{definition}
Let $\pi_E:E \rightarrow C$ be a rank 2 holomorphic vector bundle over
a curve $C$.
A {\em sequence of Hecke modifications}
$E
\hmod{\alpha_1}{p_1} E_1
\hmod{\alpha_2}{p_2} \cdots
\hmod{\alpha_n}{p_n} E_n$
of $E$ at distinct points $(p_1, p_2, \cdots, p_n) \in C^n$
is a collection of rank 2 holomorphic vector bundles
$\pi_{E_i}:E_i \rightarrow C$ and Hecke modifications
$E_{i-1} \hmod{\alpha_i}{p_i} E_i$ for $i=1,2,\cdots,n$,
where $E_0 := E$.
\end{definition}

\begin{definition}
Two sequences of Hecke modifications
$E
\hmod{\alpha_1}{p_1} E_1
\hmod{\alpha_2}{p_2} \cdots
\hmod{\alpha_n}{p_n} E_n$
and
$E
\hmod{\alpha_1'}{p_1} E_1'
\hmod{\alpha_2'}{p_2} \cdots
\hmod{\alpha_n'}{p_n} E_n'$
are {\em equivalent} if
there are isomorphisms $\phi_i:E_i \rightarrow E_i'$ such that the
following diagram commutes:
\begin{eqnarray}
  \nonumber
  \begin{tikzcd}
    E \arrow{d}[swap]{=} &
    E_1 \arrow{l}[swap]{\alpha_1}\arrow{d}{\phi_1}[swap]{\cong} &
    E_2 \arrow{l}[swap]{\alpha_2}\arrow{d}{\phi_2}[swap]{\cong} &
    \cdots \arrow{l}[swap]{\alpha_3} &
    E_n \arrow{l}[swap]{\alpha_n}\arrow{d}{\phi_n}[swap]{\cong} \\
    E &
    E_1' \arrow{l}[swap]{\alpha_1'} &
    E_2' \arrow{l}[swap]{\alpha_2'} &
    \cdots \arrow{l}[swap]{\alpha_3'} &
    E_n'. \arrow{l}[swap]{\alpha_n'}
  \end{tikzcd}
\end{eqnarray}
\end{definition}

\begin{definition}
\label{def:space-h-tot}
We define the \emph{total space of Hecke modifications
$\calH^{tot}(C,E;p_1, \cdots,p_n)$} to be the set of equivalence
classes of sequences of Hecke modifications of the rank 2
vector bundle $\pi_E:E \rightarrow C$ at points
$(p_1, \cdots, p_n) \in C^n$.
For simplicity, we will often suppress the dependence on
$p_1, \cdots, p_n$ and denote this space as
$\calH^{tot}(C,E,n)$.
\end{definition}

\begin{definition}
\label{def:iso-hecke}
We say that an isomorphism of vector bundles
$\phi:E \rightarrow E'$ is an \emph{isomorphism of equivalence classes
  of sequences of Hecke modifications}
$\phi:[E
  \hmod{\alpha_1}{p_1} E_1
  \hmod{\alpha_2}{p_2} \cdots
  \hmod{\alpha_n}{p_n} E_n] \rightarrow
[E'
  \hmod{\alpha_1'}{p_1} E_1'
  \hmod{\alpha_2'}{p_2} \cdots
  \hmod{\alpha_n'}{p_n} E_n']$
if
\begin{align}
  \nonumber
  [E'
  \hmod{\beta_1}{p_1} E_1
  \hmod{\alpha_2}{p_2} \cdots
  \hmod{\alpha_n}{p_n} E_n] =
  [E'
  \hmod{\alpha_1'}{p_1} E_1'
  \hmod{\alpha_2'}{p_2} \cdots
  \hmod{\alpha_n'}{p_n} E_n'],
\end{align}
where $\beta_1 := \phi \circ \alpha_1$, or equivalently, if there are
isomorphisms $\phi_i:E_i \rightarrow E_i'$ such that the following
diagram commutes:
\begin{eqnarray}
  \nonumber
  \begin{tikzcd}
    E \arrow{d}{\phi}[swap]{\cong} &
    E_1 \arrow{l}[swap]{\alpha_1}\arrow{d}{\phi_1}[swap]{\cong} &
    E_2 \arrow{l}[swap]{\alpha_2}\arrow{d}{\phi_2}[swap]{\cong} &
    \cdots \arrow{l}[swap]{\alpha_3} &
    E_n \arrow{l}[swap]{\alpha_n}\arrow{d}{\phi_n}[swap]{\cong} \\
    E' &
    E_1' \arrow{l}[swap]{\alpha_1'} &
    E_2' \arrow{l}[swap]{\alpha_2'} &
    \cdots \arrow{l}[swap]{\alpha_3'} &
    E_n'. \arrow{l}[swap]{\alpha_n'}
  \end{tikzcd}
\end{eqnarray}
\end{definition}

In what follows, it will be useful to reinterpret equivalence
classes of sequences of Hecke modifications in terms of parabolic
bundles.
The relevant background material on parabolic bundles is discussed in
Appendix \ref{sec:parabolic-bundles}.
For our purposes here, a rank 2 parabolic bundle over a curve $C$
consists of a rank 2 holomorphic vector bundle
$\pi_E:E \rightarrow C$, a parameter $\mu > 0$, called the
\emph{weight}, and a choice of line
$\ell_{p_i} \in \mathbbm{P}(E_{p_i})$ in the fiber
$E_{p_i} = \pi_E^{-1}(p_i)$ over the point $p_i \in C$ for a
finite number of distinct points $(p_1, \cdots, p_n) \in C^n$.
The data of just the marked points and lines, without the weight, is
referred to as a \emph{quasi-parabolic structure} on $E$.
The additional data of the weight allows us to define the notions of
stable, semistable, and unstable parabolic bundles.

\begin{definition}
We define
$\calP^{tot}(C,E;p_1,\cdots,p_n) =
\mathbbm{P}(E_{p_1}) \times \cdots \times \mathbbm{P}(E_{p_n})$
to be the set of all quasi-parabolic structures with marked points
$(p_1, \cdots, p_n) \in C^n$ on a rank 2 holomorphic vector bundle
$\pi_E:E \rightarrow C$.
For simplicity, we will often suppress the dependence on
$p_1, \cdots, p_n$ and denote this space as $\calP^{tot}(C,E,n)$.
\end{definition}

Since $\mathbbm{P}(E_{p_i})$ is (noncanonically) isomorphic to
$\CP^1$, the set $\calP^{tot}(C,E;p_1,\cdots,p_n)$ naturally has the
structure of a complex manifold that is (noncanonically) isomorphic to
$(\CP^1)^n$.
We have the following generalization of Theorem
\ref{theorem:single-isomorphism}, which allows us to reinterpret Hecke
modifications in terms of parabolic bundles:

\begin{theorem}
\label{theorem:canonical-isomorphism}
There is a canonical isomorphism
$\calH^{tot}(C,E;p_1,\cdots,p_n) \rightarrow
\calP^{tot}(C,E;p_1,\cdots,p_n)$ given by
\begin{align}
  \nonumber
  [E
  \hmod{\alpha_1}{p_1} E_1
  \hmod{\alpha_2}{p_2} \cdots
  \hmod{\alpha_n}{p_n} E_n] \mapsto
  (E,\ell_{p_1},\cdots,\ell_{p_n}),
\end{align}
where
$\ell_{p_i} :=
\im ((\alpha_1 \circ \cdots \circ \alpha_i)_{p_i}:
(E_i)_{p_i} \rightarrow E_{p_i})$.
\end{theorem}

Under our reinterpretation of Hecke modifications in terms of
parabolic bundles, an isomorphism of equivalence classes of sequences
of Hecke modifications corresponds to an isomorphism of parabolic
bundles:

\begin{definition}
\label{def:iso-parabolic}
We say that an isomorphism of vector bundles $\phi:E \rightarrow E'$
is an \emph{isomorphism of parabolic bundles}
$\phi:(E,\ell_{p_1},\cdots,\ell_{p_n}) \rightarrow
(E',\ell_{p_1}',\cdots,\ell_{p_n}')$ if
$\phi(\ell_{p_i}) = \ell_{p_i}'$ for $i=1,\cdots,n$.
\end{definition}

\begin{theorem}
\label{theorem:parabolic-iso-iff-sequence-iso}
Two equivalence classes of sequences of Hecke
modifications are isomorphic if and only if their corresponding
parabolic bundles are isomorphic.
\end{theorem}

\begin{proof}
This is a direct consequence of Theorem
\ref{theorem:canonical-isomorphism} and Definitions
\ref{def:iso-hecke} and \ref{def:iso-parabolic}.
\end{proof}

For many applications, given an equivalence class
$[E
  \hmod{\alpha_1}{p_1} E_1
  \hmod{\alpha_2}{p_2} \cdots
  \hmod{\alpha_n}{p_n} E_n]$ of
sequences of Hecke modifications we will be interested only in the
isomorphism class of the terminal vector bundle $E_n$, and it is useful
to have a means of extracting this information from the corresponding
parabolic bundle $(E,\ell_{p_1}, \cdots, \ell_{p_n})$:

\begin{definition}
\label{def:hecke-transform}
Let $(E,\ell_{p_1},\cdots,\ell_{p_n})$ be a parabolic bundle over a
curve $C$.
We define the \emph{Hecke transform}
$H(E,\ell_{p_1},\cdots,\ell_{p_n})$ of $E$ to be the vector bundle $F$
that is constructed as follows.
Let $\calE$ be the sheaf of sections of $E$.
Define a subsheaf $\calF$ of $\calE$ whose set of sections over an
open set $U \subset C$ is given by
\begin{align}
  \nonumber
  \calF(U) = \{s \in \calE(U) \mid
  \textup{$p_i \in U \implies s(p_i) \in \ell_{p_i}$ for
    $i = 1,\cdots, n$}\}.
\end{align}
Now define $F$ to be the vector bundle whose sheaf of sections is
$\calF$.
\end{definition}

In particular, $H(E,\ell_{p_1},\cdots,\ell_{p_n})$ is isomorphic to
$E_n$.
We will often want to pick out an open subset of $\calP^{tot}(C,E,n)$
by using the Hecke transform to impose a semistability condition:

\begin{definition}
Given a rank 2 holomorphic vector bundle $E$ over a curve $C$ and
distinct points $(p_1,\cdots,p_n) \in C^n$, we define
$\calP(C,E,n)$ to be the subset of $\calP^{tot}(C,E,n)$
consisting of parabolic bundles $(E,\ell_{p_1},\cdots,\ell_{p_n})$
such that $H(E,\ell_{p_1},\cdots,\ell_{p_n})$ is semistable:
\begin{align}
  \nonumber
  \calP(C,E,n) =
  \{(E,\ell_{p_1}, \cdots, \ell_{p_n}) \in \calP^{tot}(C,E,n) \mid
  \textup{$H(E,\ell_{p_1},\cdots, \ell_{p_n})$ is semistable}\}.
\end{align}
For simplicity, we are suppressing the dependence of
$\calP(C,m,n)$ on $p_1, \cdots, p_n$ in the notation.
\end{definition}

\begin{theorem}
The set $\calP(C,E,n)$ is an open submanifold of $\calP^{tot}(C,E,n)$.
\end{theorem}

\begin{proof}
This follows from the fact that semistability is an open condition.
\end{proof}

We have generalized the notion of a Hecke modification to the
case of multiple points $p_1, \cdots, p_n$ by considering sequences of
Hecke modifications, for which the points must be ordered.
For most of our purposes we could equally well use an
alternative generalization, described in \cite{Kamnitzer}, for which
the the points need not be ordered.
Though we will not use it here, we briefly describe this alternative
generalization and show how it relates to parabolic bundles:

\begin{definition}
Let $\pi_E:E \rightarrow C$ be a rank 2 holomorphic vector bundle over
a curve $C$.
A \emph{simultaneous Hecke modification}
$E \hmod{\alpha}{\{p_1,\cdots,p_n\}} F$
of $E$ at a set of distinct points $\{p_1, p_2, \cdots, p_n\} \subset C$
is a rank 2 holomorphic vector bundle
$\pi_F:F \rightarrow C$ and a bundle map
$\alpha:F \rightarrow E$ that satisfies the following two conditions:
\begin{enumerate}
\item
The induced map on fibers $\alpha_q:E_q \rightarrow F_q$ is an
isomorphism for all points $q \notin \{p_1, \cdots, p_n\}$.

\item
Condition (2) of Definition \ref{def:single-hecke}, which constrains
the local behavior of $\alpha$ near a Hecke-modification point, holds
at each of the points $p_1, \cdots, p_n$.
\end{enumerate}
\end{definition}

\begin{definition}
Two simultaneous Hecke modifications
$E \hmod{\alpha}{\{p_1,\cdots,p_n\}} F$ and
$E \hmod{\alpha'}{\{p_1,\cdots,p_n\}} F'$ are \emph{equivalent} if
there is an isomorphism $\phi:F \rightarrow F'$ such that
$\alpha = \alpha' \circ \phi$.
\end{definition}

\begin{definition}
We define the \emph{total space of simultaneous Hecke modifications}
$\overline{\calH}^{tot}(C,E,\{p_1, \cdots, p_n\})$ to be the set of
equivalence classes of simultaneous Hecke modifications of the rank 2
vector bundle $\pi_E:E \rightarrow C$ at the set of points
$\{p_1, \cdots, p_n\} \subset C$.
For simplicity, we will often suppress the dependence on
$\{p_1, \cdots, p_n\}$ and denote this space as
$\overline{\calH}^{tot}(C,E,n)$.
\end{definition}

We can define a set of parabolic bundles
$\overline{\calP}^{tot}(C,E,n)$ for which the marked points are unordered.
We can define an isomorphism
$\overline{\calH}^{tot}(C,E,n) \rightarrow
\overline{\calP}^{tot}(C,E,n)$ by
\begin{align}
  \nonumber
  [E \hmod{\alpha}{\{p_1,\cdots,p_n\}} F] \mapsto
  (E,\ell_{p_1},\cdots,\ell_{p_n}),
\end{align}
where $\ell_{p_i} = \im(\alpha_{p_i}:F_{p_i} \rightarrow E_{p_i})$.
Since the Hecke transform does not depend on the ordering of the points,
it is well-defined on parabolic bundles in
$\overline{\calP}^{tot}(C,E,n)$, and we have that
$H(E,\ell_{p_1},\cdots,\ell_{p_n})$ is isomorphic to $F$.

\subsection{Moduli spaces of marked parabolic bundles}
\label{sec:marked-parabolic}

So far we have considered spaces of isomorphism classes of sequences
of Hecke modifications in which the initial vector bundle in the
sequence is held fixed.
But in what follows we will want to generalize these spaces so
the initial vector bundle is allowed to range over the
isomorphism classes in a moduli space of vector bundles.
Translating into the language of parabolic bundles, such spaces are
equivalent to spaces of isomorphism classes of parabolic bundles in
which the underlying vector bundles are allowed to range over the
isomorphism classes in a moduli space of vector bundles.
However, there is a problem with defining such spaces arising from
the fact that vector bundles may have nontrivial automorphisms.

To illustrate the problem, consider the space
$\calP^{tot}(\CP^1,\calO \oplus \calO;p_1)$, which is (noncanonically)
isomorphic to $\CP^1$.
We might want to reinterpret this space as a moduli space of
isomorphism classes of parabolic bundles of the form
$(E,\ell_{p_1})$ for $[E] \in M^{ss}(\CP^1)$, where $M^{ss}(\CP^1)$,
the moduli space of semistable rank 2 vector bundles over $\CP^1$ with
trivial determinant bundle, consists of the single point
$[\calO \oplus \calO]$.
But $\Aut(\calO \oplus \calO) = GL(2,\Complex)$, and for any pair of
parabolic bundles
$(\calO \oplus \calO, \ell_{p_1})$ and
$(\calO \oplus \calO, \ell_{p_1}')$
there is an automorphism $\phi \in \Aut(\calO \oplus \calO)$ such that
$\phi(\ell_{p_1}) = \ell_{p_1}'$.
It follows that all parabolic bundles of the form
$(\calO \oplus \calO, \ell_{p_1})$ are isomorphic and our proposed
moduli space collapses to a point, when what we wanted was $\CP^1$.

To remedy the problem, we will add marking data to eliminate
the nontrivial automorphisms.
In particular, since stable parabolic bundles have no nontrivial
automorphisms, we make the following definition:

\begin{definition}
\label{def:PMtot}
Given a curve $C$ and distinct points
$(q_1, \cdots, q_m, p_1, \cdots, p_n) \in C^{m+n}$, we define the
\emph{total space of marked parabolic bundles $\calP_M^{tot}(C,m,n)$}
to be the set of isomorphism classes of parabolic bundles of the form
$(E,\ell_{q_1}, \cdots, \ell_{q_m}, \ell_{p_1}, \cdots, \ell_{p_n})$ such
that $[E,\ell_{q_1},\cdots, \ell_{q_m}] \in M^s(C,m)$:
\begin{align}
  \nonumber
  \calP_M^{tot}(C,m,n) =
  \{[E,\ell_{q_1}, \cdots, \ell_{q_m},
    \ell_{p_1}, \cdots, \ell_{p_n}] \mid
  [E,\ell_{q_1},\cdots, \ell_{q_m}] \in M^s(C,m)\}.
\end{align}
For simplicity, we are suppressing the dependence of
$\calP_M^{tot}(C,m,n)$ on
$q_1, \cdots, q_m,p_1, \cdots, p_n$ in the notation.
\end{definition}

Here the complex manifold $M^s(C,m)$ is the moduli space of stable
rank 2 parabolic bundles over $C$ with trivial determinant bundle and
$m$ marked points.
We will refer to the lines $\ell_{q_1}, \cdots, \ell_{q_m}$ as
\emph{marking lines}, since their purpose is to add additional
structure to $E$ so as to eliminate nontrivial automorphisms.
We will refer to the lines $\ell_{p_1},\cdots,\ell_{p_n}$ as
\emph{Hecke lines}, since their purpose is to parameterize
Hecke modifications at the points $p_1, \cdots, p_n$.
Because we have defined $\calP_M^{tot}(C,m,n)$ in terms of stable
parabolic bundles, which have no nontrivial automorphisms, the
collapsing phenomenon described above does not occur, and we have the
following result:

\begin{theorem}
\label{theorem:PMtot}
The set $\calP_M^{tot}(C,m,n)$ naturally has the structure of a
complex manifold isomorphic to a $(\CP^1)^n$-bundle over $M^s(C,m)$.
\end{theorem}

The base manifold $M^s(C,m)$ constitutes the moduli space over which
the isomorphism classes of vector bundles with marking data range, and
the $(\CP^1)^n$ fibers correspond to a space of Hecke modifications
$\CP^1$ for each of the points $p_1, \cdots, p_n$.
We will prove Theorem \ref{theorem:PMtot} by constructing
$\calP_M^{tot}(C,m,n)$ from a universal $\CP^1$-bundle, which we first
describe for the case $m=0$:

\begin{lemma}
\label{lemma:cp1-ubundle}
There is a universal $\CP^1$-bundle
$P \rightarrow C \times M^s(C)$, which has the the property that
for any complex manifold $S$ and any $\CP^1$-bundle
$Q \rightarrow C \times S$, the bundle
$Q$ is isomorphic to the pullback of $P$ along $1_C \times f_Q$ for
a unique map $f_Q:S \rightarrow M^s(C)$.
\end{lemma}

Lemma \ref{lemma:cp1-ubundle} is proven in \cite{Biswas}.
One way to understand this result is as follows.
Let $M_{r,d}^s(C)$ denote the moduli space of stable vector bundles of
rank $r$ and degree $d$ on a curve $C$.
As discussed in \cite{Hoffmann}, one can define a corresponding moduli
stack $\Bun_{r,d}^s(C)$ and a
$\mathbbm{G}_m$-gerbe
$\pi:\Bun_{r,d}^s(C) \rightarrow \Hom(-,M_{r,d}^s(C))$; this is a
morphism of stacks for which all the fibers are isomorphic to
$B\mathbbm{G}_m$.
The stack $\Hom(-,C) \times \Bun_{r,d}^s(C)$ carries a
universal rank 2 vector bundle $\calE$.
One can show 
(see \cite[Corollary 3.12]{Heinloth}) that if $\gcd(r,d) = 1$ then
$M_{r,d}^s(C)$ is a fine moduli space and $\calE$ descends to a
universal vector bundle $E \rightarrow C \times M_{r,d}^s(C)$, which
can be viewed as a generalization of the Poincar\'{e} line bundle
$L \rightarrow C \times \Jac(C)$ for the case $r=1$, $d=0$.
By projectivizing $E$, we also get a universal $\CP^1$-bundle
$\mathbbm{P}(E) \rightarrow C \times M_{r,d}^s(C)$.
If $\gcd(r,d) \neq 1$, then $M_{r,d}^s(C)$ is not a fine moduli space
and $C \times M_{r,d}^s(C)$ does not carry a universal vector bundle.
It is still possible, however, to use $\calE$ to construct a
universal $\CP^{r-1}$-bundle $P \rightarrow C \times M_{r,d}^s(C)$,
only now this $\CP^{r-1}$-bundle is not the projectivization of a
universal vector bundle.
One way to make this result plausible is to note that whereas a stable
vector bundle has automorphism group $\Complex^\times$, consisting of
automorphisms that scale the fibers by a constant factor, the
projectivization of a stable vector bundle has trivial automorphism
group, consisting of just the identity automorphism.

Similar results hold for moduli spaces of stable vector bundles for
which the determinant bundle is a fixed line bundle.
In particular, the space $M^s(C)$ of stable rank 2 vector bundles with
trivial determinant bundle is not a fine moduli space and
$C \times M^s(C)$ does not carry a universal vector bundle;
nonetheless, it does carry a universal $\CP^1$-bundle
$P \rightarrow C \times M^s(C)$.
We will use this universal $\CP^1$-bundle to construct
$\calP_M^{tot}(C,m,n)$ for the case $m=0$:

\begin{proof}[Proof of Theorem \ref{theorem:PMtot}]
First consider the case $m=0$, and note that $M^s(C,0) = M^s(C)$.
Given a point $p \in C$, let $P(p) \rightarrow M^s(C)$ denote the
pullback of the universal $\CP^1$-bundle
$P \rightarrow C \times M^s(C)$ described in Lemma
\ref{lemma:cp1-ubundle} along the inclusion
$M^s(C) \rightarrow C \times M^s(C)$, $[E] \mapsto (p,[E])$.
Given distinct points $(p_1, \cdots, p_n) \in C^n$, we can pull back the
$(\CP^1)^n$-bundle
$P(p_1) \times \cdots \times P(p_n) \rightarrow (M^s(C))^n$
along the diagonal map
$M^s(C) \rightarrow (M^s(C))^n$, $[E] \mapsto ([E], \cdots, [E])$
to obtain $\calP_M^{tot}(C,0,n)$.

The proof for $m > 0$ is the same.
One can define a moduli stack corresponding to $M^s(C,m)$ that
carries a universal rank 2 parabolic bundle \cite{Hoffmann}.
Using a numerical condition analogous to the condition
$\gcd(r,d) = 1$ for vector bundles (see \cite[Example 5.7]{Hoffmann}
and \cite[Proposition 3.2]{Boden}), one can show that $M^s(C,m)$ is a
fine moduli space for $m > 0$ and the universal parabolic bundle on
the moduli stack descends to a universal parabolic bundle on
$C \times M^s(C,m)$.
We can projectivize this latter bundle and use it to construct
$\calP_M^{tot}(C,m,n)$ in the same manner as for the $m=0$ case.
\end{proof}

We will often want to pick out an open subset of
$\calP_M^{tot}(C,m,n)$ by imposing a semistability condition:

\begin{definition}
\label{def:PM}
Given a curve $C$ and distinct points
$(q_1,\cdots,q_m,p_1,\cdots,p_n) \in C^{m+n}$, we define
$\calP_M(C,m,n)$ to be the subset of $\calP_M^{tot}(C,m,n)$
consisting of points
$[E,\ell_{q_1},\cdots,\ell_{q_m},\ell_{p_1},\cdots,\ell_{p_n}]$
such that
$H(E,\ell_{p_1},\cdots,\ell_{p_n})$ is semistable:
\begin{align}
  \nonumber
  \calP_M(C,m,n) =
  \{[E,\ell_{q_1}, \cdots, \ell_{q_m},
    \ell_{p_1}, \cdots, \ell_{p_n}] \in \calP_M^{tot}(C,m,n) \mid
  \textup{$H(E,\ell_{p_1},\cdots, \ell_{p_n})$ is semistable}\}.
\end{align}
For simplicity, we are suppressing the dependence of
$\calP_M(C,m,n)$ on
$q_1, \cdots, q_m,p_1, \cdots, p_n$ in the notation.
\end{definition}

\begin{theorem}
\label{theorem:PM}
The set $\calP_M(C,m,n)$ is an open submanifold of
$\calP_M^{tot}(C,m,n)$.
\end{theorem}

\begin{proof}
This follows from the fact that semistability is an open condition.
\end{proof}

We can interpret marked parabolic bundles in terms of Hecke
modifications as follows:

\begin{definition}
We define a
\emph{sequence of Hecke modifications of a parabolic bundle}
$(E,\ell_{q_1},\cdots,\ell_{q_m})$
to be a sequence of Hecke modifications of the underlying vector
bundle $E$.
\end{definition}

\begin{definition}
We say that two sequences of Hecke modifications
\begin{align}
  \nonumber
  (E,\ell_{q_1},\cdots,\ell_{q_m})
  \hmod{\alpha_1}{p_1} E_1
  \hmod{\alpha_2}{p_2} \cdots
  \hmod{\alpha_n}{p_n} E_n, &&
  (E',\ell_{q_1}',\cdots,\ell_{q_m}')
  \hmod{\alpha_1'}{p_1} E_1'
  \hmod{\alpha_2'}{p_2} \cdots
  \hmod{\alpha_n'}{p_n} E_n'
\end{align}
are {\em equivalent} if
there are isomorphisms $\phi_i:E_i \rightarrow E_i'$ for
$i=0,\cdots,n$ such that the following diagram commutes:
\begin{eqnarray}
  \nonumber
  \begin{tikzcd}
    (E,\ell_{q_1},\cdots,\ell_{q_m}) \arrow{d}{\phi_0}[swap]{\cong} &
    E_1 \arrow{l}[swap]{\alpha_1}\arrow{d}{\phi_1}[swap]{\cong} &
    E_2 \arrow{l}[swap]{\alpha_2}\arrow{d}{\phi_2}[swap]{\cong} &
    \cdots \arrow{l}[swap]{\alpha_3} &
    E_n \arrow{l}[swap]{\alpha_n}\arrow{d}{\phi_n}[swap]{\cong} \\
    (E',\ell_{q_1}',\cdots,\ell_{q_m}') &
    E_1' \arrow{l}[swap]{\alpha_1'} &
    E_2' \arrow{l}[swap]{\alpha_2'} &
    \cdots \arrow{l}[swap]{\alpha_3'} &
    E_n'. \arrow{l}[swap]{\alpha_n'}
  \end{tikzcd}
\end{eqnarray}
\end{definition}

The space $\calP_M^{tot}(C,m,n)$ can then be interpreted as a moduli
space of equivalence classes of sequences of Hecke modifications of
parabolic bundles, and the space $\calP_M(C,m,n)$ can be
interpreted as the subspace $\calP_M^{tot}(C,m,n)$ consisting of
equivalence classes of sequences for which the terminal
vector bundles are semistable.
We will not use these interpretations here, since it is simpler to
work directly with the marked parabolic bundles.

\section{Rational curves}
\label{sec:rational-curves}

\subsection{Vector bundles on rational curves}
\label{sec:cp1-vector-bundles}

Grothendieck showed that all rank 2 holomorphic vector bundles on
(smooth projective) rational curves are decomposable
\cite{Grothendieck}; that is, they have the form
$\calO(n) \oplus \calO(m)$ for integers $n$ and $m$.
The instability degree of $\calO(n) \oplus \calO(m)$ is $|n - m|$, so
the bundle $\calO(n) \oplus \calO(m)$ is strictly semistable if
$n = m$ and unstable otherwise.
There are no stable rank 2 vector bundles on rational curves.

\subsection{List of all possible single Hecke modifications}
\label{sec:cp1-table-hecke}

Here we present a list of all possible Hecke modifications at a point
$p \in \CP^1$ of all possible rank 2 vector bundles on $\CP^1$.
We will parameterize Hecke modifications of a vector bundle $E$ at a
point $p$ in terms of lines $\ell_p \in \mathbbm{P}(E_p)$, as
described in Theorem \ref{theorem:single-isomorphism}.
Since we are always free to tensor a Hecke modification with a line
bundle, it suffices to consider vector bundles of nonnegative degree.

\begin{theorem}
\label{theorem:cp1-hecke-On-O}
Consider the vector bundle
$\calO(n) \oplus \calO$ for $n \geq 1$ (unstable, instability degree
$n$).
The possible Hecke modifications are
\begin{align}
  \nonumber
  \calO(n) \oplus \calO \leftarrow
  \left\{
  \begin{array}{ll}
    \calO(n) \oplus \calO(-1) &
    \quad \mbox{if $\ell_p = \calO(n)_p$ (a bad line),} \\
    \calO(n-1) \oplus \calO &
    \quad \mbox{otherwise (a good line).}
  \end{array}
  \right.
\end{align}
\end{theorem}

\begin{proof}
(1) The case $\ell_p = \calO(n)_p$.
A Hecke modification
$\alpha:\calO(n) \oplus \calO \rightarrow \calO(n) \oplus \calO(-1)$
corresponding to $\ell_p$ is
\begin{align}
  \nonumber
  \alpha = \left(\begin{array}{cc}
    1 & 0 \\
    0 & f
  \end{array}\right),
\end{align}
where $f:\calO(-1) \rightarrow \calO$ is the unique (up to rescaling
by a constant) nonzero morphism such that $f_p = 0$ on the fibers over
$p$.

(2) The case $\ell_p \neq \calO(n)_p$.
Since $n \geq 1$, we can choose a section $t$ of $\calO(n)$ such that
$t(p) \neq 0$.
Choose a section $s = (at, b)$ of $\calO(n) \oplus \calO$  for
$a, b \in \Complex$ such that $s(p) \neq 0$ and $s \in \ell_p$.
A Hecke modification
$\alpha:\calO(n) \oplus \calO \rightarrow \calO(n-1) \oplus \calO$
corresponding to $\ell_p$ is
\begin{align}
  \nonumber
  \alpha = \left(\begin{array}{cc}
    f & at \\
    0 & b
  \end{array}\right),
\end{align}
where $f:\calO(n-1) \rightarrow \calO(n)$ is the unique (up to
rescaling by a constant) nonzero morphism such that $f_p = 0$ on the
fibers over $p$.
\end{proof}

\begin{theorem}
\label{theorem:cp1-hecke-O-O}
Consider the vector bundle
$\calO \oplus \calO$ (strictly semistable, instability degree 0).
The possible Hecke modifications are
\begin{align}
  \nonumber
  \calO \oplus \calO \leftarrow
  \calO \oplus \calO(-1)\qquad
  \textup{for all $\ell_p$ (all lines are bad)}.
  \end{align}
\end{theorem}

\begin{proof}
Define a section $s = (a,b)$ of $\calO \oplus \calO$ for
$a,b \in \Complex$ such that $s(p) \neq 0$ and $s(p) \in \ell_p$.
If $b = 0$, then a Hecke modification
$\alpha:\calO \oplus \calO(-1) \rightarrow \calO \oplus \calO$
corresponding to $\ell_p$ is
\begin{align}
  \nonumber
  \alpha = \left(\begin{array}{cc}
    a & 0 \\
    0 & f
  \end{array}\right),
\end{align}
and if $b \neq 0$, then a Hecke modification
$\alpha:\calO \oplus \calO(-1) \rightarrow \calO \oplus \calO$
corresponding to $\ell_p$ is
\begin{align}
  \nonumber
  \alpha = \left(\begin{array}{cc}
    a & f \\
    b & 0
  \end{array}\right),
\end{align}
where $f:\calO(-1) \rightarrow \calO$ is the unique (up to rescaling
by a constant) nonzero morphism such that $f_p = 0$ on the fibers over
$p$.
\end{proof}

\subsubsection{Observations}

From this list, we make the following observations:

\begin{lemma}
\label{lemma:cp1-observations}
The following results hold for Hecke modifications of a rank 2 vector
bundle $E$ over $\CP^1$:
\begin{enumerate}
\item
  A Hecke modification of $E$ changes the instability degree by
  $\pm 1$.

\item
  Hecke modification of $E$ corresponding to a line
  $\ell_p \in \mathbbm{P}(E_p)$ changes the instability degree by $-1$
  if $\ell_p$ is a good line and $+1$ if $\ell_p$ is a bad line.

\item
  A generic Hecke modification of $E$ changes the instability degree
  by $-1$ unless $E$ has the minimum possible instability degree
  $0$, in which case all Hecke modifications of $E$ change the
  instability degree by $+1$.
\end{enumerate}
\end{lemma}

\subsection{Moduli spaces $\calP_M^{tot}(\CP^1,m,n)$ and
  $\calP_M(\CP^1,m,n)$}
\label{sec:cp1-moduli-space}

Our goal is to define a moduli space of Hecke modifications that is
isomorphic to the Seidel--Smith space $\calY(S^2,2r)$.
Kamnitzer showed that such a space can be defined as follows:

\begin{definition}[Kamnitzer \cite{Kamnitzer}]
Given distinct points $(p_1, \cdots, p_{2r}) \in (\CP^1)^{2r}$,
define the \emph{Kamnitzer space $\calH(\CP^1,2r)$} to be the subset of
$\calH^{tot}(\CP^1, \calO \oplus \calO, 2r)$
consisting of equivalence classes of sequences of Hecke modifications
$\calO \oplus \calO
\hmod{\alpha_1}{p_1} E_1
\hmod{\alpha_2}{p_2} \cdots
\hmod{\alpha_{2r}}{p_{2r}} E_{2r}$
such that $E_{2r}$ is semistable.
\end{definition}

In particular, the condition that $E_{2r}$ must be semistable implies
that $E_{2r} = \calO(-r) \oplus \calO(-r)$.

\begin{theorem}[Kamnitzer \cite{Kamnitzer}]
The Kamnitzer space $\calH(\CP^1,2r)$ has the structure of a complex
manifold isomorphic to the Seidel--Smith space $\calY(S^2,2r)$.
\end{theorem}

We will describe the Seidel--Smith space $\calY(S^2,2r)$ and
Kamnitzer's isomorphism in Section \ref{sec:seidel-smith}.
We can use the results of Section \ref{sec:sequences-hecke} to
reinterpret the Kamnitzer space $\calH(\CP^1,2r)$ in terms of
parabolic bundles:

\begin{definition}
\label{def:cp1-space-h}
Define
$\calP^{tot}(\CP^1,n) := \calP^{tot}(\CP^1,\calO \oplus \calO,n)$ and
$\calP(\CP^1,2r) := \calP(\CP^1,\calO \oplus \calO,2r)$.
\end{definition}

\begin{theorem}
\label{theorem:iso-kamnitzer-P}
There is a canonical isomorphism
$\calH(\CP^1,2r) \rightarrow \calP(\CP^1,2r)$.
\end{theorem}

\begin{proof}
This follows from restricting the domain and range of the canonical
isomorphism
$\calH^{tot}(\CP^1,\calO \oplus \calO,2r) \rightarrow
\calP^{tot}(\CP^1,\calO \oplus \calO,2r)$
described in Theorem \ref{theorem:canonical-isomorphism}.
\end{proof}

We can also reinterpret the spaces
$\calP^{tot}(\CP^1, n)$ and
$\calP(\CP^1,2r)$ in terms of the moduli spaces of marked parabolic
bundles that we defined in Section \ref{sec:marked-parabolic}.
In what follows, we will choose a global trivialization
$\calO \oplus \calO \rightarrow \CP^1 \times \Complex^2$ and identify
all the fibers of $\calO \oplus \calO$ with $\Complex^2$.
We can then identify lines $\ell_p \in \mathbbm{P}(E_p)$ with
points in $\CP^1$ and speak of lines in different fibers as being
equal or unequal.
In Appendix \ref{sec:vector-bundles} we define a moduli space
$M^{ss}(\CP^1)$ of semistable rank 2 vector bundles with trivial
determinant bundle, and in Appendix \ref{sec:parabolic-bundles} we
define a moduli space $M^s(\CP^1,m)$ of stable rank 2 parabolic
bundles with trivial determinant bundle and $m$ marked points.
From the fact that $M^{ss}(\CP^1) = \{[\calO \oplus \calO]\}$ and
$\Aut(\calO \oplus \calO) = GL(2,\Complex)$, we obtain the following
results:

\begin{theorem}
\label{theorem:Ms-cp1-3}
The moduli space
$M^s(\CP^1,3)$ consists of the single point
$[\calO \oplus \calO,\ell_{q_1},\ell_{q_2},\ell_{q_3}]$, where
$\ell_{q_1},\ell_{q_2},\ell_{q_3}$ are any three distinct lines.
Given any two stable parabolic bundles of the form
$(\calO \oplus \calO,\ell_{q_1},\ell_{q_2},\ell_{q_3})$ and
$(\calO \oplus \calO,\ell_{q_1}',\ell_{q_2}',\ell_{q_3}')$, there is a
unique (up to rescaling by a constant) automorphism
$\phi \in \Aut(\calO \oplus \calO)$ such that
$\phi(\ell_{q_i}) = \ell_{q_i}'$ for $i=1,2,3$.
\end{theorem}

\begin{corollary}
\label{cor:Ms-cp1-3-2}
There is an isomorphism
$M^s(\CP^1,3) \rightarrow M^{ss}(\CP^1)$,
$[\calO \oplus \calO,\ell_{q_1},\ell_{q_2},\ell_{q_3}] \mapsto
[\calO \oplus \calO]$.
\end{corollary}

These results motivate the following definitions of ``marked''
versions of $\calP^{tot}(\CP^1, n)$ and
$\calP(\CP^1, 2r)$:

\begin{definition}
Define $\calP_M^{tot}(\CP^1,n) := \calP_M^{tot}(\CP^1,3,n)$ and
$\calP_M(\CP^1,n) := \calP_M(\CP^1,3,n)$.
\end{definition}

The marked and unmarked versions of these spaces are easily seen to be
isomorphic:

\begin{theorem}
\label{theorem:cp1-parabolic-vb-isomorphism-a}
The spaces $\calP_M^{tot}(\CP^1,n)$ and $\calP^{tot}(\CP^1, n)$ are
(noncanonically) isomorphic.
\end{theorem}

\begin{proof}
Choose three distinct lines $\ell_{q_1}$, $\ell_{q_2}$, $\ell_{q_3}$,
and define an isomorphism
$\calP^{tot}(\CP^1,n) \rightarrow \calP_M^{tot}(\CP^1,n)$ by
\begin{align}
  \nonumber
  [\calO \oplus \calO, \ell_{p_1},\cdots,\ell_{p_n}] \mapsto
  [\calO \oplus \calO,\ell_{q_1},\ell_{q_2},\ell_{q_3},
    \ell_{p_1},\cdots,\ell_{p_n}].
\end{align}
The isomorphism is not canonical, since it depends on the choice of
lines $\ell_{q_1}$, $\ell_{q_2}$, $\ell_{q_3}$
\end{proof}

\begin{theorem}
\label{theorem:cp1-parabolic-vb-isomorphism}
The spaces $\calP_M(\CP^1,n)$ and $\calP(\CP^1,n)$ are
(noncanonically) isomorphic
\end{theorem}

\begin{proof}
This follows from restricting the domain and range of the isomorphism
$\calP^{tot}(\CP^1,n) \rightarrow \calP_M^{tot}(\CP^1,n)$
described in Theorem \ref{theorem:cp1-parabolic-vb-isomorphism-a}
\end{proof}

Our primary motivation for defining $\calP_M^{tot}(\CP^1,n)$ is
to draw a parallel with the case of elliptic curves, which we consider
in Section \ref{sec:elliptic-curves}.
But the space $\calP_M^{tot}(\CP^1,n)$ also has an advantage over
$\calP^{tot}(\CP^1,n)$ in that we can use the marking lines to render
certain constructions canonical.
For example, we can define a canonical version of the noncanonical
isomorphism
$\calP^{tot}(\CP^1,n) = \calP^{tot}(\CP^1,\calO \oplus \calO,n)
\rightarrow (\CP^1)^n$:

\begin{lemma}
\label{lemma:cp1-map-line}
Fix a parabolic bundle $(E,\ell_{q_1},\ell_{q_2},\ell_{q_3})$ such that
$[E,\ell_{q_1},\ell_{q_2},\ell_{q_3}] \in M^s(\CP^1,3)$ and a point
$p \in \CP^1$.
There is a canonical isomorphism
$\mathbbm{P}(E_p) \rightarrow M^{ss}(\CP^1,4) \cong \CP^1$ given by
\begin{align}
  \nonumber
  \ell_p \mapsto [E,\ell_{q_1},\ell_{q_2},\ell_{q_3},\ell_p].
\end{align}
\end{lemma}

\begin{proof}
This follows from the fact that
$(E,\ell_{q_1},\ell_{q_2},\ell_{q_3})$ is stable, so the
lines $\ell_{q_1},\ell_{q_2},\ell_{q_3}$ are all distinct under the
global trivialization of $E$.
\end{proof}

\begin{theorem}
\label{theorem:Pmtot-iso}
There is a canonical isomorphism
$h:\calP_M^{tot}(\CP^1,n) \rightarrow (M^{ss}(\CP^1,4))^n$.
\end{theorem}

\begin{proof}
Define maps
$h_i:\calP_M^{tot}(\CP^1,n) \rightarrow M^{ss}(\CP^1,4)$ for
$i=1,\cdots,n$ by
\begin{align}
  \nonumber
  h_i(
  [E,\ell_{q_1},\ell_{q_2},\ell_{q_3},\ell_{p_1},\cdots,\ell_{p_n}]) =
  [E,\ell_{q_1},\ell_{q_2},\ell_{q_3},\ell_{p_i}].
\end{align}
Then $h := (h_1, \cdots, h_n)$ is an
isomorphism by Theorem \ref{theorem:PMtot} and Lemma
\ref{lemma:cp1-map-line}.
\end{proof}

\begin{remark}
\label{remark:cp1-odd-n}
Definition \ref{def:PM} for $\calP_M(C,m,n)$ implies that
$\calP_M(\CP^1,m,n) = \varnothing$ for odd $n$, since
there are no semistable rank 2 vector bundles of odd degree on
$\CP^1$.
We could alternatively define $\calP_M(\CP^1,m,n)$ by requiring that
$H(E,\ell_{p_1}, \cdots, \ell_{p_n})$ have the minimal possible
instability degree, which is 0 for $n$ even and 1 for $n$ odd.
This condition is equivalent to semistability for $n$ even, but is a
distinct condition for $n$ odd, and gives a nonempty space.
\end{remark}

\subsection{Embedding
  $\calP_M(\CP^1,m,n) \rightarrow M^s(\CP^1,m+n)$}
\label{sec:cp1-hecke-embedding}

We will now describe a canonical open embedding of the space
$\calP_M(\CP^1,m,n)$ into the space of stable parabolic bundles
$M^s(\CP^1,m+n)$.
We first need two Lemmas:

\begin{lemma}
\label{lemma:cp1-embedding-a}
Given a parabolic bundle
$(\calO \oplus \calO,\ell_{p_1},\cdots, \ell_{p_n})$ over $\CP^1$, if
$\ell_{p_1},\cdots,\ell_{p_n}$ are bad in the same direction then
$H(\calO \oplus \calO,\ell_{p_1},\cdots, \ell_{p_n}) =
\calO \oplus \calO(-n)$.
\end{lemma}

\begin{proof}
Since $\ell_{p_1}, \cdots, \ell_{p_n}$ are bad in the same direction,
we have that $\ell_{p_1} = \cdots = \ell_{p_n}$ under a global
trivialization of $\calO \oplus \calO$ in which all the fibers are
identified with $\Complex^2$.
An explicit sequence of Hecke modifications with
$\ell_{p_1} = \cdots = \ell_{p_n}$ is given by
\begin{align}
  \nonumber
  \calO \oplus \calO
  \hmod{\alpha_1}{p_1}
  \calO \oplus \calO(-1)    
  \hmod{\alpha_2}{p_2} \cdots
  \hmod{\alpha_n}{p_n}
  \calO \oplus \calO(-n),
\end{align}
where
$\calO \oplus \calO \hmod{\alpha_1}{p_1} \calO \oplus \calO(-1)$
is a Hecke modification corresponding to the line $\ell_{p_1}$ and
for $i=2,\cdots,n$ we define
\begin{align}
  \nonumber
  \alpha_i = \left(\begin{array}{cc}
    1 & 0 \\
    0 & f_i
  \end{array}\right),
\end{align}
where $f_i:\calO(-i) \rightarrow \calO(-i+1)$ is the unique (up to
rescaling by a constant) nonzero morphism such that $(f_i)_{p_i} = 0$
on the fibers over $p_i$.
\end{proof}

\begin{lemma}
\label{lemma:cp1-embedding-b}
Given a parabolic bundle
$(\calO \oplus \calO,\ell_{p_1},\cdots, \ell_{p_n})$ over $\CP^1$,
if $H(\calO \oplus \calO,\ell_{p_1},\cdots, \ell_{p_n})$ is semistable
then
$(\calO \oplus \calO,\ell_{p_1},\cdots, \ell_{p_n})$ is semistable.
\end{lemma}

\begin{proof}
We will prove the contrapositive, so assume that
$(\calO \oplus \calO,\ell_{p_1},\cdots, \ell_{p_n})$ is unstable.
It follows that more $n/2$ of the lines are bad in the same direction.
Let $s$ denote the number of such lines, and choose a permutation
$\sigma \in \Sigma_n$ such that the first $s$ points of
$(\sigma(p_1), \cdots, \sigma(p_n))$
correspond to these lines.
By Lemma \ref{lemma:cp1-embedding-a} we have that
$H(\calO \oplus \calO, \ell_{\sigma(p_1)}, \cdots, \ell_{\sigma(p_s)})
=
\calO \oplus \calO(-s)$, which has instability degree $s$.
Lemma \ref{lemma:cp1-observations} states that a single Hecke
modification changes the instability degree by $\pm 1$, so
$H(\calO \oplus \calO, \ell_{\sigma(p_1)}, \cdots, \ell_{\sigma(p_n)}) =
H(\calO \oplus \calO,\ell_{p_1},\cdots, \ell_{p_n})$
has instability degree at least $s - (n-s) = 2s - n > 0$, and is thus
unstable.
\end{proof}

\begin{remark}
The converse to Lemma \ref{lemma:cp1-embedding-b} does not always
hold; for example, consider the semistable parabolic bundle
$(\calO \oplus \calO, \ell_{p_1}, \ell_{p_2}, \ell_{p_3},
\ell_{p_4})$ for points $p_i = [1:\mu_i] \in \CP^1$, where
\begin{align}
  \nonumber
  \ell_{p_1} &= [1:0], &
  \ell_{p_2} &= [0:1], &
  \ell_{p_3} &= [1:1], &
  \ell_{p_4} &=
  [(\mu_3 - \mu_1)(\mu_4 - \mu_2) : (\mu_3 - \mu_2)(\mu_4 - \mu_1)].
\end{align}
One can show that
$H(\calO \oplus \calO, \ell_{p_1}, \ell_{p_2}, \ell_{p_3},
\ell_{p_4}) = \calO(-3) \oplus \calO(-1)$, which is unstable.
\end{remark}

\begin{theorem}
\label{theorem:cp1-hecke-embedding}
There is a canonical open embedding
$\calP_M(\CP^1,m,n) \rightarrow M^s(\CP^1,m+n)$.
\end{theorem}

\begin{proof}
Take
$[E,\ell_{q_1},\cdots,\ell_{q_m},\ell_{p_1},\cdots, \ell_{p_n}]
\in \calP_M(\CP^1,m,n)$; note that $E = \calO \oplus \calO$.
Since $(E,\ell_{q_1},\cdots,\ell_{q_m})$ is stable,
fewer than $m/2$ of the lines $\ell_{q_1}, \cdots, \ell_{q_m}$ are
equal under the global trivialization of $E$.
Since $H(E,\ell_{p_1},\cdots, \ell_{p_n})$ is semistable, it follows
from Lemma \ref{lemma:cp1-embedding-b} that
$(E,\ell_{p_1},\cdots, \ell_{p_n})$ is semistable,
so at most $n/2$ of the lines $\ell_{p_1}, \cdots, \ell_{p_n}$ are
equal under the global trivialization of $E$.
It follows that fewer than $(m+n)/2$ of the lines
$\ell_{q_1},\cdots,\ell_{q_m},\ell_{p_1},\cdots, \ell_{p_n}$
are equal, so
$(E,\ell_{q_1},\cdots,\ell_{q_m},\ell_{p_1},\cdots, \ell_{p_n})$
is stable.
So $\calP_M(\CP^1,m,n)$ is a subset of $M^s(\CP^1,m+n)$.
Specifically, the set $\calP_M(\CP^1,m,n)$ consists of points
$[E,\ell_{q_1},\cdots,\ell_{q_m},\ell_{p_1},\cdots, \ell_{p_n}]
\in M^s(\CP^1,m+n)$
such that $(E,\ell_{q_1},\cdots,\ell_{q_m})$ is stable and
$H(E,\ell_{p_1},\cdots, \ell_{p_n})$ is semistable.
Since stability and semistability are open conditions,
we have that
$\calP_M(\CP^1,m,n)$ is an open subset of $M^s(\CP^1,m+n)$.
\end{proof}

\subsection{Examples}
We can generalize the Kamnitzer space $\calH(\CP^1,n)$ to allow
for both even and odd $n$, in analogy with the generalization
described in Remark \ref{remark:cp1-odd-n}:

\begin{definition}
\label{def:Kamnitzer-gen}
Given distinct points $(p_1, \cdots, p_n) \in (\CP^1)^n$,
define the \emph{Kamnitzer space $\calH(\CP^1,n)$} to be the subset of
$\calH^{tot}(\CP^1, \calO \oplus \calO, n)$
consisting of equivalence classes of sequences of Hecke modifications
$\calO \oplus \calO
\hmod{\alpha_1}{p_1} E_1
\hmod{\alpha_2}{p_2} \cdots
\hmod{\alpha_n}{p_n} E_n$
such that $E_n$ has the minimum possible instability degree (0 for $n$
even, 1 for $n$ odd.)
\end{definition}

Here we compute Kamnitzer space $\calH(\CP^1,n)$ for $n=0,1,2,3$.

\subsubsection{Calculate $\calH(\CP^1,0)$}

We have
\begin{align}
  \nonumber
  \calH(\CP^1,0) = \calH^{tot}(\CP^1,0) = \{\calO \oplus \calO\}.
\end{align}

\subsubsection{Calculate $\calH(\CP^1,1)$}

All Hecke modifications of $\calO \oplus \calO$ give
$\calO \oplus \calO(-1)$, which has instability degree 1, so
\begin{align}
  \nonumber
  \calH(\CP^1,1) = \calH^{tot}(\CP^1,1) = \CP^1.
\end{align}

\subsubsection{Calculate $\calH(\CP^1,2)$}
\label{sec:cp1-example-2}

A sequence of two Hecke modifications of
$\calO \oplus \calO$ must have one of two forms:
\begin{align}
  \nonumber
  &\calO \oplus \calO
  \hmod{\alpha_1}{p_1}
  \calO \oplus \calO(-1)
  \hmod{\alpha_2}{p_2}
  \calO(-1) \oplus \calO(-1), &
  \nonumber
  &\calO \oplus \calO
  \hmod{\alpha_1}{p_1}
  \calO \oplus \calO(-1)
  \hmod{\alpha_2}{p_2}
  \calO \oplus \calO(-2).
\end{align}
In the first case the terminal bundle $\calO(-1) \oplus \calO(-1)$
is semistable, whereas in the second case the terminal bundle
$\calO \oplus \calO(-2)$ is unstable.
So $\calH(\CP^1,2)$ is the complement in $\calH^{tot}(\CP^1,2)$ of
sequences of Hecke modifications of the second form.
As we showed in the proof of Lemma \ref{lemma:cp1-embedding-a},
the resulting space is
\begin{align}
  \nonumber
  \calH(\CP^1,2) = (\CP^1)^2 - \{(a,a) \mid a \in \CP^1\}.
\end{align}

\subsubsection{Calculate $\calH(\CP^1,3)$}

Now consider a sequence of three Hecke modifications of
$\calO \oplus \calO$.
The only sequences for which the terminal bundle does not have
instability degree 1 are of the form
\begin{align}
  \nonumber
  &\calO \oplus \calO
  \hmod{\alpha_1}{p_1}
  \calO \oplus \calO(-1)
  \hmod{\alpha_2}{p_2}
  \calO \oplus \calO(-2)
  \hmod{\alpha_3}{p_3}
  \calO \oplus \calO(-3).
\end{align}
So $\calH(\CP^1,3)$ is the complement in $\calH^{tot}(\CP^1,3)$ of
sequences of Hecke modifications of this form.
As we showed in the proof of Lemma \ref{lemma:cp1-embedding-a},
the resulting space is
\begin{align}
  \nonumber
  \calH(\CP^1,3) = (\CP^1)^3 - \{(a,a,a) \mid a \in \CP^1\}.
\end{align}

\subsection{The Seidel--Smith space $\calY(S^2,2r)$}
\label{sec:seidel-smith}

Here we compare the embedding of
$\calP_M(\CP^1,2r)$ into $M^s(\CP^1,2r+3)$ that we defined in
Theorem \ref{theorem:cp1-hecke-embedding} with an embedding
of the Seidel--Smith space $\calY(S^2,2r)$ into $M^s(\CP^1,2r+3)$ due
to Woodward.
We begin by defining the Seidel--Smith space $\calY(S^2,2r)$.

\begin{definition}
We define the {\em Slodowy slice} $S_{2r}$ to be the subspace of
$\mathfrak{gl}(2r,\Complex)$ consisting of matrices with $2 \times 2$
identity matrices $I$ on the superdiagonal, arbitrary $2 \times 2$
matrices in the left column, and zeros everywhere else.
\end{definition}

\begin{example}
Elements of $S_6$ have the form
\begin{align}
  \nonumber
  \left(
  \begin{array}{ccc}
    Y_1 & I & 0 \\
    Y_2 & 0 & I \\
    Y_3 & 0 & 0 \\
  \end{array}
  \right),
\end{align}
where $Y_1$, $Y_2$, and $Y_3$ are arbitrary $2\times 2$ complex
matrices.
\end{example}

\begin{definition}
Define a map $\chi:S_{2r} \rightarrow \Complex^{2r}/\Sigma_{2r}$ that
sends a matrix to the multiset of the roots of its characteristic
polynomial, where a root of multiplicity $m$ occurs $m$ times in the
multiset.
\end{definition}

\begin{definition}
\label{def:seidel-smith}
Given distinct points $(\mu_1, \cdots, \mu_{2r}) \in \Complex^{2r}$,
define the {\em Seidel--Smith space} $\calY(S^2,2r)$ to be the
fiber $\chi^{-1}(\{\mu_1, \cdots, \mu_{2r}\})$.
For simplicity, we are suppressing the dependence of
$\calY(S^2,2r)$ on $\mu_1, \cdots, \mu_{2r}$ in the notation.
This space was introduced in \cite{Seidel-Smith}, which denotes
$\calY(S^2,2r)$ by $\calY_r$.
\end{definition}

The Seidel--Smith space $\calY(S^2,2r)$ naturally has the structure of
a complex manifold, in fact a smooth complex affine variety.
In what follows, it will be useful to define
local coordinates $\xi:U \rightarrow V$ on $\CP^1$, where
$U = \{[1:z] \mid z \in \Complex\} \subset \CP^1$,
$V = \Complex$, and
$\xi([1:z]) = z$.
We define points $p_i := \xi^{-1}(\mu_i) \in \CP^1$ corresponding to
$\mu_i$ for $i = 1, \cdots, 2r$.

\subsubsection{Kamnitzer isomorphism
  $\calH(\CP^1,2r) \rightarrow \calY(S^2,2r)$}
\label{sec:kamnitzer-isomorphism}

Here we describe an isomorphism due to Kamnitzer from the space of
Hecke modifications $\calH(\CP^1,2r)$ to the Seidel--Smith space
$\calY(S^2,2r)$.

Define global meromorphic sections $s_n$ of
$\calO(n)$ such that $\Div s_n = n \cdot [\infty]$.
For each rank 2 vector bundle $E = \calO(n) \oplus \calO(m)$, define
standard meromorphic sections
\begin{align}
  \nonumber
  e_E^1 &= (s_n, 0) ,&
  e_E^2 &= (0, s_m),
\end{align}
and define a standard local trivialization
$\psi_E:\pi_E^{-1}(U) \rightarrow U \times \Complex^2$
of $E$ over $U$ by
\begin{align}
  \nonumber
  e_E^1(p) \mapsto (p, (1,0)), &&
  e_E^2(p) \mapsto (p, (0,1)).
\end{align}

Consider an element of $\calH(\CP^1,2r)$:
\begin{align}
  \label{eqn:kamnitzer-isomorphism-seq}
  [E_0
    \hmod{\alpha_1}{p_1} E_1
    \hmod{\alpha_2}{p_2} \cdots
    \hmod{\alpha_{2r}}{p_{2r}} E_{2r}],
\end{align}
where $E_0 = \calO \oplus \calO$.
Define rank 2 free $\Complex[z]$-modules $L_i$ for $i=0,\cdots,2r$ as
spaces of sections of $E_i$ over $U$:
\begin{align}
  \nonumber
  L_i &= \Gamma(U,E_i) = \Complex[z]\cdot\{e_{E_i}^1, e_{E_i}^2\}.
\end{align}
The sequence of Hecke modifications
(\ref{eqn:kamnitzer-isomorphism-seq}) then yields a sequence of
$\Complex[z]$-module morphisms $\bar{\alpha}_i$:
\begin{eqnarray}
  \nonumber
  \begin{tikzcd}
    L_0 &
    \arrow{l}[swap]{\bar{\alpha}_1} L_1 &
    \arrow{l}[swap]{\bar{\alpha}_2} \cdots &
    \arrow{l}[swap]{\bar{\alpha}_{2r}} L_{2r}.
  \end{tikzcd}
\end{eqnarray}
We can also view $\bar{\alpha}_i$ as a holomorphic map
$\bar{\alpha}_i:V \rightarrow M(2,\Complex)$, defined as in Definition
\ref{def:single-hecke} such that
\begin{align}
  \nonumber
  (\psi_{E_{i-1}} \circ \alpha_i \circ \psi_{E_i}^{-1})(q,v) =
  (q,\bar{\alpha}_i(\xi(q))v).
\end{align}
Define an $2r$-dimensional complex vector space $V$ by
\begin{align}
  \nonumber
  V = \coker (\bar{\alpha}_1 \circ \bar{\alpha}_2 \circ \cdots \circ
  \bar{\alpha}_{2r}) =
  L_0/(\bar{\alpha}_1 \circ \bar{\alpha}_2 \circ \cdots \circ
  \bar{\alpha}_{2r})(L_{2r}).
\end{align}
One can show that an ordered basis for $V$ is given by
\begin{align}
  \nonumber
  (
  z^{r-1}e_{E_0}^1,\,
  z^{r-1}e_{E_0}^2,\,  \cdots,
  ze_{E_0}^1,\, ze_{E_0}^2,\,
  e_{E_0}^1,\, e_{E_0}^2
  ).
\end{align}
Note that $z$ acts $\Complex$-linearly on $V$, and thus defines a
$2r \times 2r$ complex matrix $A$ relative to this basis.

\begin{theorem}[Kamnitzer \cite{Kamnitzer}]
We have an isomorphism
$\calH(\CP^1,2r) \rightarrow \calY(S^2,2r)$ given by
\begin{align}
  \nonumber
  [E_0
    \hmod{\alpha_1}{p_1} E_1
    \hmod{\alpha_2}{p_2} \cdots
    \hmod{\alpha_{2r}}{p_{2r}} E_{2r}]
  \mapsto A.
\end{align}
\end{theorem}

To perform calculations, it is
useful to have explicit expressions for the maps $\bar{\alpha}_i$.
For each vector bundle $E$ and point $p = [1:\mu] \in U$, we use the
standard trivialization
$\psi_E:\pi^{-1}(E) \rightarrow U \times \Complex^2$ to
identify $\mathbbm{P}(E_p)$ with $\CP^1$.
For each line $\ell_p \in \mathbbm{P}(E_p) = \CP^1$, we give a
holomorphic map
$\bar{\alpha}:V \rightarrow M(2,\Complex)$ that describes a Hecke
modification $\alpha:F \rightarrow E$ corresponding to $\ell_p$: \\
\\
Hecke modifications of $\calO(n) \oplus \calO$ for $n \geq 1$:
\begin{align}
  \nonumber
  \ell_p &= [1:0]: &
  \calO(n) \oplus \calO
  &\hmod{\alpha}{p}
  \calO(n) \oplus \calO(-1), &
  \bar{\alpha}(z) &= \left(\begin{array}{cc}
    1 & 0\\
    0 & z - \mu
    \end{array}
  \right), \\
  \nonumber
  \ell_p &= [\lambda:1]: &
  \calO(n) \oplus \calO
  &\hmod{\alpha}{p}
  \calO(n-1) \oplus \calO, &
  \bar{\alpha}(z) &= \left(\begin{array}{cc}
    z-\mu & \lambda\\
    0 & 1
  \end{array}
    \right).
\end{align}
Hecke modifications of $\calO \oplus \calO$:
\begin{align}
  \nonumber
  \ell_p &= [1:0]: &
  \calO \oplus \calO
  &\hmod{\alpha}{p}
  \calO \oplus \calO(-1), &
  \bar{\alpha}(z) &= \left(\begin{array}{cc}
    1 & 0 \\
    0 & z - \mu
  \end{array}
  \right), \\
  \nonumber
  \ell_p &= [\lambda:1]: &
  \calO \oplus \calO
  &\hmod{\alpha}{p}
  \calO \oplus \calO(-1), &
  \bar{\alpha}(z) &= \left(\begin{array}{cc}
    \lambda & z - \mu\\
    1 & 0
  \end{array}
  \right).
\end{align}

\subsubsection{Woodward embedding
  $\calY(S^2,2r) \rightarrow M^s(\CP^1,2r+3)$}
\label{sec:woodward-embedding}

Here we describe an embedding due to Woodward \cite{Woodward} of the
Seidel--Smith space $\calY(S^2,2r)$ into the space of stable rank 2
parabolic bundles $M^s(\CP^1,2r+3)$.
We first make the following definition:

\begin{definition}
Given distinct points $(p_1, \cdots, p_n) \in (\CP^1)^n$,
define a subspace $\calP^{ss}(\CP^1,n)$ of
$\calP^{tot}(\CP^1,\calO \oplus \calO,n)$ consisting of
semistable parabolic bundles
$(\calO \oplus \calO, \ell_{p_1}, \cdots, \ell_{p_n})$.
\end{definition}

In particular, $\calP^{ss}(\CP^1,n)$ consists of parabolic bundles
$(\calO \oplus \calO, \ell_{p_1}, \cdots, \ell_{p_n})$ for which at
most $n/2$ of the lines are equal to any given line under a global
trivialization of $\calO \oplus \calO$.
Given distinct points
$(q_1, q_2, q_3, p_1, \cdots, p_n) \in (\CP^1)^{n+3}$ and
distinct lines $\ell_{q_1}, \ell_{q_2}, \ell_{q_3} \in \CP^1$, we can
define an embedding
$\calP^{ss}(\CP^1,n) \rightarrow M^s(\CP^1,n+3)$,
\begin{align}
  \nonumber
  (\calO \oplus \calO, \ell_{p_1}, \cdots, \ell_{p_n}) \mapsto
  [\calO \oplus \calO, \ell_{q_1}, \ell_{q_2}, \ell_{q_3},
    \ell_{p_1}, \cdots, \ell_{p_n}].
\end{align}
We will define an embedding
$\calY(S^2,2r) \rightarrow \calP^{ss}(\CP^1,2r)$.
Composing with
$\calP^{ss}(\CP^1,2r) \rightarrow M^s(\CP^1,2r+3)$
will then yield the Woodward embedding.
We first define some vectors.
Define vectors $x, y \in \Complex^2$ by
\begin{align}
  \nonumber
  x &= (1,0), &
  y &= (0,1).
\end{align}
Define vectors
$x_1, \cdots, x_r, y_1, \cdots, y_r \in \Complex^{2r}$ by
\begin{align}
  \nonumber
  x_1 &= (x,\, 0,\, \cdots,\, 0), &
  x_2 &= (0,\, x,\, 0,\, \cdots,\, 0), &
  &\cdots, &
  x_r &= (0,\, \cdots,\, 0,\, x), \\
  \nonumber
  y_1 &= (y,\, 0,\, \cdots,\, 0), &
  y_2 &= (0,\, y,\, 0,\, \cdots,\, 0), &
  &\cdots, &
  y_r &= (0,\, \cdots,\, 0,\, y).
\end{align}
Define vectors $x(\mu), y(\mu) \in \Complex^{2r}$ by
\begin{align}
  \nonumber
  x(\mu) &=
  (\mu^{r-1}x,\, \mu^{r-2}x,\, \cdots,\, \mu x,\, x) =
  \mu^{r-1}x_1 + \mu^{r-2}x_2 + \cdots + \mu x_{r-1} + x_r, \\
  \nonumber
  y(\mu) &=
  (\mu^{r-1}y,\, \mu^{r-2}y,\, \cdots,\, \mu y,\, y) =
  \mu^{r-1}y_1 + \mu^{r-2}y_2 + \cdots + \mu y_{r-1} + y_r.
\end{align}
We use the vectors to define a subspace
$W(s,t)$ of $\Complex^{2r}$, and we calculate its dimension:

\begin{definition}
Given $(s,t) \in \Complex^2$, define a subspace
$W(s,t) =
\Complex \cdot \{sx(\mu) + ty(\mu) \mid \mu \in \Complex\}$ of
$\Complex^{2r}$.
\end{definition}

\begin{lemma}
\label{lemma:W}
For $(s,t) \in \Complex^2 - \{0\}$, we have that $\dim W(s,t) = r$.
\end{lemma}

\begin{proof}
Define a vector $w(s,t,\mu) \in \Complex^{2r}$ by
\begin{align}
  \label{eqn:w-s-t-mu-basis}
  w(s,t,\mu) := sx(\mu) + ty(\mu) =
  \mu^{r-1}(sx_1 + t y_1) + \cdots +
  \mu(sx_{r-1} + ty_{r-1}) +
  (sx_r + t y_r).
\end{align}
Let $S \subset \Complex^{2r}$ denote the span of the linearly
independent vectors
$\{sx_1 + ty_1,\, \cdots,\, sx_r + ty_r\}$.
Clearly $W(s,t) \subseteq S$.
Form an $r \times r$ matrix $V$ whose $i$-th row vector
consists of the components of $w(s,t,i)$ relative to the ordered basis
$(sx_1 + ty_1,\, \cdots,\, sx_r + ty_r)$ of $S$.
From equation (\ref{eqn:w-s-t-mu-basis}), it follows that the $(i,j)$
matrix element of $V$ is given by
\begin{align}
  \nonumber
  V_{ij} = (i)^{r-j}.
\end{align}
So $V$ is a Vandermonde matrix corresponding to the distinct integers
$(1,2,\cdots,r)$, and thus has nonzero determinant.
It follows that the vectors $\{w(s,t,1),\cdots, w(s,t,r)\}$ are
linearly independent, hence $W(s,t) = S$ and
$\dim W(s,t) = \dim S = r$.
\end{proof}

We are now ready to define the Woodward embedding.
Take a matrix $A \in \calY(S^2,2r)$.
Let $v(\mu) \in \Complex^{2r}$ be a left-eigenvector of $A$ with
eigenvalue $\mu$:
\begin{align}
  \nonumber
  v(\mu)A = \mu v(\mu).
\end{align}
Given the form of $A$, it follows that
\begin{align}
  \nonumber
  v(\mu) = X(\mu) x(\mu) + Y(\mu) y(\mu)
\end{align}
for some $X(\mu),Y(\mu) \in \Complex$.
Since $A \in \calY(S^2,2r)$, the eigenvalues of $A$ are
$\mu_1,\, \cdots,\, \mu_{2r} \in \Complex$.
Define lines $\ell_{p_i} \in \CP^1$ for $i= 1, \cdots, 2r$ by
\begin{align}
  \nonumber
  \ell_{p_i} = [X(\mu_i) : Y(\mu_i)].
\end{align}

\begin{theorem}[Woodward \cite{Woodward}]
\label{theorem:woodward-embedding}
We have an embedding
$\calY(S^2,2r) \rightarrow \calP^{ss}(\CP^1,2r)$,
$A \mapsto (\calO \oplus \calO, \ell_{p_1}, \cdots, \ell_{p_{2r}})$.
\end{theorem}

\begin{proof}
A priori the codomain of the map is $\calP^{tot}(\CP^1,2r)$, so we
need to show that the image is in fact contained in
$\calP^{ss}(\CP^1,2r)$.
Note that if $\ell_{p_i} = [s:t]$ then  $v(\mu_i) \in W(s,t)$.
Since the eigenvalues $\mu_1, \cdots, \mu_{2r}$ are distinct, the
eigenvectors $\{v(\mu_1), \cdots, v(\mu_{2r})\}$ are
linearly independent, so the maximum number of
eigenvectors that can live in $W(s,t)$ is $\dim W(s,t) = r$ by
Lemma \ref{lemma:W}.
So at most $r$ of the lines $\ell_{p_1}, \cdots, \ell_{p_{2r}}$ can be
equal to any given line $[s:t]$ in $\CP^1$, and thus
$(\calO \oplus \calO, \ell_{p_1}, \cdots, \ell_{p_{2r}})$ is
semistable.
\end{proof}

Lemma \ref{lemma:cp1-embedding-b} states that we have an embedding
$\calP(\CP^1,2r) \rightarrow \calP^{ss}(\CP^1,2r)$.
We can precompose this embedding with the
canonical isomorphism $\calH(\CP^1,2r) \rightarrow \calP(\CP^1,2r)$
described in Theorem \ref{theorem:iso-kamnitzer-P}
to obtain an embedding
$\calH(\CP^1,2r) \rightarrow \calP^{ss}(\CP^1,2r)$, and we obtain a
commutative diagram
\begin{eqnarray}
  \nonumber
  \begin{tikzcd}
    \calH(\CP^1,2r) \arrow{r}\arrow{d}[swap]{\cong} &
    \calP^{ss}(\CP^1,2r) \arrow{d} \\
    \calP_M(\CP^1,2r) \arrow{r} &
    M^{ss}(\CP^1,2r+3),
  \end{tikzcd}
\end{eqnarray}
where the bottom horizontal arrow is the embedding described in
Theorem \ref{theorem:cp1-hecke-embedding}.
It is interesting to compare the embedding
$\calH(\CP^1,2r) \rightarrow \calP^{ss}(\CP^1,2r)$
to the embedding
$\calY(S^2,2r) \rightarrow \calP^{ss}(\CP^1,2r)$ from Theorem
\ref{theorem:woodward-embedding}.
We make the following conjecture:

\begin{conjecture}
\label{conjecture:woodward-hecke}
There is a commutative diagram
\begin{eqnarray}
  \nonumber
  \begin{tikzcd}
    \calH(\CP^1,2r) \arrow{r}\arrow{d}{\cong} &
    \calP^{ss}(\CP^1,2r) \arrow{d}{\cong} \\
    \calY(S^2,2r) \arrow{r} &
    \calP^{ss}(\CP^1,2r),
    \end{tikzcd}
\end{eqnarray}
where the left downward arrow is the Kamnitzer isomorphism and the
right downward arrow is the map on parabolic bundles induced by
$\phi \in \Aut(\calO \oplus \calO) = GL(2,\Complex)$, where
\begin{align}
  \nonumber
  \phi =
  \left(\begin{array}{cc}
    0 & -1 \\
    1 & 0
  \end{array}\right).
\end{align}
\end{conjecture}

\begin{theorem}
Conjecture \ref{conjecture:woodward-hecke} holds
for $r=1$ and $r=2$.
\end{theorem}

\begin{proof}
This can be shown by a direct calculation.
\end{proof}

\section{Elliptic curves}
\label{sec:elliptic-curves}

\subsection{Vector bundles on elliptic curves}
\label{sec:elliptic-vector-bundles}

Vector bundles on elliptic curves have been classified by Atiyah
\cite{Atiyah}:

\begin{definition}
Define $\calE(r,d)$ to be the set of isomorphism classes of
indecomposable vector bundles of rank $r$ and degree $d$ on an
elliptic curve $X$.
\end{definition}

The set $\calE(r,d)$ naturally has the structure of a complex
manifold, and we have the following result:

\begin{theorem}[Atiyah \cite{Atiyah}]
There are isomorphisms
$\Jac(X) \rightarrow \calE(r,d)$ for all $r$ and $d$.
\end{theorem}

In particular, $\calE(1,d)$ is the set of isomorphism classes of line
bundles of degree $d$, and the
isomorphism $\Jac(X) \rightarrow \calE(1,d)$ is given by
$[L] \mapsto [L \otimes \calO(d\cdot e)]$ for a choice of
basepoint $e \in X$.
Here we summarize the facts we will need regarding line bundles and
rank 2 vector bundles on elliptic curves.
Results that are well-known will be stated without proof; full proofs
can be found in \cite{Atiyah,Iena,Teixidor}.

\begin{definition}
We say that a degree 0 line bundle $L$ is \emph{2-torsion} if
$L^2 = \calO$.
\end{definition}

There are four 2-torsion line bundles on an elliptic curve.
We will denote the 2-torsion line bundles by $L_i$ for $i=1,2,3,4$,
with the convention that $L_1 = \calO$.

\begin{definition}
Given line bundles $L$ and $M$ on an elliptic curve, an
\emph{extension} of $L$ by $M$ is an exact sequence
\begin{eqnarray}
  \nonumber
  \begin{tikzcd}
    0 \arrow{r} &
    M \arrow{r} &
    E \arrow{r} &
    L \arrow{r} &
    0,
  \end{tikzcd}
\end{eqnarray}
where $E$ is a rank 2 vector bundle.
\end{definition}

\begin{lemma}
Given line bundles $L$ and $M$ on an elliptic curve, equivalence
classes of extensions of $L$ by $M$ are classified by
$\Ext^1(L,M) = H^0(L \otimes M^{-1})$.
\end{lemma}

\begin{lemma}[{Teixidor \cite[Lemma 4.5]{Teixidor}}]
\label{lemma:deg-indecomp}
If $[E] \in {\mathcal E}(2,d)$, then $h^0(E) = 0$ if $d < 0$ and
$h^0(E) = d$ if $d > 0$, where $h^0(E) := \dim H^0(E)$.
\end{lemma}

We will now list the rank 2 vector bundles on an elliptic curve $X$.
Up to tensoring with a line bundle, we have the following vector
bundles:

\subsubsection{Rank 2 decomposable vector bundles}

Decomposable bundles have the form $L_1 \oplus L_2$, where $L_1$ and
$L_2$ are line bundles.
The instability degree of $L_1 \oplus L_2$ is
$|\deg L_1 - \deg L_2|$, so $L_1 \oplus L_2$ is strictly semistable if
$\deg L_1 = \deg L_2$ and unstable otherwise.
The proof of the following result is straightforward:

\begin{lemma}
Let $E = L_1 \oplus L_2$, where $L_1$ and $L_2$ are line bundles such
that $\deg L_1 > \deg L_2$.
At a point $p \in X$ the line $(L_1)_p$ is bad, and all other lines in
$\mathbbm{P}(E_p)$ are good.
\end{lemma}

A semistable decomposable bundle must have even degree, so after
tensoring with a suitable line bundle it has the form
$L \oplus L^{-1}$, where $L$ is a degree 0 line bundle.
There are two subclasses of such bundles: the four bundles
$L_i \oplus L_i$, and bundles $L \oplus L^{-1}$ such that
$L^2 \neq \calO$.
These two subclasses of semistable decomposable bundles have very
different properties:

\begin{lemma}
\label{lemma:Li-Li}
The bundle $L_i \oplus L_i$ has no good lines, and
$\Aut(L_i \oplus L_i) = GL(2,\Complex)$.
\end{lemma}

\begin{lemma}
\label{lemma:L-Lm1}
Let $E = L \oplus L^{-1}$, where $L$ is a degree 0 line bundle such
that $L^2 \neq \calO$.
The automorphism group $\Aut(E)$ is the subgroup of $GL(2,\Complex)$
matrices of the form
\begin{align}
  \nonumber
  \left(\begin{array}{cc}
    A & 0 \\
    0 & D
  \end{array}\right).
\end{align}
At a point $p \in X$ the lines $L_p$ and $(L^{-1})_p$ are bad, and all
other lines in $\mathbbm{P}(E_p)$ are good.
Given a pair of good lines $\ell_p, \ell_p' \in \mathbbm{P}(E_p)$,
there is a unique (up to rescaling by a constant) automorphism
$\phi \in \Aut(E)$ such that $\phi(\ell_p) = \phi(\ell_p')$.
\end{lemma}

The proofs of Lemmas \ref{lemma:Li-Li} and \ref{lemma:L-Lm1} are
straightforward and have been omitted.

\subsubsection{Rank 2 degree 0 indecomposable bundles}

There is a unique indecomposable bundle $F_2$ that can be
obtained via an extension of $\calO$ by $\calO$:
\begin{eqnarray}
  \begin{tikzcd}
    \label{ses:F2}
    0 \arrow{r} &
    \calO \arrow{r}{\alpha} &
    F_2 \arrow{r}{\beta} &
    \calO \arrow{r} &
    0.
  \end{tikzcd}
\end{eqnarray}
The bundle $F_2$ is strictly semistable, and hence has instability
degree 0.
The map
$\Jac(X) \rightarrow \calE(2,0)$, $[L] \mapsto [F_2 \otimes L]$ is
an isomorphism, so in particular
$F_2 \otimes L = F_2$ if and only if $L = \calO$.

\begin{lemma}
\label{lemma:L-F2}
If $L$ is a degree 0 line bundle, then
\begin{align}
  \nonumber
  \Hom(L,F_2) &=
  \left\{
  \begin{array}{ll}
    \Complex \cdot \alpha & \quad \mbox{if $L = \calO$,} \\
    0 & \quad \mbox{otherwise.}
  \end{array}
  \right.
  &
  \Hom(F_2,L) &=
  \left\{
  \begin{array}{ll}
    \Complex \cdot \beta & \quad \mbox{if $L = \calO$,} \\
    0 & \quad \mbox{otherwise.}
  \end{array}
  \right.
\end{align}
\end{lemma}

\begin{proof}
Apply $\Hom(L,-)$ to the short exact sequence (\ref{ses:F2}) to obtain
\begin{eqnarray}
  \label{les:F2-hom-L}
  \begin{tikzcd}
    0 \arrow{r} &
    \Hom(L,\calO) \arrow{r}{\alpha_*} &
    \Hom(L,F_2) \arrow{r}{\beta_*} &
    \Hom(L,\calO) \arrow{r}{\delta} &
    \Ext^1(L,\calO).
  \end{tikzcd}
\end{eqnarray}
If $L \neq \calO$, then $\Hom(L,\calO)=0$ and the long exact sequence
(\ref{les:F2-hom-L}) implies that $\Hom(L,F_2) = 0$.
So assume $L = \calO$.
Then $\Hom(L,\calO) = \Complex \cdot 1_{\calO}$.
To prove that $\Hom(L,F_2) = \Complex \cdot \alpha$, it suffices to
show that $\delta(1_{\calO}) \neq 0$.
Assume for contradiction that this is not the case.
Then the long exact sequence (\ref{les:F2-hom-L}) implies that
there is a morphism $f \in \Hom(L,F_2)$ such that
$\beta_*(f) = \beta \circ f = 1_{\calO}$.
It follows that the short exact sequence (\ref{ses:F2}) splits,
contradiction.

The claim regarding $\Hom(F_2,L)$ can be proven in a similar manner by
applying $\Hom(-,L)$ to the short exact sequence (\ref{ses:F2}).
\end{proof}

\begin{lemma}
\label{lemma:sections-F2}
Given a point $q \in X$, there are nonzero sections $t_0$ and $t_1$ of
$F_2 \otimes \calO(q)$ such that
\begin{enumerate}
\item
$H^0(F_2 \otimes \calO(q)) = \Complex \cdot \{t_0, t_1\}$,

\item
$\Div t_0 = 0$ and $\Div t_1 = q$,

\item
$t_1(p) \in \calO(q)_p$ for all $p \in X$, where
$\calO(q) \rightarrow F_2 \otimes \calO(q)$ is the unique degree 1
line subbundle of $F_2 \otimes \calO(q)$,

\item
$\{t_0(p), t_1(p)\}$ are linearly
independent for all $p \in X$ such that $p \neq q$.
\end{enumerate}
\end{lemma}

\begin{proof}
Tensoring $\alpha:\calO \rightarrow F_2$ with
$\calO(q)$ and precomposing with the unique (up to
rescaling by a constant) nonzero morphism
$\calO \rightarrow \calO(q)$, we obtain a section $t_1$ of
$F_2 \otimes \calO(q)$ such that $\Div t_1 = q$ and
$t_1(p) \in \calO(q)_p$ for all $p \in X$.
By Lemma \ref{lemma:deg-indecomp} we have that
$h^0(F_2 \otimes \calO(q)) = 2$, so we can choose
a section $t_0$ of $F_2 \otimes \calO(q)$ linearly independent from
$t_1$.

We claim that $\Div t_0 = 0$.
Assume for contradiction that this is not the case.
We obtain a subbundle
$\calO(\Div t_0) \rightarrow F_2 \otimes \calO(q)$, and by
semistability of $F_2 \otimes \calO(q)$ it follows that
$\Div t_0 = p$ for some $p \in X$.
We thus obtain a subbundle
$\calO(p) \rightarrow F_2 \otimes \calO(q)$, hence a subbundle
$\calO(p-q) \rightarrow F_2$.
But this contradicts Lemma \ref{lemma:L-F2} unless $p = q$, in which
case $t_0$ and $t_1$ are linearly dependent.

We claim $t_0(p)$ and $t_1(p)$ are linearly independent at all points
$p \in X$ such that $p \neq q$.
Assume for contradiction that they are linearly dependent at some
point $p$ distinct from $q$.
Then we can choose a nonzero section
$s = at_0 + bt_1$ of $F_2 \otimes \calO(q)$ for $a, b \in \Complex$
such that $s(p) = 0$.
We thus obtain a subbundle
$\calO(\Div s) \rightarrow F_2 \otimes \calO(q)$.
We have that $p \in \Div s$, so semistability of
$F_2 \otimes \calO(q)$ implies that $\Div s = p$.
We thus obtain a subbundle
$\calO(p) \rightarrow F_2 \otimes \calO(q)$, hence a subbundle
$\calO(p-q) \rightarrow F_2$.
But this contradicts Lemma \ref{lemma:L-F2}.
\end{proof}

\begin{lemma}
\label{lemma:F2}
The automorphism group $\Aut(F_2)$ is the subgroup of $GL(2,\Complex)$
matrices of the form
\begin{align}
  \nonumber
  \left(\begin{array}{cc}
    A & B \\
    0 & A
  \end{array}\right).
\end{align}
At a point $p \in X$ the line $\calO_p$ is bad, where
$\calO \rightarrow F_2$ is the unique degree 0 subbundle of $F_2$,
and all other lines in $\mathbbm{P}((F_2)_p)$ are good.
Given a pair of good lines $\ell_p, \ell_p' \in \mathbbm{P}((F_2)_p)$,
there is a unique (up to rescaling by a constant) automorphism
$\phi \in \Aut(E)$ such that $\phi(\ell_p) = \phi(\ell_p')$.
\end{lemma}

\begin{proof}
Apply $\Hom(-,F_2)$ to the short exact sequence (\ref{ses:F2}) to
obtain
\begin{eqnarray}
  \label{ses:F2-hom}
  \begin{tikzcd}
    0 \arrow{r} &
    \Hom(\calO, F_2) \arrow{r}{\beta^*} &
    \Hom(F_2, F_2) \arrow{r}{\alpha^*} &
    \Hom(\calO, F_2).
  \end{tikzcd}
\end{eqnarray}
Note that $\alpha^*(1_{F_2}) = \alpha$, so Lemma \ref{lemma:L-F2}
implies that $\alpha^*$ is surjective and thus the
sequence (\ref{ses:F2-hom}) is in fact short exact.
It follows that $\Hom(F_2,F_2) = \Complex\cdot \{1_{F_2},\, \eta\}$,
where $\eta := \beta^*(\alpha) = \alpha \circ \beta$.
Note that $\eta \circ \eta = 0$, so we can define an injective group
homomorphism $\Aut(F_2) \rightarrow GL(2,\Complex)$ by
\begin{align}
  \nonumber
  A\, 1_{F_2} + B\, \eta \mapsto
  \left(\begin{array}{cc}
    A & B \\
    0 & A
  \end{array}\right).
\end{align}

The fact that $\calO_p$ is the unique bad line of $(F_2)_p$ follows
from Lemma \ref{lemma:L-F2}.
Given good lines $\ell_p, \ell_p' \in \mathbbm{P}((F_2)_p)$,
choose nonzero vectors $v, v', w \in (F_2)_p$ such that
$v \in \ell_p$, $v' \in \ell_p'$, and $w \in \calO_p$.
Since $\ell_p \neq \calO_p$, it follows that $\{v,w\}$ is a basis
for $(F_2)_p$.
Define $a, b \in \Complex$ such that $v' = av + bw$; note that
since $\ell_p' \neq \calO_p$ we have that $a \neq 0$.
Define $c \in \Complex$ such that $\eta_p(v) = cw$;
note that since $\eta_p(w) = 0$ and $\eta_p \neq0$, we have that
$c \neq 0$.
Then $v' = \phi_p(v)$, where
$\phi = a 1_F + (b/c) \eta \in \Aut(F_2)$.
Hence $\phi(\ell_p) = \ell_p'$, and $\phi$ is clearly unique up to
rescaling by a constant.
\end{proof}

\subsubsection{Rank 2 degree 1 indecomposable bundles}

Given a point $p \in X$, there is a unique degree 1 indecomposable
bundle $G_2(p)$ that can be obtained via an extension of $\calO(p)$ by
$\calO$:
\begin{eqnarray}
  \nonumber
  \begin{tikzcd}
    0 \arrow{r} &
    \calO \arrow{r} &
    G_2(p) \arrow{r} &
    \calO(p) \arrow{r} &
    0.
  \end{tikzcd}
\end{eqnarray}
The bundle $G_2(p)$ is stable, with instability degree $-1$.
The map ${\mathcal E}(1,1) \rightarrow {\mathcal E}(2,1)$,
$[\calO(p)] \mapsto [G_2(p)]$ is an isomorphism, with inverse
isomorphism given by
$\det:{\mathcal E}(2,1) \rightarrow {\mathcal E}(1,1)$,
$[E] \rightarrow [\det E]$.
It follows that for any degree 0 divisor $D$ on $X$ we have that
\begin{align}
  \nonumber
  G_2(p + 2D) = G_2(p) \otimes \calO(D),
\end{align}
and in particular $G_2(p) \otimes L \cong G_2(p)$ if and only if
$L^2 = \calO$.

\begin{lemma}
We have that $\Aut(G_2(p)) = \Complex^\times$ consists only of trivial
automorphisms that scale the fibers by a constant factor.
\end{lemma}

\begin{proof}
This follows from the fact that $G_2(p)$ is stable.
\end{proof}

\begin{lemma}
\label{lemma:g2-subbundle}
Any degree 0 line bundle $L$ is a subbundle of $G_2(p)$ via a unique
(up to rescaling by a constant) inclusion map $L \rightarrow G_2(p)$.
\end{lemma}

\begin{proof}
Let $L$ be a degree 0 line bundle.
By Lemma \ref{lemma:deg-indecomp} we have that
$h^0(G_2(p) \otimes L^{-1}) = 1$, hence $G_2(p) \otimes L^{-1}$ has
a nonzero section $s$.
We thus obtain a subbundle
$\calO(\Div s) \rightarrow G_2(p) \otimes L^{-1}$.
By stability of $G_2(p) \otimes L^{-1}$, we must have
$\Div s = 0$.
Tensoring with $L$, we obtain a subbundle $L \rightarrow G_2(p)$.
The claim regarding uniqueness follows from the fact that
$h^0(G_2(p) \otimes L^{-1}) = 1$.
\end{proof}

\begin{corollary}
All lines of $G_2(p)$ are bad.
\end{corollary}

\begin{proof}
This is shown in Theorem \ref{theorem:hecke-G2}.
\end{proof}

\subsection{List of all possible single Hecke modifications}
\label{sec:elliptic-table-hecke}

Here we present a list of all possible Hecke modifications at a point
$p \in X$ of all possible rank 2 vector bundles on $X$, up to
tensoring with a line bundle.
We will parameterize Hecke modifications of a vector bundle $E$ at a
point $p$ in terms of lines $\ell_p \in \mathbbm{P}(E_p)$, as
described in Theorem \ref{theorem:single-isomorphism}.
Since we are always free to tensor a Hecke modification with a line
bundle, it suffices to consider vector bundles of nonnegative degree.

To construct the list, we will often use the following strategy.
By tensoring $E$ with a line bundle of sufficiently high degree if
necessary, we can assume without loss of generality that $E$ is
generated by global sections.
Consider a Hecke modification $\alpha:F \rightarrow E$ of $E$ at $p$
corresponding to a line $\ell_p := \im \alpha_p \in \mathbbm{P}(E_p)$.
Since we have assumed $E$ is generated by global sections, there
is a section $s$ of $E$ such that $s(p) \neq 0$ and $s(p) \in \ell_p$.
We then get a subbundle $\calO(\Div s) \rightarrow E$ and a
commutative diagram
\begin{eqnarray}
  \nonumber
  \begin{tikzcd}
    0 \arrow{r} &
    \calO(\Div s) \arrow{r} \arrow{d}{=}&
    F \arrow{r}\arrow{d}{\alpha} &
    L \otimes \calO(-p) \arrow{r}\arrow{d}{f} &
    0 \\
    0 \arrow{r} &
    \calO(\Div s) \arrow{r} &
    E \arrow{r} &
    L \arrow{r} &
    0,
  \end{tikzcd}
\end{eqnarray}
where $f$ is the unique (up to rescaling by a constant) nonzero
morphism $L \otimes \calO(-p) \rightarrow L$.
Thus $F$ is an extension of $L \otimes \calO(-p)$ by $\calO(\Div s)$,
and we can often use this information to determine $F$.

\subsubsection{Rank 2 bundles of degree greater than 1}

\begin{theorem}
\label{theorem:hecke-L-O}
Consider a bundle of the form
$L \oplus \calO$ for $L$ a line bundle of degree greater than 1
(unstable, instability degree $\deg L$).
The possible Hecke modifications are
  \begin{align}
    \nonumber
    L \oplus \calO \leftarrow
    \left\{
    \begin{array}{ll}
      L \oplus \calO(-p) &
      \quad \mbox{if $\ell_p = L_p$ (a bad line),} \\
      (L \otimes \calO(-p)) \oplus \calO &
      \quad \mbox{otherwise (a good line).}
    \end{array}
    \right.
  \end{align}
\end{theorem}

\begin{proof}
(1) The case $\ell_p = L_p$.
A Hecke modification
$\alpha:L \oplus \calO(-p) \rightarrow L \oplus \calO$ corresponding
to $\ell_p$ is
\begin{align}
  \nonumber
  \alpha = \left(\begin{array}{cc}
    1 & 0 \\
    0 & f
  \end{array}\right),
\end{align}
where $f$ is the unique (up to rescaling by a constant) nonzero
morphism $\calO(-p) \rightarrow \calO$.

(2) The case $\ell_p \neq L_p$.
Since $\deg L > 1$, we can choose a nonzero section $t$ of $L$
such that $t(p) \neq 0$.
Since $t$ is nonvanishing at $p$, we can choose a section
$s = (at, b)$ of $L \oplus \calO$ for
$a,b \in \Complex$ such that $s(p) \neq 0$ and $s(p) \in \ell_p$.
A Hecke modification
$\alpha:(L \otimes \calO(-p)) \oplus \calO \rightarrow L \oplus \calO$
corresponding to $\ell_p$ is
\begin{align}
  \nonumber
  \alpha = \left(\begin{array}{cc}
    f & at \\
    0 & b
  \end{array}\right),
\end{align}
where $f$ is the unique (up to rescaling by a constant) nonzero
morphism $L \otimes \calO(-p) \rightarrow L$.
\end{proof}

\subsubsection{Rank 2 bundles of degree 1}

\begin{theorem}
\label{theorem:hecke-Oq-O}
Consider the bundle
$\calO(q) \oplus \calO$ with $q \neq p$ (unstable, instability degree
1).
The possible Hecke modifications are
  \begin{align}
    \nonumber
    \calO(q) \oplus \calO \leftarrow
    \left\{
    \begin{array}{ll}
      \calO(q) \oplus \calO(-p) &
      \quad \mbox{if $\ell_p = \calO(q)_p$ (a bad line),} \\
      \calO(q-p) \oplus \calO &
      \quad \mbox{otherwise (a good line).}
    \end{array}
    \right.
  \end{align}
\end{theorem}

\begin{proof}
One can prove this result by using the fact that $\calO(q)$ has a
section $t$ such that $t(p) \neq 0$ and writing down explicit Hecke
modifications, as in the proof of Theorem \ref{theorem:hecke-L-O}.
\end{proof}

\begin{theorem}
\label{theorem:hecke-Op-O}
Consider the bundle
$\calO(p) \oplus \calO$ (unstable, instability degree 1).
The possible Hecke modifications are
\begin{align}
  \nonumber
      \calO(p) \oplus \calO \leftarrow
      \left\{
    \begin{array}{ll}
      \calO(p) \oplus \calO(-p) &
      \quad \mbox{if $\ell_p = \calO(p)_p$ (a bad line),} \\
      \calO \oplus \calO &
      \quad \mbox{if $\ell_p = \calO_p$ (a good line),} \\
      F_2 &
      \quad \mbox{otherwise (a good line).}
    \end{array}
    \right.
 \end{align}
\end{theorem}

\begin{proof}
(1) The case $\ell_p = \calO(p)_p$.
A Hecke modification
$\alpha:\calO(p) \oplus \calO(-p) \rightarrow
\calO(p) \oplus \calO$ is given by
\begin{align}
  \nonumber
  \alpha = \left(\begin{array}{cc}
    1 & 0 \\
    0 & f
  \end{array}\right),
\end{align}
where $f$ is the unique (up to rescaling by a constant) nonzero
morphism $\calO(-p) \rightarrow \calO$.

(2) The case $\ell_p = \calO_p$.
A Hecke modification
$\alpha:\calO \oplus \calO \rightarrow
\calO(p) \oplus \calO$ is given by
\begin{align}
  \nonumber
  \alpha = \left(\begin{array}{cc}
    t & 0 \\
    0 & 1
  \end{array}\right),
\end{align}
where $t$ is the unique (up to rescaling by a constant) nonzero
morphism $\calO \rightarrow \calO(p)$.

(3) The case $\ell_p \neq \calO(p)_p$ and $\ell_p \neq \calO_p$.
Pick a point $q \in X$ such that $q \neq p$.
Choose a nonzero section $t_0$ of $\calO(p+q)$ such that
$t_0(q) \neq 0$ and $t_0(p) \neq 0$.
Choose a nonzero section $t_1$ of $\calO(q)$.
Note that $\Div t_1 = q$.
Since $t_0(p) \neq 0$ and $t_1(p) \neq 0$,
we can define a section $s = (a t_0, b t_1)$ of
$\calO(p+q) \oplus \calO(q)$ for
$a, b \in \Complex$ such that $s(p) \neq 0$ and $s(p) \in \ell_p$.
Since $\ell_p \neq \calO(p)_p$ and $\ell_p \neq \calO_p$, it follows
that $a \neq 0$ and $b \neq 0$, thus $\Div s = 0$ and we obtain a
subbundle
$\calO(\Div s) = \calO \rightarrow \calO(p+q) \oplus \calO(q)$,
$1 \mapsto s$.
Thus we have a commutative diagram
\begin{eqnarray}
  \nonumber
  \begin{tikzcd}
    0 \arrow{r} &
    \calO \arrow{r} \arrow{d}{=}&
    F \arrow{r}\arrow{d}{\alpha} &
    \calO(2q) \arrow{r}\arrow{d} &
    0, \\
    0 \arrow{r} &
    \calO \arrow{r} &
    \calO(p+q) \oplus \calO(q) \arrow{r} &
    \calO(2q+p) \arrow{r} &
    0.
  \end{tikzcd}
\end{eqnarray}
The bundle $F$ cannot split, since there are no nonzero morphisms
$\calO(2q) \rightarrow
\calO(p+q) \oplus \calO(q)$,
hence $F$ is indecomposable.
Since $\det F = \calO(2q)$ it follows that
$F = F_2 \otimes \calO(q) \otimes L$ for a $2$-torsion line bundle
$L$.
We can compose $\alpha$ with projection onto the second summand
of $\calO(p+q) \oplus \calO(q)$ to obtain a nonzero morphism
$F \rightarrow \calO(q)$, so Lemma \ref{lemma:L-F2} implies
$L = \calO$ and $F = F_2 \otimes \calO(q)$.
\end{proof}

\begin{theorem}
\label{theorem:hecke-G2}
Consider the bundle $G_2(p)$ (stable, instability degree $-1$).
There is a canonical isomorphism
$\mathbbm{P}(G_2(p)_p) \rightarrow M^{ss}(X) \cong \CP^1$ given by
\begin{align}
  \nonumber
  \ell_p \mapsto [H(G_2(p),\ell_p)].
\end{align}
All lines of $G_2(p)$ are bad.
\end{theorem}

\begin{proof}
By Lemma \ref{lemma:g2-subbundle}, any degree 0 line bundle $L$ is a
subbundle of $G_2(p)$ via a unique (up to rescaling by a constant)
inclusion map $L \rightarrow G_2(p)$.
Thus we have a commutative diagram
\begin{eqnarray}
  \nonumber
  \begin{tikzcd}
    0 \arrow{r} &
    L \arrow{r} \arrow{d}{=}&
    F \arrow{r}\arrow{d}{\alpha} &
    L^{-1} \arrow{r}\arrow{d} &
    0, \\
    0 \arrow{r} &
    L \arrow{r} &
    G_2(p) \arrow{r} &
    L^{-1} \otimes \calO(p) \arrow{r} &
    0.
  \end{tikzcd}
\end{eqnarray}
Note that $\Ext^1(L^{-1},L) = H^0(L^{-2})$.
If $L^2 \neq \calO$ then $H^0(L^{-2}) = 0$, so $F$ splits, thus
$F = L \oplus L^{-1}$.

Now suppose $L^2 = \calO$.
We claim that $F$ is indecomposable; assume for contradiction that this
is not the case.
Then $F = L \oplus L$,
so $\alpha:F \rightarrow G_2(p)$ gives a map
$\calO \oplus \calO \rightarrow G_2(p) \otimes L$
that is an isomorphism away from $p$, so we obtain two linearly
independent sections of
$G_2(p) \otimes L$.
But by Lemma \ref{lemma:deg-indecomp} we have that
$h^0(G_2(p) \otimes L) = 1$, contradiction.
It follows that $F$ is indecomposable.
Since $\det F = \calO$, it follows that $F = F_2 \otimes M$ for
a $2$-torsion line bundle $M$.
Since we have a nonzero morphism $L \rightarrow F$,
Lemma \ref{lemma:L-F2} implies that $M = L$ and
$F = F_2 \otimes L$.

The above considerations show that we have a surjection
$\Jac(X) \rightarrow M^{ss}(X)$, $[L] \mapsto [F]$.
The vector bundle
$F$ is isomorphic to $H(G_2(p),\ell_p)$, where
$\ell_p \in \mathbbm{P}(G_2(p)_p)$ is the line corresponding to
$[G_2(p)_p \hmod{\alpha}{p} F] \in \calH^{tot}(X,G_2(p);p)$ under the
canonical isomorphism described in Theorem
\ref{theorem:single-isomorphism}, and we have a commutative diagram
\begin{eqnarray}
  \nonumber
  \begin{tikzcd}
    \Jac(X) \arrow{r}\arrow{d} &
    M^{ss}(X) \\
    \mathbbm{P}(G_2(p)_p). \arrow{ur}
  \end{tikzcd}
\end{eqnarray}
Here
$\Jac(X) \rightarrow \mathbbm{P}(G_2(p)_p)$ is given by
$[L] \mapsto L_p$ and
$\mathbbm{P}(G_2(p)_p) \rightarrow M^{ss}(X)$ is given by
$\ell_p \mapsto [H(G_2(p),\ell_p)]$.
Since $\Jac(X) \rightarrow M^{ss}(X)$ is surjective, we have that
$\Jac(X) \rightarrow \mathbbm{P}(G_2(p)_p)$ is surjective and
$\mathbbm{P}(G_2(p)_p) \rightarrow M^{ss}(X)$ is an isomorphism.
The surjectivity of $\Jac(X) \rightarrow \mathbbm{P}(G_2(p)_p)$
implies that all lines of $\mathbbm{P}(G_2(p)_p)$ are bad.
Since $G_2(p) = G_2(q) \otimes M$ for a suitable degree 0 line bundle
$M$, all lines of $G_2(p)$ are bad.
\end{proof}

\subsubsection{Rank 2 bundles of degree 0}

\begin{theorem}
\label{theorem:hecke-O-O}
Consider the bundle
$\calO \oplus \calO$ (strictly semistable, instability degree 0).
The possible Hecke modifications are
\begin{align}
  \nonumber
      \calO \oplus \calO \leftarrow
      \calO \oplus \calO(-p) \qquad
      \textup{for all $\ell_p$ (all lines are bad).}
  \end{align}
\end{theorem}
 
\begin{proof}
We can choose a section $s$ of $\calO \oplus \calO$ such that
$s(p) \neq 0$ and $s = \ell_p$.
We thus obtain a subbundle
$\calO \rightarrow \calO \oplus \calO$, $1 \mapsto s$ and a
commutative diagram
\begin{eqnarray}
  \nonumber
  \begin{tikzcd}
    0 \arrow{r} &
    \calO \arrow{r} \arrow{d}{=}&
    F \arrow{r}\arrow{d}{\alpha} &
    \calO(-p) \arrow{r}\arrow{d} &
    0, \\
    0 \arrow{r} &
    \calO \arrow{r} &
    \calO \oplus \calO \arrow{r} &
    \calO \arrow{r} &
    0.
  \end{tikzcd}
\end{eqnarray}
Since
$\Ext^1(\calO(-p),\calO) = H^0(\calO(-p)) = 0$,
we have that $F$ splits, thus
$F = \calO \oplus \calO(-p)$.
Alternatively, one can write down explicit Hecke modifications, as in
the proof of Theorem \ref{theorem:hecke-L-O}.
\end{proof}

\begin{theorem}
\label{theorem:hecke-L-Lm1}
Consider a bundle of the form $L \oplus L^{-1}$, where $L$ is a degree
0 line bundle such that $L^2 \neq \calO$ (strictly semistable,
instability degree 0).
The possible Hecke modifications are
\begin{align}
  \nonumber
  L \oplus L^{-1} \leftarrow
  \left\{
  \begin{array}{ll}
    L \oplus (L^{-1} \otimes \calO(-p)) &
    \quad \mbox{if $\ell_p = L_p$ (a bad line),} \\
    (L \otimes \calO(-p)) \oplus L^{-1} &
    \quad \mbox{if $\ell_p = (L^{-1})_p$ (a bad line),} \\
    G_2(p) \otimes \calO(-p) &
    \quad \mbox{otherwise (a good line).} \\
  \end{array}
  \right.
\end{align}
\end{theorem}

\begin{proof}
For $\ell_p = L_p$ or $\ell_p = (L^{-1})_p$, we can write down
explicit Hecke modifications, as in the proof of Theorem
\ref{theorem:hecke-L-O}.
So assume $\ell_p \neq L_p$ and $\ell_p \neq (L^{-1})_p$.
Choose a point $e \in X$ such that
$(L \oplus L^{-1}) \otimes \calO(e) =
\calO(q_1) \oplus \calO(q_2)$ for points $q_1, q_2 \in X$
distinct from $p$.
Since $L^2 \neq \calO$, it follows that $q_1 \neq q_2$.
Note that $q_1 + q_2 = 2e$.
Let $t_k$ be the unique (up to rescaling by a constant) nonzero
section of $\calO(q_k)$; note that $\Div t_k = q_k$.
We can define a section
$s = (a t_1, b t_2)$ of $\calO(q_1) \oplus \calO(q_2)$
for $a, b \in \Complex$ such that $s(p) \neq 0$ and $s(p) \in \ell_p$.
Since $\ell_p \neq L_p$ and $\ell_p \neq (L^{-1})_p$,
it follows that $a \neq 0$ and $b \neq 0$,
thus $\Div s = 0$.
We thus obtain a subbundle
$\calO(\Div s) = \calO \rightarrow \calO(q_1) \oplus \calO(q_2)$,
$1 \mapsto s$ and a commutative diagram
\begin{eqnarray}
  \nonumber
  \begin{tikzcd}
    0 \arrow{r} &
    \calO \arrow{r} \arrow{d}{=}&
    F \arrow{r}\arrow{d}{\alpha} &
    \calO(q_1 + q_2 - p) \arrow{r}\arrow{d} &
    0, \\
    0 \arrow{r} &
    \calO \arrow{r} &
    \calO(q_1) \oplus \calO(q_2) \arrow{r} &
    \calO(q_1 + q_2) \arrow{r} &
    0.
  \end{tikzcd}
\end{eqnarray}
There are no nonzero morphisms
$\calO(q_1 + q_2 - p) \rightarrow
\calO(q_1) \oplus \calO(q_2)$, so $F$ cannot split.
Since $\det F = \calO(q_1 + q_2 - p)$, we have that
$F = G_2(q_1 + q_2 - p) = G_2(2(e - p) + p) =
G_2(p) \otimes \calO(e-p)$.
\end{proof}

\begin{theorem}
\label{theorem:hecke-F2}
Consider the bundle $F_2$ (strictly semistable, instability degree 0).
The possible Hecke modifications are
\begin{align}
  \nonumber
  F_2 \leftarrow
  \left\{
  \begin{array}{ll}
    \calO \oplus \calO(-p) &
    \quad \mbox{if $\ell_p = \calO_p$ (a bad line),} \\
    G_2(p) \otimes \calO(-p) &
    \quad \mbox{otherwise (a good line),}
  \end{array}
    \right.
\end{align}
where $\calO \rightarrow F_2$ is the unique degree 0 line subbundle of
$F_2$.
\end{theorem}

\begin{proof}
(1) The case $\ell_p = \calO_p$.
We have a commutative diagram
\begin{eqnarray}
  \nonumber
  \begin{tikzcd}
    0 \arrow{r} &
    \calO \arrow{r} \arrow{d}{=}&
    F \arrow{r}\arrow{d}{\alpha} &
    \calO(-p) \arrow{r}\arrow{d} &
    0, \\
    0 \arrow{r} &
    \calO \arrow{r} &
    F_2 \arrow{r} &
    \calO \arrow{r} &
    0.
  \end{tikzcd}
\end{eqnarray}
Since $\Ext^1(\calO(-p),\calO) = H^0(\calO(-p)) = 0$, we have that $F$
splits, thus $F = \calO \oplus \calO(-p)$.

(2) The case $\ell_p \neq \calO_p$.
Pick a point $q \in X$ such that $q \neq p$.
Choose sections $t_0$ and $t_1$ of $F_2 \otimes \calO(q)$ as in
Lemma \ref{lemma:sections-F2}.
We can define a section
$s = a t_0 + b t_1$ of $F_2 \otimes \calO(q)$ for $a, b \in \Complex$
such that $s(p) \neq 0$ and $s(p) \in \ell_p$.
Since $\ell_p \neq \calO_p$ it follows that $a \neq 0$, thus
$\Div s = 0$.
We thus obtain a subbundle
$\calO(\Div s) = \calO \rightarrow F_2 \otimes \calO(q)$,
$1 \mapsto s$ and a commutative diagram
\begin{eqnarray}
  \nonumber
  \begin{tikzcd}
    0 \arrow{r} &
    \calO \arrow{r} \arrow{d}{=}&
    F \arrow{r}\arrow{d}{\alpha} &
    \calO(2q-p) \arrow{r}\arrow{d} &
    0, \\
    0 \arrow{r} &
    \calO \arrow{r} &
    F_2 \otimes \calO(q) \arrow{r} &
    \calO(2q) \arrow{r} &
    0.
  \end{tikzcd}
\end{eqnarray}
We claim that $F$ cannot split.
Assume for contradiction that $F$ splits, thus
$F = \calO \oplus \calO(2q-p)$.
Then we can precompose $\alpha$ with the inclusion
$\calO(2q-p) \rightarrow \calO \oplus \calO(2q-p)$ to obtain a nonzero
morphism $\calO(2q-p) \rightarrow F_2 \otimes \calO(q)$, contradicting
Lemma \ref{lemma:L-F2}.
Since $F$ does not split and $\det F = \calO(2q-p)$, we have that
$F = G_2(2q - p) = G_2(2(q-p)+p) = G_2(p) \otimes \calO(q - p)$.
\end{proof}

\subsubsection{Observations}

From this list, we make the following observations:

\begin{lemma}
\label{lemma:x-observations}
The following results hold for Hecke modifications of a rank 2 vector
bundle $E$ on an elliptic curve:
\begin{enumerate}
\item
  A Hecke modification of $E$ changes the instability degree by $\pm 1$.

\item
  Hecke modification of $E$ corresponding to a line
  $\ell_p \in \mathbbm{P}(E_p)$ changes the instability degree by $-1$
  if $\ell_p$ is a good line and $+1$ if $\ell_p$ is a bad line.

\item
  A generic Hecke modification of $E$ changes the instability degree
  by $-1$ unless $E$ has the minimum possible instability degree
  $-1$, in which case all Hecke modifications of $E$ change the
  instability degree by $+1$.
\end{enumerate}
\end{lemma}

\subsection{Moduli spaces $\calP_M^{tot}(X,m,n)$ and $\calP_M(X,m,n)$}
\label{sec:elliptic-moduli-space}

In Section \ref{sec:cp1-moduli-space} we defined a total space of
marked parabolic bundles
$\calP_M^{tot}(\CP^1,n) = \calP_M^{tot}(\CP^1,3,n)$ for rational
curves, and we showed that the Seidel--Smith space $\calY(S^2,2r)$
could be reinterpreted as the subspace
$\calP_M(\CP^1,2r) = \calP_M(\CP^1,3,2r)$ of
$\calP_M^{tot}(\CP^1,2r)$.
We now want to generalize the spaces $\calP_M^{tot}(\CP^1,n)$ and
$\calP_M(\CP^1,2r)$ to the case of an elliptic curve $X$.
There are obvious candidates: namely, the spaces
$\calP_M^{tot}(X,m,n)$ and $\calP_M(X,m,n)$ for some value of $m$,
which should be chosen to obtain the correct generalization.
One possibility is to use same value $m=3$ that we used for rational
curves.
But $m=1$ is also yields a reasonable generalization, as can be
understood from the following considerations.

In Appendix \ref{sec:vector-bundles} we define a
moduli space $M^{ss}(C)$ of semistable rank 2 vector bundles over a
curve $C$ with trivial determinant bundle, and in Appendix
\ref{sec:parabolic-bundles} we define
a moduli space $M^s(C,m)$ of stable rank 2 parabolic bundles
over a curve $C$ with trivial determinant bundle and $m$ marked
points.
Recall that for rational curves we chose $m=3$ marking lines because
we wanted $\calP_M^{tot}(\CP^1,m,n)$ to be isomorphic to the space of
parabolic bundles
$\calP^{tot}(\CP^1,\calO \oplus \calO,n)$ in which the underlying
vector bundle is $\calO \oplus \calO$, and
\begin{align}
  \nonumber
  \calP_M^{tot}(\CP^1,3,0) = M^s(\CP^1,3) = M^{ss}(\CP^1) =
  \{[\calO \oplus \calO]\}.
\end{align}
For an elliptic curve $X$, however, the corresponding spaces
$M^s(X,3)$ and $M^{ss}(X)$ are \emph{not} isomorphic: the space
$M^s(X,3)$ is a complex manifold of dimension 3, whereas $M^{ss}(X)$
is isomorphic to $\CP^1$.
Instead we have the following results, which can be viewed as
elliptic-curve analogs to Theorem \ref{theorem:Ms-cp1-3} and Corollary
\ref{cor:Ms-cp1-3-2} for rational curves:

\begin{theorem}
\label{theorem:elliptic-Ms-1}
The moduli space $M^s(X,1)$ consists of points $[E,\ell_q]$, where
$\ell_q$ is a good line and either
$E = F_2 \otimes L_i$ or
$E = L \oplus L^{-1}$ for $L$ a degree 0 line bundle such that
$L^2 \neq \calO$.
Given any two parabolic bundles of the form
$(E,\ell_q)$ and $(E,\ell_q')$ representing points of $M^s(X,1)$,
there is a unique (up to rescaling by a constant) automorphism
$\phi \in \Aut(E)$ such that $\phi(\ell_q) = \ell_q'$.
\end{theorem}

\begin{proof}
If $[E,\ell_q] \in M^s(X,1)$ then $E$ is semistable,
$\det E = \calO$, and $\ell_q$ is a good line.
Since $E$ is semistable and $\det E = \calO$, it must be
$L_i \oplus L_i$,
$F_2 \otimes L_i$, or
$L \oplus L^{-1}$ for $L$ a degree 0 line bundle such that
$L^2 \neq \calO$.
But Lemma \ref{lemma:Li-Li} states that $L_i \oplus L_i$ has no good
lines, so $E$ cannot be $L_i \oplus L_i$.
Lemmas \ref{lemma:L-Lm1} and \ref{lemma:F2} show that the remaining
two possibilities for $E$ do have good lines and also prove the
statement regarding unique automorphisms.
\end{proof}

\begin{corollary}
\label{corollary:map-bundle-elliptic}
The map $M^s(X,1) \rightarrow M^{ss}(X)$, $[E,\ell_q] \mapsto [E]$
is an isomorphism.
\end{corollary}

From these results, we see that there are \emph{two} natural
generalizations of $\calP_M^{tot}(\CP^1,n)$ to an elliptic curve.
The generalization of
$\calP_M^{tot}(\CP^1,3,0) = M^{ss}(\CP^1)$ is
\begin{align}
  \nonumber
  \calP_M^{tot}(X,1,0) = M^s(X,1) = M^{ss}(X),
\end{align}
which would lead us to choose $m=1$ marking lines.
The generalization of
$\calP_M^{tot}(\CP^1,3,0) = M^s(\CP^1,3)$ is
\begin{align}
  \nonumber
  \calP_M^{tot}(X,3,0) = M^s(X,3),
\end{align}
which would lead us to choose $m=3$ marking lines.
We will address the question of which of these values of $m$ yields
the correct generalization of the Seidel--Smith space in Section
\ref{sec:applications}.

From Theorem \ref{theorem:PMtot}, we have that $\calP_M^{tot}(X,1,n)$
is a $(\CP^1)^n$-bundle over $M^s(X,1) \cong \CP^1$.
We will show that this bundle is trivial.
To prove this result, we will use the marking line
of $\calP_M^{tot}(X,1,n)$ to canonically identify
$\mathbbm{P}(E_p)$ with $M^{ss}(X) \cong \CP^1$ for
$[E,\ell_{q_1},\ell_{p_1},\cdots,\ell_{p_n}] \in \calP_M^{tot}(X,1,n)$:

\begin{lemma}
\label{lemma:map-line-elliptic}
Fix a parabolic bundle $(E,\ell_q)$ such that
$[E,\ell_q] \in M^s(X,1)$, a point $p \in X$ such that $p \neq q$, and
a point $e \in X$ such that $p + q = 2e$.
There is a canonical isomorphism
$\mathbbm{P}(E_p) \rightarrow M^{ss}(X)$ given by
\begin{align}
  \nonumber
  \ell_p \mapsto [H(E,\ell_q,\ell_p) \otimes \calO(e)].
\end{align}
\end{lemma}

\begin{proof}
Theorem \ref{theorem:elliptic-Ms-1} implies that $\ell_q$ is a
good line and either
$E = F_2 \otimes L_i$ or
$E = L \oplus L^{-1}$ for $L$ a degree 0 line bundle such that
$L^2 \neq \calO$.
From Theorems \ref{theorem:hecke-L-Lm1} and \ref{theorem:hecke-F2}, it
follows that
\begin{align}
  \nonumber
  H(E,\ell_q) = G_2(q) \otimes \calO(-q) =
  G_2(p + 2(q-e)) \otimes \calO(-q) =
  G_2(p) \otimes \calO(-e).
\end{align}
The result now follows from Theorem \ref{theorem:hecke-G2}.
\end{proof}

Lemma \ref{lemma:map-line-elliptic} can be viewed as the
elliptic-curve analog to Lemma \ref{lemma:cp1-map-line} for
rational curves.
To perform calculations, it will be useful to explicitly evaluate the
map $\mathbbm{P}(E_p) \rightarrow M^{ss}(X)$ for bad lines
$\ell_p \in \mathbbm{P}(E_p)$.
In general, we prove:

\begin{lemma}
\label{lemma:map-line-elliptic-bad}
Fix distinct points $p,q \in X$ and a point $e \in X$ such that
$p + q = 2e$.
If $\ell_q$ is a good line, then
\begin{align}
  \nonumber
  H(L \oplus L^{-1}, \ell_q, L_p) \otimes \calO(e) &=
  M \oplus M^{-1},\,
  \textup{where $M = L \otimes \calO(p-e) = L \otimes \calO(e-q)$},
  \\
  \nonumber
  H(L \oplus L^{-1}, \ell_q,(L^{-1})_p ) \otimes \calO(e) &=
  M \oplus M^{-1},\,
  \textup{where $M = L \otimes \calO(q-e) = L \otimes \calO(e-p)$},
  \\
  \nonumber
  H(L \oplus L^{-1},L_q,(L^{-1})_p) \otimes \calO(e) &=
  M \oplus M^{-1},\,
  \textup{where $M = L \otimes \calO(q-e) = L \otimes \calO(e-p)$},
  \\
  \nonumber
  H(F_2, \ell_q, \calO_p) \otimes \calO(e) &= M \oplus M^{-1},\,
  \textup{where $M = \calO(p-e) = \calO(e-q)$}.
\end{align}
\end{lemma}

\begin{proof}
These results are straightforward calculations using the list of
Hecke modifications in Section \ref{sec:elliptic-table-hecke}.
As an example, we will prove the result involving $F_2$.
From Theorem \ref{theorem:hecke-F2} we have that
\begin{align}
  \nonumber
  H(F_2, \calO_p) = \calO \oplus \calO(-p).
\end{align}
Since $\ell_q$ is a good line, the bundle
$H(F_2, \calO_p,\ell_q) = H(F_2,\ell_q,\calO_p)$
must be semistable, and the result now follows from Theorem
\ref{theorem:hecke-Oq-O}.
\end{proof}

\begin{theorem}
\label{theorem:Pmtot-iso-elliptic}
There is a canonical isomorphism
$h:\calP_M^{tot}(X,1,n) \rightarrow (M^{ss}(X))^{n+1}$.
\end{theorem}

\begin{proof}
Define $h_0:\calP_M^{tot}(X,1,n) \rightarrow M^{ss}(X)$ by
\begin{align}
  \nonumber
  h_0([E,\ell_{q_1}, \ell_{p_1}, \cdots, \ell_{p_n}]) &= [E].
\end{align}
For $i=1,\cdots,n$, choose a point $e_i \in X$ such that
$q_1 + p_i = 2e_i$ and define
$h_i:\calP_M^{tot}(X,1,n) \rightarrow M^{ss}(X)$ by
\begin{align}
  \nonumber
  h_i([E,\ell_{q_1}, \ell_{p_1}, \cdots, \ell_{p_n}]) &=
  [H(E,\ell_{q_1},\ell_{p_i}) \otimes \calO(e_i)].
\end{align}
Then $h := (h_0, h_1, \cdots, h_n)$ is an isomorphism by
Theorem \ref{theorem:PMtot},
Corollary \ref{corollary:map-bundle-elliptic}, and Lemma
\ref{lemma:map-line-elliptic}.
\end{proof}

For $n=1$, the isomorphism
$h:\calP_M^{tot}(X,1,n) \rightarrow (M^{ss}(X))^{n+1}$ appears to be
closely related to an isomorphism
$M^{ss}(X,2) \rightarrow (\CP^1)^2$ defined in \cite{Vargas},
and our definition of $h$ was motivated by this isomorphism.

\subsection{Embedding $\calP_M(X,m,n) \rightarrow M^s(X,m+n)$}
\label{sec:elliptic-hecke-embedding}

We will now describe a canonical open embedding of the space
$\calP_M(X,m,n)$ into the space of stable
parabolic bundles $M^s(X,m+n)$.
We first need two Lemmas:

\begin{lemma}
\label{lemma:elliptic-embedding-a}
Let $(E,\ell_{p_1}, \cdots, \ell_{p_n})$ be a parabolic bundle over an
elliptic curve $X$ such that $E$ is semistable.
If the lines $\ell_{p_1}, \cdots, \ell_{p_n}$ are bad in the same
direction then $H(E,\ell_{p_1},\cdots,\ell_{p_n})$ has instability
degree $n$.
\end{lemma}

\begin{proof}
Up to tensoring with a line bundle, the bundle $E$ has one of three
forms:

(1) $E = \calO \oplus \calO$.
Since $\ell_{p_1}, \cdots, \ell_{p_n}$ are bad in the same direction,
we have that $\ell_{p_1} = \cdots = \ell_{p_n}$ under a global
trivialization of $E$ in which all the fibers are identified with
$\Complex^2$.
A sequence of Hecke modifications with
$\ell_{p_1} = \cdots = \ell_{p_n}$ is given by
\begin{align}
  \nonumber
  \calO \oplus \calO
  \hmod{\alpha_1}{p_1}
  \calO \oplus \calO(-p_1)
  \hmod{\alpha_2}{p_2} \cdots
  \hmod{\alpha_n}{p_n}
  \calO \oplus \calO(-p_1 - \cdots -p_n).
\end{align}
Here
$\calO \oplus \calO \hmod{\alpha_1}{p_1} \calO \oplus \calO(-p_1)$
is a Hecke modification corresponding to $\ell_{p_1}$, and for
$i=2,\cdots,n$ we define
\begin{align}
  \nonumber
  \alpha_i = \left(\begin{array}{cc}
    1 & 0 \\
    0 & f_i\\
    \end{array}\right),
\end{align}
where $f_i$ is the unique (up to rescaling by a constant) morphism
from $\calO(-p_1-\cdots-p_i)$ to $\calO(-p_1-\cdots-p_{i-1})$.
Thus
$H(\calO \oplus \calO,\ell_{p_1},\cdots,\ell_{p_n}) =
\calO \oplus \calO(-p_1 - \cdots - p_n)$ has instability degree $n$.

(2) $E = F_2$.
Then $\ell_{p_i} = \calO_{p_i}$ for $i = 1, \cdots, n$.
A sequence of Hecke modifications with
$\ell_{p_i} = \calO_{p_i}$ for $i = 1, \cdots, n$ is given by
\begin{align}
  \nonumber
  F_2
  \hmod{\alpha_1}{p_1}
  \calO \oplus \calO(-p_1)
  \hmod{\alpha_2}{p_2} \cdots
  \hmod{\alpha_n}{p_n}
  \calO \oplus \calO(-p_1 - \cdots -p_n),
\end{align}
where
$F_2 \hmod{\alpha_1}{p_1} \calO \oplus \calO(-p_1)$
is a Hecke modification corresponding to $\ell_{p_1} = \calO_{p_1}$
and $\alpha_i$ is as above for $i=1,\cdots, n$.
Thus
$H(F_2,\calO_{p_1},\cdots,\calO_{p_n}) =
\calO \oplus \calO(-p_1 - \cdots - p_n)$ has instability degree $n$.

(3) $E = L \oplus L^{-1}$ for a degree 0 line bundle $L$ such that
$L^2 \neq \calO$.
Then either
$\ell_{p_i} = L_{p_i}$ for $i = 1, \cdots, n$ or
$\ell_{p_i} = (L^{-1})_{p_i}$ for $i = 1, \cdots, n$.
A sequence of Hecke modifications with
$\ell_{p_i} = L_{p_i}$ for $i = 1, \cdots, n$ is given by
\begin{align}
  \nonumber
  L \oplus L^{-1}
  \hmod{\alpha_1}{p_1}
  L \oplus (L^{-1} \otimes \calO(-p_1))
  \hmod{\alpha_2}{p_2} \cdots
  \hmod{\alpha_n}{p_n}
  L \oplus (L^{-1} \otimes \calO(-p_1 - \cdots - p_n)),
\end{align}
where
\begin{align}
  \nonumber
  \alpha_i = \left(\begin{array}{cc}
    1 & 0 \\
    0 & 1 \otimes f_i\\
    \end{array}\right)
\end{align}
and $f_i$ is as above.
Thus
$H(L \oplus L^{-1},L_{p_1},\cdots,L_{p_n}) =
L \oplus (L^{-1}\otimes \calO(-p_1 - \cdots - p_n))$
has instability degree $n$.
We can write down a similar sequence of Hecke modifications to show
that
$H(L \oplus L^{-1}, (L^{-1})_{p_1},\cdots, (L^{-1})_{p_n}) =
(L \otimes \calO(-p_1 - \cdots - p_n)) \oplus L^{-1}$ has instability
degree $n$.
\end{proof}

Using Lemma \ref{lemma:elliptic-embedding-a} in place of Lemma
\ref{lemma:cp1-embedding-a}, the proofs of
Lemma \ref{lemma:cp1-embedding-b} and
Theorem \ref{theorem:cp1-hecke-embedding} for rational curves carry
over to the case of elliptic curves.
We thus obtain:

\begin{lemma}
\label{lemma:elliptic-embedding-b}
Let $(E,\ell_{p_1}, \cdots, \ell_{p_n})$ be a parabolic bundle over an
elliptic curve $X$ such that $E$ is semistable.
If
$H(E,\ell_{p_1},\cdots, \ell_{p_n})$ is semistable
$(E,\ell_{p_1},\cdots, \ell_{p_n})$ is semistable.
\end{lemma}

\begin{theorem}
\label{theorem:elliptic-hecke-embedding}
There is a canonical open embedding
$\calP_M(X,m,n) \rightarrow M^s(X,m+n)$.
\end{theorem}

\subsection{Examples}

Here we compute the space $\calP_M(X,1,n)$ for $n=0,1,2$.
We first make some definitions:

\begin{definition}
The \emph{Abel-Jacobi isomorphism} $X \rightarrow \Jac(X)$ is given by
$p \mapsto [\calO(p-e)]$ for a choice of basepoint $e \in X$.
\end{definition}

\begin{definition}
We define a map $\pi:\Jac(X) \rightarrow M^{ss}(X)$,
$[L] \mapsto [L \oplus L^{-1}]$.
\end{definition}

Note that $\pi$ is surjective and $\pi(L) = \pi(L^{-1})$, so
$\pi:\Jac(X) \cong X \rightarrow M^{ss}(X) \cong \CP^1$ is a 2:1
branched cover with four branch points $[L_i \oplus L_i]$
corresponding to the four 2-torsion line bundles $L_i$.

\begin{definition}
Given a degree 0 divisor $D$ on an elliptic curve $X$, define the
\emph{translation map}
$\tau_D:\Jac(X) \rightarrow \Jac(X)$,
$[L] \mapsto [L \otimes \calO(D)]$.
\end{definition}

\subsubsection{Calculate $\calP_M(X,1,0)$}

We have that
\begin{align}
  \nonumber
  \calP_M(X,1,0) = \calP_M^{tot}(X,1,0) = M^s(X,1) = M^{ss}(X) =
  \CP^1.
\end{align}
Note that the embedding
$\calP_M(X,1,0) \rightarrow M^s(X,1)$ defined in Theorem
\ref{theorem:elliptic-hecke-embedding} is an isomorphism.

\subsubsection{Calculate $\calP_M(X,1,1)$}
\label{sec:pm-x-1-1}

\begin{theorem}
\label{theorem:pm-x-1-1}
The map $g:\Jac(X) \rightarrow (M^{ss}(X))^2$,
$g = (\pi,\, \pi \circ \tau_{p_1-e_1})$ is injective and has image
the complement of $h(\calP_M(X,1,1))$,
where $h:\calP_M^{tot}(X,1,1) \rightarrow (M^{ss}(X))^2$ is the
isomorphism described in Theorem \ref{theorem:Pmtot-iso-elliptic}.
\end{theorem}

\begin{proof}
First we show that $g$ has image the complement of $h(\calP_M(X,1))$
in $(M^{ss}(X))^2$
Take a point $[E,\ell_{q_1},\ell_{p_1}] \in \calP_M^{tot}(X,1,1)$.
From Theorems
\ref{theorem:hecke-L-Lm1},
\ref{theorem:hecke-F2}, and
\ref{theorem:elliptic-Ms-1},
it follows that $H(E,\ell_{p_1}) = G_2(p_1) \otimes \calO(-p_1)$ is
stable if $\ell_{p_1}$ is a good line, and
$H(E,\ell_{p_1})$ is unstable if $\ell_{p_1}$ is a bad line.
So the complement of $\calP_M(X,1,1)$ in
$\calP_M^{tot}(X,1,1)$ consists of isomorphism classes
$[E,\ell_{q_1},\ell_{p_1}]$ such that $\ell_{p_1}$ is a bad line, and is
thus given by the union of the sets
\begin{align}
  \nonumber
  S_1 &= \{[L \oplus L^{-1},\ell_{q_1},L_{p_1}] \mid
    [L] \in \Jac(X),\,L^2 \neq \calO\}, \\
  \nonumber
  S_2 &= \{[L \oplus L^{-1},\ell_{q_1},(L^{-1})_{p_1}] \mid
    [L] \in \Jac(X),\,L^2 \neq \calO\}, \\
  \nonumber
  S_3 &= \{[F_2 \otimes L_i, \ell_{q_1}, (L_i)_{p_1}] \mid
  i=1,2,3,4\},
\end{align}
where in each case $\ell_{q_1}$ is a good line.
From Lemma \ref{lemma:map-line-elliptic-bad}, it follows that
the complement of $h(\calP_M(X,1,1))$ in $(M^{ss}(X))^2$ is
given by the union of the sets
\begin{align}
  \nonumber
  h(S_1) &=
  \{(\pi([L]),\,(\pi \circ \tau_{p_1 - e_1})([L])) \mid
  [L] \in \Jac(X),\,
  L^2 \neq \calO\}, \\
  \nonumber
  h(S_2) &=
  \{(\pi([L]),\,(\pi \circ \tau_{e_1 - p_1})([L])) \mid
  [L] \in \Jac(X),\,
  L^2 \neq \calO\}, \\
  \nonumber
  h(S_3) &=
  \{(\pi([L_i]),\,(\pi \circ \tau_{p_1 - e_1})([L_i])) \mid
  i=1,2,3,4\}.
\end{align}
Note that
\begin{align}
  \nonumber
  (\pi([L]),\,(\pi \circ \tau_{e_1 - p_1})([L])) =
  (\pi([L^{-1}]),\,(\pi \circ \tau_{p_1 - e_1})([L^{-1}])),
\end{align}
so $h(S_1) = h(S_2)$, and we have that
\begin{align}
  \nonumber
  h(S_1) \cup h(S_2) \cup h(S_3) =
  \{(\pi([L]),\,(\pi \circ \tau_{p_1 - e_1})([L])) \mid
  [L] \in \Jac(X)\} = \im g.
\end{align}
So the image of $g$ is the complement of $h(\calP_M(X,1,1))$ in
$(M^{ss}(X))^2$

Next we show that $g$ is injective.
If $g(L) = g(L')$, then projection onto the first factor of
$(M^{ss}(X))^2$ gives $\pi(L) = \pi(L')$, hence either
$L' = L$ or $L' = L^{-1}$.
Suppose $L' = L^{-1}$.
Then projection onto the second factor of $(M^{ss}(X))^2$ gives
$\pi(L \otimes \calO(p_1-e_1)) = \pi(L^{-1} \otimes \calO(p_1-e_1))$,
hence either
$L \otimes \calO(p_1-e_1) = L^{-1} \otimes \calO(p_1-e_1)$ or
$L \otimes \calO(p_1-e_1) = L \otimes \calO(e_1-p_1)$.
The first case implies $L = L^{-1}$.
The second case implies
$2p_1 = 2e_1$, but we chose $e_1$ such that $p_1 + q_1 = 2e_1$, hence
$p_1 = q_1$, contradiction.
Thus $L' = L$, so $g$ is injective.
\end{proof}

If we use the Abel-Jacobi isomorphism to identify $X$ and $\Jac(X)$,
the (canonical) isomorphism
$h:\calP_M^{tot}(X,1,1) \rightarrow (M^{ss}(X))^2$ to identify
$\calP_M^{tot}(X,1,1)$ and $(M^{ss}(X))^2$, and the (noncanonical)
isomorphism $M^{ss}(X) \cong \CP^1$ to identify $M^{ss}(X)$ and
$\CP^1$, we find that
\begin{align}
  \nonumber
  \calP_M(X,1,1) &= (\CP^1)^2 - g(X).
\end{align}

\begin{remark}
Using results from the proof of Theorem \ref{theorem:pm-x-1-1}, it
is straightforward to show that
\begin{align}
  \nonumber
  M^{ss}(X,2) &= \calP_M^{tot}(X,1,1) = (\CP^1)^2, &
  M^s(X,2) &= \calP_M(X,1,1) = (\CP^1)^2 - g(X).
\end{align}
These calculations reproduce the results of \cite{Vargas} for
$M^{ss}(X,2)$ and $M^s(X,2)$.
\end{remark}

\subsubsection{Calculate $\calP_M(X,1,2)$}
\label{sec:pm-x-1-2}

The same method that we used to prove Theorem \ref{theorem:pm-x-1-1}
can be used to calculate $\calP_M(X,1,2)$:

\begin{theorem}
\label{theorem:pm-x-1-2}
The map $f:\Jac(X) \rightarrow (M^{ss}(X))^3$,
$f = (\pi,\, \pi \circ \tau_{p_1-e_1},\, \pi \circ \tau_{p_2 - e_2})$ is
injective and has image the complement of $h(\calP_M(X,1,2))$,
where $h:\calP_M^{tot}(X,1,2) \rightarrow (M^{ss}(X))^3$ is the
isomorphism described in Theorem \ref{theorem:Pmtot-iso-elliptic}.
\end{theorem}

If we use the Abel-Jacobi isomorphism to identify $X$ and $\Jac(X)$,
the (canonical) isomorphism
$h:\calP_M^{tot}(X,1,2) \rightarrow (M^{ss}(X))^3$ to identify
$\calP_M^{tot}(X,1,2)$ and $(M^{ss}(X))^3$, and the (noncanonical)
isomorphism $M^{ss}(X) \cong \CP^1$ to identify $M^{ss}(X)$ and
$\CP^1$, we find that
\begin{align}
  \nonumber
  \calP_M(X,1,2) &= (\CP^1)^3 - f(X).
\end{align}

\section{Possible applications to topology}
\label{sec:applications}

Here we briefly outline some possible applications of our results to
topology.
We have proposed complex manifolds $\calP_M(X,1,2r)$ and
$\calP_M(X,3,2r)$ as candidates for a space $\calY(T^2,2r)$ that
generalizes the Seidel--Smith space $\calY(S^2,2r)$ and that could
potentially be used to construct symplectic Khovanov homology for lens
spaces.
The following tasks remain to be done to complete the construction:
\begin{enumerate}
\item
We need to define a suitable symplectic form on $\calP_M(X,m,2r)$.
One possibility is to pull back the canonical symplectic form on
$M^s(X,2r+m)$ using the open embedding
$\calP_M(X,m,2r) \rightarrow M^s(X,2r+m)$.

\item
We need to find a suitable action of the mapping class group
$\MCG_{2r}(T^2)$ on $\calP_M(X,m,2r)$ that is defined up to
Hamiltonian isotopy.
Such an action might be obtained via symplectic monodromy by viewing
$\calP_M(X,m,2r)$ as the fiber of a larger space that fibers over
the moduli space of genus 1 curves with marked points.
Such an approach would be analogous to the way Seidel and Smith obtain
an action of the braid group on the Seidel--Smith space via monodromy
around loops in the configuration space \cite{Seidel}, and similar
methods are used to define mapping class group actions for
constructing Reshetikhin-Turaev-Witten invariants \cite{Bakalov}.

\item
We need to define suitable Lagrangians $L_r$ in $\calP_M(X,m,2r)$
corresponding to $r$ unknotted arcs in a solid torus.
For $\calP_M(X,1,2r)$ we would expect $L_r$ to be homeomorphic to
$S^1 \times (S^2)^r$, and for
$\calP_M(X,3,2r)$ we would expect $L_r$ to be homeomorphic to
$S^3 \times (S^2)^r$.
Perhaps such Lagrangians can be constructed in a manner analogous to
Seidel--Smith by viewing $\calP_M(X,m,2r)$ as the fiber of a larger
space that fibers over the configuration space $\Conf_{2r}(X)$ of $2r$
unordered points in $X$ and looking for vanishing cycles as points are
successively brought together in pairs.

\item
We need to prove that the Lagrangian Floer homology of a knot $K$ in a
lens space is $Y$ invariant under different Heegaard splittings of
$(Y,K)$ into solid tori.

\item
We need to verify that our construction of symplectic Khovanov
homology reproduces ordinary Khovanov homology for the case of knots
in $S^3$.
\end{enumerate}

Several of our results appear to be related to a possible connection
between Khovanov homology and symplectic instanton homology.
Roughly speaking, symplectic instanton homology is defined as follows.
Given a knot $K$ in a 3-manifold $Y$, one Heegaard-splits $(Y,K)$
along a Heegaard surface $\Sigma$ to obtain handlebodies $U_1$ and
$U_2$.
Each handlebody $U_i$ contains a portion of the knot
$A_i := U_i \cap K$ consisting of $r$ arcs that pairwise
connect points $p_1, \cdots, p_{2r}$ in $\Sigma$.
To the marked surface $(\Sigma,p_1, \cdots, p_{2r})$ one associates a
character variety $R(\Sigma,2r)$, which has the structure of a
symplectic manifold, and to the handlebody pairs $(U_i,A_i)$ one
associates Lagrangians $L_i \subset R(\Sigma,2r)$.
The symplectic instanton homology of $(Y,K)$ is then defined to be the
Lagrangian Floer homology of the pair of Lagrangians $(L_1,L_2)$.

In fact, there are several technical difficulties that must be
overcome in order to get a well-defined homology theory.
For example, one needs to introduce a \emph{framing} in order to
eliminate singularities in the character variety $R(\Sigma,2r)$.
One way to introduce a framing is by replacing the knot $K$ with
$K \cup \Theta$, where $\Theta$ is the theta graph shown in Figure
\ref{fig:graphs}(a); this approach is described in \cite{Horton}.
We Heegaard-split $(Y,K \cup \Theta)$ along a Heegaard surface
$\Sigma$ that is chosen to transversely
intersect each edge $e_i$ of the theta graph in a single point $q_i$.
The marked Heegaard surface is now
$(\Sigma,q_1,q_2,q_3,p_1,\cdots,p_{2r})$, corresponding to the
character variety $R(\Sigma,2r+3)$, and the handlebody pairs are now
$(U_i,A_i \cup \epsilon_i)$, where $\epsilon_i$ is the epsilon
graph shown in Figure \ref{fig:graphs}(b).
The character variety $R(\Sigma,2r+3)$ has the structure of a
symplectic manifold that is symplectomorphic to the moduli space of
stable parabolic bundles $M^s(C,2r+3)$, where $C$ is any complex curve
homeomorphic to $\Sigma$.
(The space $M^s(C,2r+3)$ has a canonical symplectic form.)

\begin{figure}
  \centering
  \includegraphics[scale=0.5]{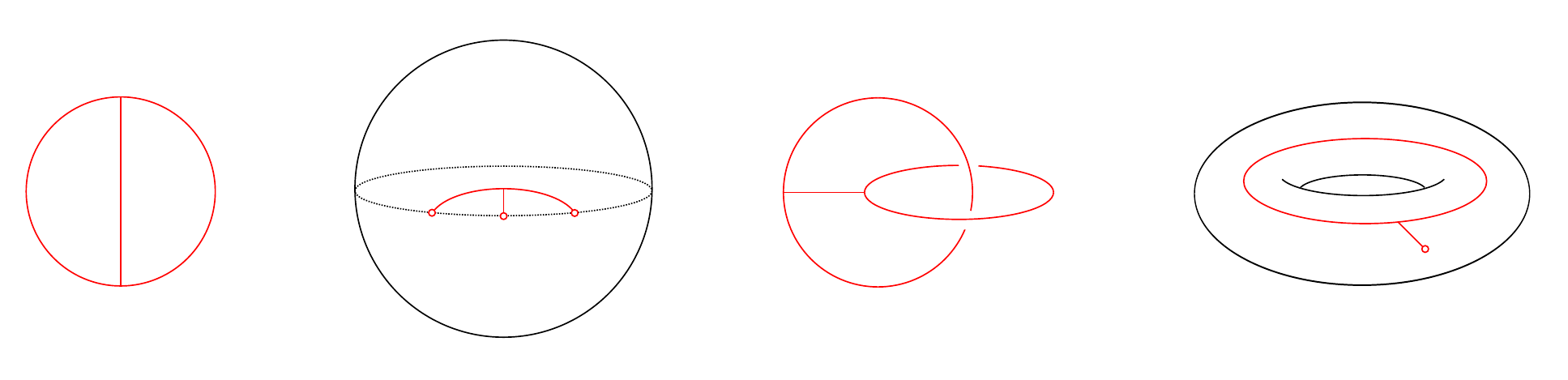}
  \caption{
    \label{fig:graphs}
    (a) The graph $\Theta$.
    (b) The graph $\epsilon$ in $B^3$.
    (c) The graph $D_p \subset S^3$ for $p=1$.
    (d) The graph $\sigma$ in $S^1 \times D^2$.
  }
\end{figure}

Symplectic instanton homology can be viewed as a symplectic
replacement for singular instanton homology, a knot homology theory
defined using gauge theory, and the two theories are conjectured to be
isomorphic.
This is an example of an Atiyah-Floer conjecture; such conjectures broadly
relate Floer-theoretic invariants defined using gauge theory to
corresponding invariants defined using symplectic topology.
Kronheimer and Mrowka constructed a spectral sequence from
Khovanov homology to singular instanton homology \cite{Kronheimer-3},
and the embedding $\calP_M(\CP^1,2r) \rightarrow M^s(\CP^1,2r+3)$
described in Theorem \ref{sec:cp1-hecke-embedding} suggests that it
may be possible to construct an analogous spectral sequence from
symplectic Khovanov homology to symplectic instanton homology.
(This idea for constructing a spectral sequence was suggested to the
author by Ivan Smith and Chris Woodward.)
If so, perhaps the fact that we have an embedding
$\calP_M(X,3,2r) \rightarrow M^s(X,2r+3)$, as described in Theorem
\ref{sec:elliptic-hecke-embedding}, is evidence that the correct
generalization of the Seidel--Smith space is $\calP_M(X,3,2r)$.
Indeed, a calculation of the Lagrangian intersection $L_1 \cap L_2$ in
the traceless character variety $R(T^2,3)$ for $(S^3,\Theta)$ yields a
single point, a space whose cohomology is the correct Khovanov
homology for the empty knot, and
calculations of the Lagrangian intersections in the traceless
character variety $R(T^2,5)$ for $(S^3,\textup{unknot} \cup \Theta)$
and $(S^3,\textup{trefoil} \cup \Theta)$
yield $S^2$ and $\RP^3 \amalg S^2$, spaces whose cohomology gives the
correct Khovanov homology for the unknot and trefoil.
Based on these speculations, we make the following conjectures:

\begin{conjecture}
The space $\calP_M(C,3,2r)$ is the correct generalization of the
Seidel--Smith space $\calY(S^2,2r)$ to a curve $C$ of arbitrary genus.
\end{conjecture}

\begin{conjecture}
Given a curve $C$ of arbitrary genus, there is a canonical open
embedding $\calP_M(C,m,n) \rightarrow M^s(C,m+n)$.
\end{conjecture}

On the other hand, perhaps the embedding
$\calP_M(X,1,2r) \rightarrow M^s(X,2r+1)$ described in Theorem
\ref{sec:elliptic-hecke-embedding} is related to a spectral
sequence from a Khovanov-like knot homology theory to symplectic
instanton homology defined with a novel framing.
Rather than using a theta graph, perhaps for the lens space $L(p,q)$
one could introduce a framing specific to that lens space by
using a $p$-linked dumbbell graph $D_p$, as shown in Figure
\ref{fig:graphs}(c) for the case $p=1$.
There is a unique edge $e_1$ of the dumbbell graph that connects the
two vertices, and one can choose a Heegaard surface $\Sigma$ that
transversely intersects $e_1$ in a single point $q_1$.
The marked Heegaard surface is now
$(\Sigma,q_1,p_1,\cdots,p_{2r})$, corresponding to the character
variety $R(\Sigma,2r+1)$, and the handlebody pairs $(U_i,A_i)$ are now
$(U_i,A_i \cup \sigma_i)$, where $\sigma_i$ is the sigma graph
shown in Figure \ref{fig:graphs}(d).
The character variety $R(\Sigma,2r+1)$ has the structure of a
symplectic manifold that is symplectomorphic to the moduli space of
stable parabolic bundles $M^s(X,2r+1)$, which is the codomain
of the embedding $\calP_M(X,1,2r) \rightarrow M^s(X,2r+1)$.

\begin{appendix}

\addtocontents{toc}{\protect\setcounter{tocdepth}{1}}

\section{Vector bundles}
\label{sec:vector-bundles}

Here we briefly review some results on holomorphic vector bundles and
their moduli spaces that we will use throughout the paper.
Some useful references on vector bundles are
\cite{LePotier,Schaffhauser,Teixidor,Tu}.

\begin{definition}
The {\em slope} of a holomorphic vector bundle $E$ over a curve $C$ is
$\slope E := (\deg E)/(\rank E) \in \Rats$.
\end{definition}

\begin{definition}
A holomorphic vector bundle $E$ over a curve $C$ is
{\em stable} if $\slope F < \slope E$ for any proper
subbundle $F \subset E$,
{\em semistable} if $\slope F \leq \slope E$ for any proper
subbundle $F \subset E$,
\emph{strictly semistable} if it is semistable but not stable, and
{\em unstable} if there is a proper subbundle $F \subset E$ such that
$\slope F > \slope E$.
\end{definition}

If $E$ is a stable vector bundle, then $\Aut(E) = \Complex^\times$
consists only of trivial automorphisms that scale the fibers by a
constant factor.

\begin{definition}
\label{def:jordan-holder-filtration}
Given a semistable vector bundle $E$, a
{\em Jordan-H\"{o}lder filtration} of $E$ is a filtration
\begin{align}
  \nonumber
  F_0 = 0 \subset F_1 \subset F_2 \subset \cdots \subset F_n = E
\end{align}
of $E$ by subbundles $F_i \subset E$ for $i=0, \cdots, n$ such that
the composition factors $F_i/F_{i-1}$ are stable and
$\slope F_i/F_{i-1} = \slope E$ for $i=1, \cdots, n$.
\end{definition}

Every semistable vector bundle $E$ admits a
Jordan-H\"{o}lder filtration.
The filtration is not unique, but the composition factors
$F_i/F_{i-1}$ for $i=1,\cdots,n$ are independent (up to
permutation) of the choice of filtration.

\begin{definition}
Given a semistable holomorphic vector bundle $E$ over a curve $C$, the
{\em associated graded} vector bundle $\gr E$ is defined to be
\begin{align}
  \nonumber
  \gr E = \bigoplus_{i=1}^n F_i/F_{i-1},
\end{align}
where $F_0 = 0 \subset F_1 \subset \cdots \subset F_n = E$ is a
Jordan-H\"{o}lder filtration of $E$.
\end{definition}

The bundle
$\gr E$ is independent (up to isomorphism) of the choice of
filtration, and $\slope(\gr E) = \slope E$.

\begin{definition}
Two semistable vector bundles are said to be
{\em $S$-equivalent} if their associated graded bundles are
isomorphic.
\end{definition}

\begin{example}
In Section
\ref{sec:elliptic-vector-bundles}
we define a strictly semistable rank 2 vector bundle $F_2$ and a
stable rank 2 vector bundle $G_2(p)$ over an elliptic curve $X$.
A Jordan-H\"{o}lder filtration of $F_2$ is $\calO \subset F_2$,
and the associated graded bundle is
$\gr F_2 = \calO \oplus \calO$.
It follows that $F_2$ and $\calO \oplus \calO$ are $S$-equivalent.
A Jordan-H\"{o}lder filtration of $G_2(p)$ is just $G_2(p)$, and the
associated graded bundle is $\gr G_2(p) = G_2(p)$.
\end{example}

Isomorphic bundles are $S$-equivalent.
For rational curves, $S$-equivalent bundles are isomorphic, but this
is not true in general.
For example, on an elliptic curve the bundles
$F_2$ and $\calO \oplus \calO$ are $S$-equivalent but not
isomorphic.

\begin{definition}
We define $M^{ss}(C)$ (respectively $M^s(C)$) to be the moduli space
of semistable (respectively stable) rank 2 holomorphic vector bundles
over curve $C$ with trivial determinant bundle, mod $S$-equivalence.
This space is defined in \cite{Seshadri}; see also \cite{Mukai}.
\end{definition}

\begin{remark}
An alternative way of interpreting $M^{ss}(C)$ is as the
space of flat $SU(2)$-connections on a trivial rank 2 complex vector
bundle $E \rightarrow C$, mod gauge transformations.
Yet another way of interpreting the space $M^{ss}(C)$ is as the
character variety $R(C)$ of conjugacy classes of group
homomorphisms $\pi_1(C) \rightarrow SU(2)$.
We will not use these interpretations here.
\end{remark}

The moduli space $M^s(C)$ has the structure of a complex manifold of
dimension $3(g-1)$, where $g$ is the genus of the curve $C$.
The space $M^s(C)$ carries a canonical symplectic form, which is
obtained by interpreting $M^s(C)$ as a Hamiltonian reduction of a
space of $SU(2)$-connections.

\begin{example}
\label{example:cp1-vb-moduli-space}
For rational curves, the bundle $\calO \oplus \calO$ is the unique
semistable rank 2 bundle with trivial determinant bundle, and there
are no stable rank 2 bundles, so
\begin{align}
  \nonumber
  M^{ss}(\CP^1) &= \{pt\} = \{[\mathcal O \oplus \calO]\}, &
  M^{s}(\CP^1) &= \varnothing.
\end{align}
\end{example}

\begin{example}
\label{example:elliptic-vb-moduli-space}
For an elliptic curve $X$, semistable rank 2 bundles with trivial
determinant bundle have the form
$L \oplus L^{-1}$, where $L$ is a degree 0 line bundle, or
$F_2 \otimes L_i$, where $L_i$ for $i=1,\cdots, 4$ are the four
2-torsion line bundles.
The bundles $L_i \oplus L_i$ and $F_2 \otimes L_i$ are $S$-equivalent.
The bundles $L \oplus L^{-1}$ and $L^{-1} \oplus L$ are
isomorphic, hence $S$-equivalent.
There are no stable rank 2 bundles with trivial determinant bundle.
As shown in \cite{Tu}, we have that
\begin{align}
  \nonumber
  M^{ss}(X) &= \{[L \oplus L^{-1}] \mid [L] \in \Jac(X)\} = \CP^1, &
  M^{s}(X) &= \varnothing.
\end{align}
\end{example}

\section{Parabolic bundles}
\label{sec:parabolic-bundles}

Here we briefly review some results on parabolic bundles and
their moduli spaces that we will use throughout the paper.
Some useful references on parabolic bundles are
\cite{Mehta-Seshadri, Nagaraj}.

\subsection{Definition of a parabolic bundle}

The concept of a parabolic bundle was introduced in
\cite{Mehta-Seshadri}:

\begin{definition}
A {\em parabolic bundle} of rank $r$ on a curve $C$ consists of
following data:
\begin{enumerate}
\item
  A rank $r$ holomorphic vector bundle $\pi_E:E \rightarrow C$.

\item
  Distinct marked points $(p_1, \cdots, p_n) \in C^n$.

\item
  For each marked point $p_i$, a flag of vector spaces $E_{p_i}^j$ in
  the fiber $E_{p_i} = \pi_E^{-1}(p_i)$ over the point $p_i$:
  \begin{align}
    \nonumber
    E_{p_i}^0 = 0 \subset E_{p_i}^1 \subset E_{p_i}^2 \subset \cdots
    \subset E_{p_i}^{s_i} = E_{p_i}.
  \end{align}
\item
  For each marked point $p_i$, a strictly decreasing list of
  \emph{weights} $\lambda_{p_i}^j \in \Reals$:
  \begin{align}
    \nonumber
    \lambda_{p_i}^1 > \lambda_{p_i}^2 > \cdots > \lambda_{p_i}^{s_i}.
  \end{align}
\end{enumerate}
We refer the data of the marked points, the flags, and the weights as
a {\em parabolic structure} on $E$.
We refer to the data of just the marked points and flags, without the
weights, as a \emph{quasi-parabolic structure} on $E$.
We define the \emph{multiplicity} of the weight
$\lambda_{p_i}^j$ to be
$m_{p_i}^j := \dim(E_{p_i}^j) - \dim(E_{p_i}^{j-1})$.
The definition of a parabolic bundle given in \cite{Mehta-Seshadri}
differs slightly from our definition, in that the marked points are
unordered and the weights are required to lie in the range $[0,1)$.
\end{definition}

\begin{definition}
Two parabolic bundles with underlying vector bundles $E$ and $F$ are
{\em isomorphic} if the marked points and weights for the two bundles
are the same, and there is a bundle isomorphism $\alpha:E \rightarrow
F$ that carries each flag of $E$ to the corresponding flag of $F$;
that is, $\alpha(E_{p_i}^j) = F_{p_i}^j$ for $j=1,\cdots, s_i$ and
$i=1, \cdots, n$.
\end{definition}

\begin{definition}
The {\em rank} of a parabolic bundle is the rank of its underlying
vector bundle.
\end{definition}

\begin{definition}
The {\em parabolic degree} and {\em parabolic slope} of a parabolic
bundle ${\mathcal E}$ with underlying vector bundle $E$ are defined to be
\begin{align}
  \nonumber
  \pdeg {\mathcal E} &=
  \deg E +
  \sum_{i=1}^n \sum_{j=1}^{s_i} m_{p_i}^j \lambda_{p_i}^j \in \Reals, &
  \pslope {\mathcal E} &=
  (\pdeg {\mathcal E})/(\rank {\mathcal E}) \in \Rats.
\end{align}
\end{definition}

We will not need the full generality of the concept of a parabolic
bundle; rather, we will consider only parabolic line bundles and rank
2 parabolic bundles of a certain restricted form.

First we consider parabolic line bundles.
For such bundles there is no flag data, so the parabolic structure is
specified by a list of marked points
$p_1, \cdots, p_n$ and a list of weights
$\lambda_{p_1}^1, \cdots, \lambda_{p_n}^1$.
We fix a parameter $\mu > 0$ and restrict to the case
$\lambda_{p_i}^1 \in \{\pm \mu\}$ for $i=1,\cdots,n$.
A parabolic line bundle of this form thus consists
of the data $(L, \sigma_{p_1}, \cdots, \sigma_{p_n})$, where
$\pi_L:L \rightarrow C$ is a holomorphic line bundle and
$\sigma_{p_i} \in \{\pm 1\}$.
The parabolic degree and parabolic slope of a parabolic line bundle
$(L, \sigma_{p_1}, \cdots, \sigma_{p_n})$ are given by
\begin{align}
  \nonumber
  \pdeg (L, \sigma_{p_1}, \cdots, \sigma_{p_n}) =
  \pslope (L, \sigma_{p_1}, \cdots, \sigma_{p_n}) =
  \deg L + \mu \sum_{i=1}^n \sigma_{p_i} =
  \slope L + \mu \sum_{i=1}^n \sigma_{p_i}.
\end{align}

Next we consider rank 2 parabolic bundles.
We fix a parameter $\mu > 0$ and restrict to the case
$s_i = 2$, $m_{p_i}^1 = m_{p_i}^2 = 1$, and
$\lambda_{p_i}^1 = -\lambda_{p_i}^2 = \mu$ for $i=1,\cdots,n$.
A rank 2 parabolic bundle of this form thus consists of the data
$(E,\ell_{p_1},\cdots,\ell_{p_n})$, where
$\pi_E:E \rightarrow C$ is a rank 2 holomorphic vector bundle and
$\ell_{p_i} \in \mathbbm{P}(E_{p_i})$ is a line in the fiber
$E_{p_i} = \pi_E^{-1}(p_i)$ over the point $p_i$ for $i=1,\cdots,n$.
The parabolic slope and parabolic degree of a rank 2 parabolic bundle
$(E,\ell_{p_1},\cdots,\ell_{p_n})$ are given by
\begin{align}
  \nonumber
  \pdeg (E,\ell_{p_1},\cdots,\ell_{p_n}) &= \deg E, &
  \pslope (E,\ell_{p_1},\cdots,\ell_{p_n}) &=
  \slope E.
\end{align}

\subsection{Stable, semistable, and unstable parabolic bundles}

Consider a rank 2 parabolic bundle $(E,\ell_{p_1},\cdots,\ell_{p_n})$
and a line subbundle $L \subset E$.
There are induced parabolic structures on the line bundles
$L$ and $E/L$ given by
$(L,\sigma_{p_1},\cdots,\sigma_{p_n})$ and
$(E/L,-\sigma_{p_1},\cdots,-\sigma_{p_n})$,
where
\begin{align}
  \nonumber
  \sigma_{p_i} =
  \left\{
  \begin{array}{ll}
    +1 &
    \quad \mbox{if $L_{p_i} = \ell_{p_i}$,} \\
    -1 &
    \quad \mbox{if $L_{p_i} \neq \ell_{p_i}$.} \\
  \end{array}
  \right.
\end{align}

\begin{definition}
Given a rank 2 parabolic bundle $(E,\ell_{p_1},\cdots,\ell_{p_n})$ and a
line subbundle $L \subset E$, we say that the induced parabolic bundle
$(L,\sigma_{p_1},\cdots,\sigma_{p_n})$ is a {\em parabolic subbundle}
of $(E,\ell_{p_1},\cdots,\ell_{p_n})$ and the induced parabolic bundle
$(E/L,-\sigma_{p_1},\cdots,-\sigma_{p_n})$ is a
{\em parabolic quotient bundle} of $(E,\ell_{p_1},\cdots,\ell_{p_n})$.
\end{definition}

\begin{definition}
A rank 2 parabolic bundle $(E,\ell_{p_1},\cdots,\ell_{p_n})$ is
said to be {\em decomposable} if there exists a decomposition
$E = L \oplus L'$ for line bundles $L$ and $L'$ such that
$\ell_{p_i} \in \{L_{p_i}, L_{p_i}'\}$ for $i=1,\cdots,n$.
For a rank 2 decomposable parabolic bundle
$(E,\ell_{p_1},\cdots,\ell_{p_n})$ we write
\begin{align}
  \nonumber
  (E,\ell_{p_1},\cdots,\ell_{p_n}) =
  (L,\sigma_{p_1},\cdots,\sigma_{p_n}) \oplus
  (L',\sigma_{p_1}',\cdots,\sigma_{p_n}'),
\end{align}
where $(L,\sigma_{p_1},\cdots,\sigma_{p_n})$ and
$(L',\sigma_{p_1}',\cdots,\sigma_{p_n}')$
are the induced parabolic structures on $L$ and $L'$.
\end{definition}

\begin{definition}
A rank 2 parabolic bundle is {\em stable} if its
parabolic slope is strictly greater than the parabolic slope of any of
its proper parabolic subbundles,
{\em semistable} if its parabolic slope is greater
than or equal than the parabolic slope of any of its proper
parabolic subbundles,
\emph{strictly semistable} if it is semistable but not stable,
and
{\em unstable} if it has a proper parabolic
subbundle of strictly greater slope.
\end{definition}

If $\calE$ is a stable parabolic bundle, then
$\Aut(\calE) = \Complex^\times$ consists only of trivial automorphisms
that scale the fibers of the underlying vector bundle by a constant
factor.

\begin{theorem}
\label{theorem:psemistable-semistable}
If the rank 2 parabolic bundle $(E, \ell_{p_1}, \cdots, \ell_{p_n})$
is semistable and $\mu < 1/2n$, then $E$ is semistable.
\end{theorem}

\begin{proof}
We will prove the contrapositive, so assume that $E$ is unstable.
Then there is a line subbundle $L \subset E$ such that
$\slope L > \slope E$.
Consider the parabolic structure
$(L,\sigma_{p_1},\cdots,\sigma_{p_n})$ induced on $L$ by
$(E, \ell_{p_1}, \cdots, \ell_{p_n})$.
We have that
\begin{align}
  \label{eqn:pslope-slope}
  \pslope(L,\sigma_{p_1},\cdots,\sigma_{p_n}) -
  \pslope(E, \ell_{p_1}, \cdots, \ell_{p_n}) =
  \slope L + \mu \sum_{i=1}^n \sigma_{p_i} - \slope E.
\end{align}
Since $\slope L$ is an integer, $\slope E$ is an integer or
half-integer, and
$\slope L > \slope E$, it follows that
$\slope L - \slope E \geq 1/2$.
From equation (\ref{eqn:pslope-slope}) and the assumption that
$\mu < 1/2n$, it follows that
\begin{align}
  \nonumber
  \pslope(L,\sigma_{p_1},\cdots,\sigma_{p_n}) -
  \pslope(E, \ell_{p_1}, \cdots, \ell_{p_n}) \geq
  1/2 - n\mu > 0,
\end{align}
so $(E, \ell_{p_1}, \cdots, \ell_{p_n})$ is unstable.
\end{proof}

Throughout this paper we will always assume $\mu \ll 1$, by which we
mean that $\mu$ is always chosen to be sufficiently small such that
Theorem \ref{theorem:psemistable-semistable} holds under whatever
circumstances we are considering.

Consider a rank 2 vector bundle $E$ over a curve $C$.
If $E$ is unstable, then
Theorem \ref{theorem:psemistable-semistable} implies that the
parabolic bundle
$(E,\ell_{p_1},\cdots,\ell_{p_n})$ is unstable.
If $E$ is a semistable, then the stability of the parabolic bundle
$(E,\ell_{p_1},\cdots,\ell_{p_n})$ can be characterized as follows:

\begin{theorem}
\label{theorem:parabolic-semistable}
Consider a rank 2 parabolic bundle of the form
$(E,\ell_{p_1},\cdots,\ell_{p_n})$ with $E$ semistable.
Let $m$ be the maximum number of lines that are bad in the
same direction.
Such a parabolic bundle is
stable if and only if $m < n/2$,
semistable if and only if $m \leq n/2$, and
unstable if and only if $m > n/2$.
In particular, if $n$ is odd then stability and semistability are
equivalent.
\end{theorem}

\begin{example}
\label{example:cp1-pstable}
As a special case of Theorem \ref{theorem:parabolic-semistable},
consider parabolic bundles $(E, \ell_{p_1}, \cdots, \ell_{p_n})$ over
$\CP^1$ with underlying vector bundle $E = \calO \oplus \calO$.
We can globally trivialize $E$ and identify all the fibers with
$\Complex^2$.
All lines of $E$ are bad, and lines
are bad in the same direction if and only if they are equal
under the global trivialization.
Let $m$ denote the maximum number of lines $\ell_{p_i}$ equal to any
given line in $\CP^1$.
From Theorem \ref{theorem:parabolic-semistable}, we find that
$(E, \ell_{p_1}, \cdots, \ell_{p_n})$ is stable if
$m < n/2$, semistable if $m \leq n/2$, and
unstable if $m > n/2$.
For example,
$(E, \ell_{p_1})$ is unstable,
$(E, \ell_{p_1}, \ell_{p_2})$ is strictly semistable if
the lines are distinct and unstable otherwise, and
$(E, \ell_{p_1}, \ell_{p_2}, \ell_{p_3})$ is stable if
the lines are distinct and unstable otherwise.
\end{example}

\subsection{$S$-equivalent semistable parabolic bundles}

There is Jordan-H\"{o}lder theorem for parabolic bundles, which
asserts that any semistable parabolic bundle of parabolic degree 0 has
a filtration in which quotients of successive parabolic bundles
(composition factors) in the filtration are stable with parabolic
slope 0 (see \cite[Remark 1.16]{Mehta-Seshadri}).
The filtration is not unique, but the composition factors are unique
up to permutation.
It follows that one can define an associated graded bundle of a
semistable parabolic bundle of parabolic degree 0 that is unique up to
isomorphism.

We will need the concept of an associated graded parabolic bundle only
for the case of semistable rank 2 parabolic bundles.
If such a parabolic bundle $\calE$ is stable, then its associated
graded parabolic bundle $\gr \calE$ is just $\calE$.
Now consider a strictly semistable parabolic bundle
$\calE = (E,\ell_{p_1},\cdots,\ell_{p_n})$.
Under our standard assumption that $\mu \ll 1$, it follows from
Theorem \ref{theorem:psemistable-semistable} that
$E$ is semistable.
The associated graded parabolic bundle $\gr \calE$ is given by
\begin{align}
  \nonumber
  \gr (E,\ell_{p_1},\cdots,\ell_{p_n}) =
  (L,\sigma_{p_1},\cdots,\sigma_{p_n}) \oplus
  (E/L,-\sigma_{p_1},\cdots,-\sigma_{p_n}),
\end{align}
where $L \subset E$ is a line subbundle such that
$\slope L = \slope E$, and
$(L,\sigma_{p_1},\cdots,\sigma_{p_n})$ and
$(E/L,-\sigma_{p_1},\cdots,-\sigma_{p_n})$ are the induced parabolic
structures on $L$ and $E/L$.
Note that
\begin{align}
\nonumber
\pslope(\gr(E,\ell_{p_1},\cdots,\ell_{p_n})) =
\pslope(E,\ell_{p_1},\cdots,\ell_{p_n}) = \slope E.
\end{align}

\begin{definition}
We say that two semistable rank 2 parabolic bundles are
{\em $S$-equivalent} if their associated graded bundles are
isomorphic.
\end{definition}

Isomorphic parabolic bundles are $S$-equivalent.
Here we give an example to show that the converse does not always
hold:

\begin{example}
\label{example:cp1-s-equiv}
Consider parabolic bundles over $\CP^1$ with underlying vector bundle
$E = \calO \oplus \calO$.
We can globally trivialize $E$ and identify all the fibers with
$\Complex^2$.
Let $A$, $B$, and $C$ be distinct lines in $\CP^1$, and consider the
two strictly semistable parabolic bundles
\begin{align}
  \nonumber
  \calE :=
  (E,\,\ell_{p_1} = A,\,\ell_{p_2} = A,\, \ell_{p_3} = B,\, \ell_{p_4} = C),
  &&
  \calE' :=
  (E,\,
  \ell_{p_1}' = B,\,\ell_{p_2}' = C,\, \ell_{p_3}' = A,\, \ell_{p_4}' = A).
\end{align}
The bundles $\calE$ and $\calE'$ are
not isomorphic but are $S$-equivalent, since the associated graded
bundles of both bundles are isomorphic to
\begin{align}
  \nonumber
  (\calO,\,\sigma_{p_1} = 1,\,\sigma_{p_2} = 1,\,\sigma_{p_3} = -1,\,
  \sigma_{p_4} = -1) \oplus
  (\calO,\,\sigma_{p_1} = -1,\,\sigma_{p_2} = -1,\,\sigma_{p_3} = 1,\,
  \sigma_{p_4} = 1).
\end{align}
\end{example}

\subsection{Moduli spaces of rank 2 parabolic bundles}

\begin{definition}
We define $M^{ss}(C,n)$ (respectively $M^s(C,n)$), to be the moduli
space of semistable (respectively stable) rank 2
parabolic bundles of the form $(E,\ell_{p_1},\cdots,\ell_{p_n})$ such
that $E$ has trivial determinant bundle, mod $S$-equivalence.
In particular, $M^{ss}(C,0) = M^{ss}(C)$ and $M^s(C,0) = M^s(C)$.
As always, we assume that $\mu \ll 1$.
This space is defined in \cite{Mehta-Seshadri}; see also
\cite{Bhosle}.
\end{definition}

\begin{remark}
\label{remark:parabolic-moduli-interp}
An alternative way of interpreting $M^{ss}(C,n)$ is as
the space of flat $SU(2)$-connections on a trivial rank 2 complex
vector bundle $E \rightarrow C - \{p_1, \cdots, p_n\}$, where
the holonomy around each puncture point $p_i$ is conjugate to
$\diag(e^{2\pi i\mu}, e^{-2\pi i\mu})$,
mod $SU(2)$ gauge transformations.
Yet another way of interpreting the space $M^{ss}(C,n)$ is as the
character variety $R(C,n)$ of conjugacy
classes of group homomorphisms
$\pi_1(C - \{p_1, \cdots, p_n\}) \rightarrow SU(2)$ that take
loops around the marked points to matrices conjugate to
$\diag(e^{2\pi i\mu}, e^{-2\pi i\mu})$.
Note that $\mu = 1/4$ corresponds to the traceless character variety.
We will not use these interpretations here.
\end{remark}

For rational $\mu$, the moduli space $M^s(C,n)$ has the structure of a
complex manifold of dimension $3(g - 1) + n$, where $g$ is the genus of
the curve $C$, and $M^s(C,n)$ is compact for $n$ odd.
The space $M^s(C,n)$ carries a canonical symplectic
form, which is obtained by viewing $M^s(C,n)$ as a Hamiltonian
reduction of a space of $SU(2)$-connections with prescribed
singularities.

\begin{example}
\label{example:cp1-parabolic-moduli-space}
Let $C = \CP^1$ be a rational curve.
If we fix $n \leq 3$, then all rank 2 parabolic bundles of the form
$(\calO \oplus \calO,\ell_{p_1},\cdots,\ell_{p_n})$ for which all the
lines are distinct are isomorphic.
From this fact, together with the results of Example
\ref{example:cp1-pstable}, we find that
\begin{align}
  \nonumber
  M^{ss}(\CP^1,0) &= \{pt\}, &
  M^{ss}(\CP^1,1) &= \varnothing, &
  M^{ss}(\CP^1,2) &= \{pt\}, &
  M^{ss}(\CP^1,3) &= \{pt\}, \\
  \nonumber
  M^s(\CP^1,0) &= \varnothing, &
  M^s(\CP^1,1) &= \varnothing, &
  M^s(\CP^1,2) &= \varnothing, &
  M^s(\CP^1,3) &= \{pt\}.
\end{align}
Using the cross-ratio and considerations of $S$-equivalence as
described in Example \ref{example:cp1-s-equiv}, one can show
\begin{align}
  \nonumber
  M^{ss}(\CP^1,4) &= \CP^1, &
  M^{s}(\CP^1,4) &= \CP^1 - \{\textup{3 points}\}.
\end{align}
\end{example}

\begin{example}
Let $X$ be an elliptic curve.
From
Corollary \ref{corollary:map-bundle-elliptic},
Example \ref{example:elliptic-vb-moduli-space}, and
Theorem \ref{theorem:parabolic-semistable},
we have that
\begin{align}
  \nonumber
  M^{ss}(X,0) &= \CP^1, &
  M^s(X,0) &= \varnothing, &
  M^{ss}(X,1) &= \CP^1, &
  M^s(X,1) &= \CP^1.
\end{align}
In \cite{Vargas} it is shown that
\begin{align}
  \nonumber
  M^{ss}(X,2) &= (\CP^1)^2, &
  M^s(X,2) &= (\CP^1)^2 - g(X),
\end{align}
where $g:X \rightarrow (\CP^1)^2$ is a holomorphic embedding.
We also derive this result in Section \ref{sec:pm-x-1-1}.
\end{example}

\begin{remark}
Throughout this paper we assume $\mu \ll 1$, but for some
applications one wants to take $\mu = 1/4$ in order to interpret
$M^{ss}(C,n)$ as a traceless character variety, as described in
Remark \ref{remark:parabolic-moduli-interp}.
In general $M^{ss}(C,n)$ depends on $\mu$; for example, for
$0 \leq n \leq 4$ the space $M^{ss}(\CP^1,n)$ is the same for
$\mu \ll 1/4$ and $\mu = 1/4$, but, as shown in \cite{Seidel},
for $n = 5$ we have
\begin{align}
  \nonumber
  M^{ss}(\CP^1,5) = \left\{
  \begin{array}{ll}
    \CP^2 \# 4\overline{\CP}^2 &
    \quad \mbox{for $\mu \ll 1$,} \\
    \CP^2 \# 5\overline{\CP}^2 &
    \quad \mbox{for $\mu = 1/4$.} \\
  \end{array}
  \right.
\end{align}
The dependence of $M^{ss}(C,n)$ on $\mu$ is discussed
in \cite{Boden}.
\end{remark}

\end{appendix}

\section*{Acknowledgments}

The author would like to express his gratitude towards Ciprian
Manolescu for suggesting the problem considered in this paper and
for providing invaluable guidance, Joel Kamnitzer for helpful
discussions, Burt Totaro for helpful advice and for offering extensive
comments on an earlier version of this paper that led to significant
changes in the current version, and Chris Woodward for sharing his
personal notes on the Woodward embedding described in Section
\ref{sec:woodward-embedding}.
The author was partially supported by NSF grant number
DMS-1708320.

\bibliographystyle{abbrv}
\bibliography{v2-hecke-modifications}

\begin{thebibliography}{10}

\bibitem{Abouzaid-Smith}
M.~Abouzaid and I.~Smith.
\newblock {Khovanov} homology from {Floer} cohomology.
\newblock {\em J. Amer. Math. Soc.}, 32:1--79, 2019.

\bibitem{Atiyah}
M.~F. Atiyah.
\newblock Vector bundles over an elliptic curve.
\newblock {\em Proc. London Math. Soc.}, 7:414--452, 1957.

\bibitem{Bakalov}
B.~Bakalov and A.~Kirillov.
\newblock {\em Lectures on tensor categories and modular functors}, volume~21
  of {\em University Lecture Series}.
\newblock American Mathematical Society, 2001.

\bibitem{Bhosle}
U.~N. Bhosle.
\newblock Parabolic vector bundles on curves.
\newblock {\em Ark. Math.}, 27(1--2):15--22, 1989.

\bibitem{Biswas}
I.~Biswas, L.~Brambila-Paz, and P.~E. Newstead.
\newblock Stability of projective {Poincar\'{e}} and {Picard} bundles.
\newblock {\em Bull. London Math. Soc.}, 41:458--472, 2009.

\bibitem{Boden}
H.~U. Boden and Y.~Hu.
\newblock Variations of moduli of parabolic bundles.
\newblock {\em Math. Ann.}, 301:539--559, 1995.

\bibitem{Grothendieck}
A.~Grothendieck.
\newblock Sur la classification des fibres holomorphes sur la sphere de
  {Riemann}.
\newblock {\em Am. J. Math.}, 79:121--138, 1957.

\bibitem{Heinloth}
J.~Heinloth.
\newblock Lectures on the moduli stack of vector bundles on a curve.
\newblock In {\em Affine Flag Manifolds and Principal Bundles}, Trends in
  Mathematics, pages 123--153. Springer, 2010.

\bibitem{Hoffmann}
N.~Hoffmann.
\newblock Rationality and {Poincar\'{e}} families for vector bundles with extra
  structure on a curve.
\newblock {\em Int. Math. Res. Not.}, article ID rnm010, 2007.

\bibitem{Horton}
H.~T. Horton.
\newblock A symplectic instanton homology via traceless character varieties.
\newblock {\em arXiv preprint arXiv:1611.09927v2}, 2016.

\bibitem{Iena}
O.~Iena.
\newblock Vector bundles on elliptic curves and factors of automorphy.
\newblock {\em Rend. Istit. Mat. Univ. Trieste}, 43:61--94, 2011.

\bibitem{Jones}
V.~Jones.
\newblock A polynomial invariant for knots via von {Neumann} algebras.
\newblock {\em Bull. Amer. Math. Soc.}, 12:103--111, 1985.

\bibitem{Kamnitzer}
J.~Kamnitzer.
\newblock {Grassmannians}, moduli spaces and vector bundles.
\newblock In {\em The Beilinson-Drinfeld Grassmannian and symplectic knot
  homology}, volume~14, pages 81--94. Clay Math. Proc., 2011.

\bibitem{Kapustin-Witten}
A.~Kapustin and E.~Witten.
\newblock Electric-magnetic duality and the geometric {Langlands} program.
\newblock {\em Commun. Number Theory Phys.}, 1(1):1--236, 2007.

\bibitem{Khovanov}
M.~Khovanov.
\newblock A categorification of the {Jones} polynomial.
\newblock {\em Duke Math. J.}, 101(3):359--426, 2000.

\bibitem{Kronheimer-2}
P.~Kronheimer and T.~Mrowka.
\newblock {Khovanov} homology is an unknot-detector.
\newblock {\em Publ. Math. Inst. Hautes \'{E}tudes Sci.}, 113:97--208, 2011.

\bibitem{Kronheimer-3}
P.~Kronheimer and T.~Mrowka.
\newblock Filtrations on instanton homology.
\newblock {\em Quantum Topol.}, 5:61--97, 2014.

\bibitem{LePotier}
J.~Le~Potier.
\newblock {\em Lectures on Vector Bundles}.
\newblock Cambridge University Press, 1997.

\bibitem{Mehta-Seshadri}
V.~B. Mehta and C.~S. Seshadri.
\newblock Moduli of vector bundles on curves with parabolic structures.
\newblock {\em Math. Ann.}, 248:205--239, 1980.

\bibitem{Mukai}
S.~Mukai and W.~M. Oxbury.
\newblock {\em An Introduction to Invariants and Moduli}.
\newblock Cambridge University Press, 2012.

\bibitem{Nagaraj}
D.~S. Nagaraj.
\newblock Parabolic bundles and parabolic {Higgs} bundles.
\newblock {\em arXiv preprint arXiv:1702.05295}, 2017.

\bibitem{Schaffhauser}
F.~Schaffhauser.
\newblock Differential geometry of holomorphic vector bundles on a curve.
\newblock {\em arXiv preprint arXiv:1509.01734}, 2015.

\bibitem{Seidel}
P.~Seidel.
\newblock Lectures on four-dimensional {Dehn} twists.
\newblock {\em arXiv preprint arXiv:math/0309012}, 2003.

\bibitem{Seidel-Smith}
P.~Seidel and I.~Smith.
\newblock A link invariant from the symplectic geometry of nilpotent slices.
\newblock {\em Duke Math. J.}, 134(3):453--514, 2006.

\bibitem{Seshadri}
C.~S. Seshadri.
\newblock Space of unitary vector bundles on a compact {Riemann} surface.
\newblock {\em Ann. of Math.}, 85(2):303--336, 1967.

\bibitem{Teixidor}
M.~Teixidor~i Bigas.
\newblock Vector bundles on curves.
\newblock \url{http://emerald.tufts.edu/~mteixido/files/vectbund.pdf}.
\newblock Accessed 2018-05-23.

\bibitem{Tu}
L.~W. Tu.
\newblock Semistable bundles over an elliptic curve.
\newblock {\em Adv. Math.}, 98:1--26, 1993.

\bibitem{Vargas}
N.~F. Vargas.
\newblock Geometry of the moduli of parabolic bundles on elliptic curves.
\newblock {\em Trans. Am. Math., to appear; arXiv preprint arXiv:1611.05417},
  2016.

\bibitem{Witten-2}
E.~Witten.
\newblock {Khovanov} homology and gauge theory.
\newblock {\em arXiv preprint arXiv:1108.3103}, 2011.

\bibitem{Witten}
E.~Witten.
\newblock Fivebranes and knots.
\newblock {\em Quantum Topol.}, 3(1):1--137, 2012.

\bibitem{Witten-3}
E.~Witten.
\newblock Two lectures on gauge theory and {Khovanov} homology.
\newblock {\em arXiv preprint arXiv:1603.03854}, 2016.

\bibitem{Woodward}
C.~Woodward.
\newblock Unpublished notes.

\end{thebibliography}

\end{document}